\documentclass[11pt, leqno]{article}

\usepackage{times,amsfonts,amsmath,amssymb,latexsym,setspace,epsfig,amsthm}
\usepackage{subfigure, graphics}
\usepackage{natbib}
\usepackage{todonotes}
\usepackage{url} 
\usepackage[colorlinks,linkcolor=black, citecolor=black, urlcolor=black]{hyperref}
\usepackage{multirow}
\usepackage{booktabs}
\usepackage{pdflscape,rotating}
\usepackage{bbm}
\usepackage[linesnumbered,lined,ruled]{algorithm2e}
\usepackage{siunitx}
\usepackage[title]{appendix}

\allowdisplaybreaks

\newcommand{\DK}{Davies \& Kovac }
\newcommand{\abs}[1]{\left| #1 \right|}
\newcommand{\code}[1]{\texttt{#1}}

\def\bs{\boldsymbol}
\def\Ex{{\rm I\!E}}
\def\Pr{{\rm I\!P}}
\def\be{\begin{equation}}
\def\ee{\end{equation}}
\def\bea{\begin{eqnarray*}}
\def\eea{\end{eqnarray*}}
\def\bean{\begin{eqnarray}}
\def\eean{\end{eqnarray}}
\def\nn{\nonumber}

\def\ra{\rightarrow}

\def\Bl{\Bigl}
\def\Br{\Bigr}
\def\alp{\alpha}

\def\eps{\epsilon}
\def\lam{\lambda}

\def\R{{\bs{R}}}
\def\AA{{\mathcal A}}
\def\BB{{\mathcal B}}
\def\II{{\mathcal I}}
\def\JJ{{\mathcal J}}

\def\lLR{\mathrm{logLR}_n}
\def\nbin{N_{\tt bin}}
\def\pen{\ell}
\def\N{\mathbb{N}}
\def\ind{\mathbbm{1}}

\newtheorem{theorem}{Theorem}
\newtheorem{proposition}{Proposition}

\newtheorem{lemma}{Lemma}

\theoremstyle{definition}
\newtheorem{remark}{Remark}

\begin{document}
\setstretch{1.25}

\title{The Essential Histogram}
\author{Housen Li${}^{1,*}$, Axel Munk${}^{1}$, Hannes Sieling${}^{1}$, Guenther Walther${}^{2}$ \\
\vspace{0.1cm}\\
{\small${}^{1}$Institute for Mathematical Stochastics, University of G\"{o}ttingen}\\
{\small${}^{2}$Department of Statistics, Stanford University}}

\date{June 2018}

\maketitle

\begin{abstract}
The histogram is widely used as a simple, exploratory display of data, but it is usually not clear how to choose the number and size of bins. We construct a confidence set of distribution functions that optimally address the two main tasks of the histogram: estimating probabilities and detecting features such as increases and modes in the distribution. We define the {essential histogram} as the histogram in the confidence set with the fewest bins. Thus the essential histogram is the simplest visualization of the data  that optimally achieves the main tasks of the histogram. The only assumption we make is that the data are independent and identically distributed.  We provide a fast algorithm for the essential histogram, and illustrate our methodology with examples.  An R-package is available on CRAN.
\end{abstract}

\noindent\textbf{Keywords.} Histogram; Mode detection; Multiscale testing; Optimal estimation;  Significant feature.

\section{Introduction}

The histogram, introduced by Karl Pearson in 1895, is one of the most basic but still one of the most widely used tools to visualize data. However, the construction of the histogram is not unique, leaving the user considerable freedom to choose the locations and number of breakpoints, see~\cite{FrPiPu07}. This arbitrariness allows for radically different visual representations of the data, and it appears that no satisfactory rule for the construction is known, as evidenced by the large number of rules proposed in the literature. In the case of equal bin widths, popular examples of rules for the number of bins are those given by \cite{Stu26}, which is still the default rule in {\tt R}, \cite{Sco79}, \cite{FreDia81}, \cite{Tay87}, and \cite{BirRoz06}. Most of them are derived by viewing the histogram as an estimator of a density and choosing the number of bins to minimize an asymptotic risk estimate. This leads to questions about the performance for small samples as well as about smoothness assumptions that are not verifiable. Instead of having all bins equally wide, it is also common to give equal area to all blocks. \cite{DenMal09} point out that the first approach typically leads to oversmoothing in regions of high density and is poor at identifying sharp peaks, whereas the second oversmooths in regions of low density and does not identify small outlying groups of data. They advocate a compromise of these two approaches that is motivated by regarding the histogram as an exploratory tool to identify structure in the data such as gaps and spikes, rather than as a density estimator, and they argue that relying on asymptotic risk minimization may lead to inappropriate recommendations for the number of bins. This is in line with recent findings for the regressogram~\citep{Tuk61}, the regression counterpart for the histogram~\citep{FrMuSi14,LiMuSi16}. Here the bin choice corresponds to finding {locations} of constant segments, which is a different target than conventional risk minimization, e.g. of the $L^p$ norm, $p \geq 1$. 

This paper proposes a rule for constructing a histogram that is motivated by the two main goals of the histogram \citep[see][]{FrPiPu07}: The histogram provides estimates of probabilities via relative areas; and it provides a display of the density that is simple but informative, i.e. it aims to {have few bins}, but still shows the important features of the data, such as modes.

The idea of the paper is to construct a confidence set of distribution functions such that each one in the confidence set satisfies the first goal~in an asymptotically optimal way. To meet the second goal, we select the simplest distribution function in the confidence set, i.e., the one with the fewest bins, as our histogram distribution function. The resulting histogram is the simplest one that shows important features of the data, such as increases or modes. We call this the essential histogram. Our approach is motivated by the fact that simplicity is a key aspect of the histogram: not only is it implicit in its goal to serve as an exploratory tool, but also in its definition as a piecewise constant function, which should capture the major features of data and the underlying distribution well. We show that in a large sample setting, each distribution function in the confidence set estimates probabilities of intervals with a standardized simultaneous estimation error that is at most twice what is achievable and which is typically much smaller than those obtained from histograms with traditional rules. Likewise, we show that the distribution functions are asymptotically optimal for detecting important features, such as increases or modes of the distribution. Therefore, we attain the above two goals of the histogram asymptotically, but one of the main benefits of our construction is that it provides finite sample guaranteed confidence statements about features of the data: large increases of any histogram in the confidence set, and hence of the essential histogram, indicate significant increases in the true density (cf.~Theorem~\ref{thmFeatureInfer}). We illustrate this by an example in Fig.~\ref{fig:empIntro}. Our finite sample guarantee ensures that the true density has an increase on the two dot-dash intervals, and has a decrease on the two dotted ones, with simultaneous confidence at least 90\%. It implies that the true density has two modes and one trough, as the plotted intervals are disjoint~\citep[cf.][]{DueWal08}. These intervals are a selection of a much larger set of intervals of increase and decrease at all scales, which the method offers (see \S\ref{computation} and \S\ref{optimalfeatures}). Thus, we can state with 90\% guaranteed finite sample confidence that these modes or troughs are really there in the underlying population. These confidence statements are quite valuable enhancements to the essential histogram as an exploratory tool. Any other histogram can be accompanied with our method to obtain such statements for it in order to justify or question modes it suggests {(see \S\ref{ss:evaTool}).} {Indeed, the only  parameter of the essential histogram is the significance level $\alpha$. One should set $\alpha = 0.1$ or even smaller if   confidence statements have to be made, while one can set $\alpha = 0.9$ if the goal is to explore the data for potential features with tolerance to false positives. As a trade-off, we recommend  $\alpha = 0.5$ as default.}

\begin{figure}
\centering
\begin{tabular}{ll}
{\small (a)} & {\small (b)} \\
\includegraphics[width=0.45\textwidth]{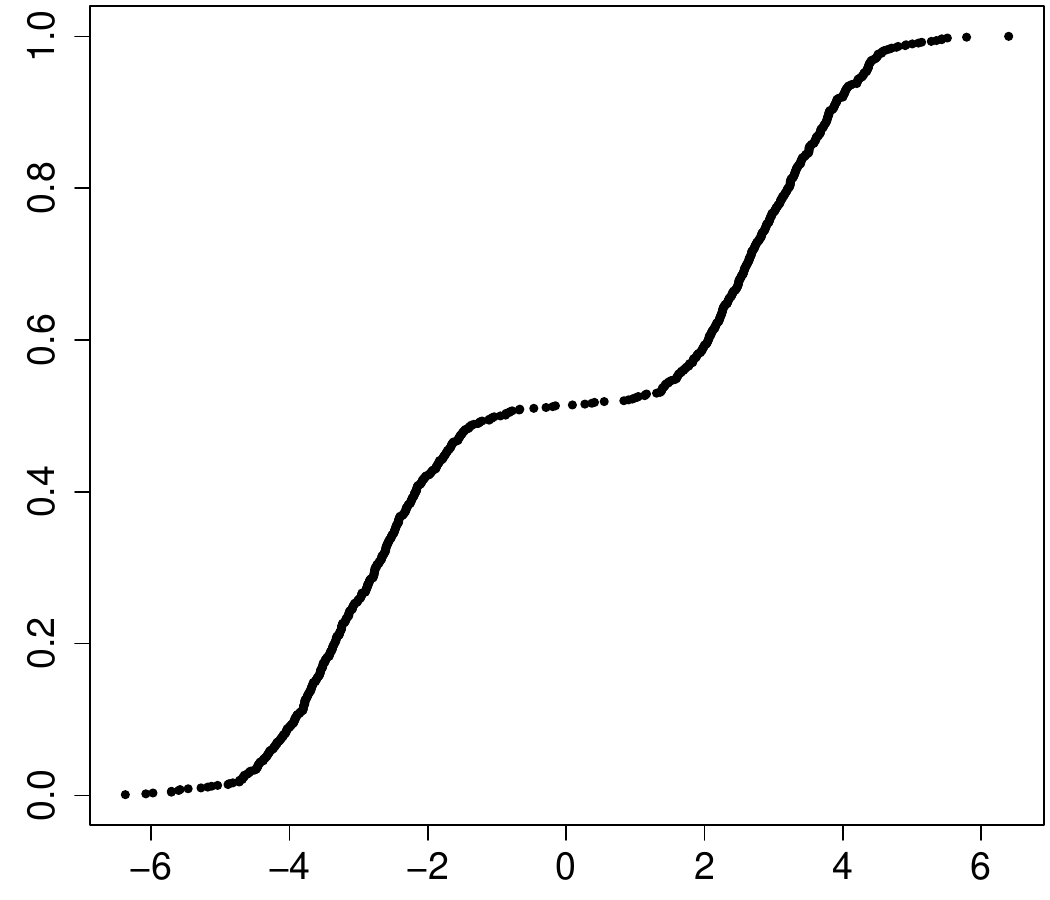} &
\includegraphics[width=0.45\textwidth]{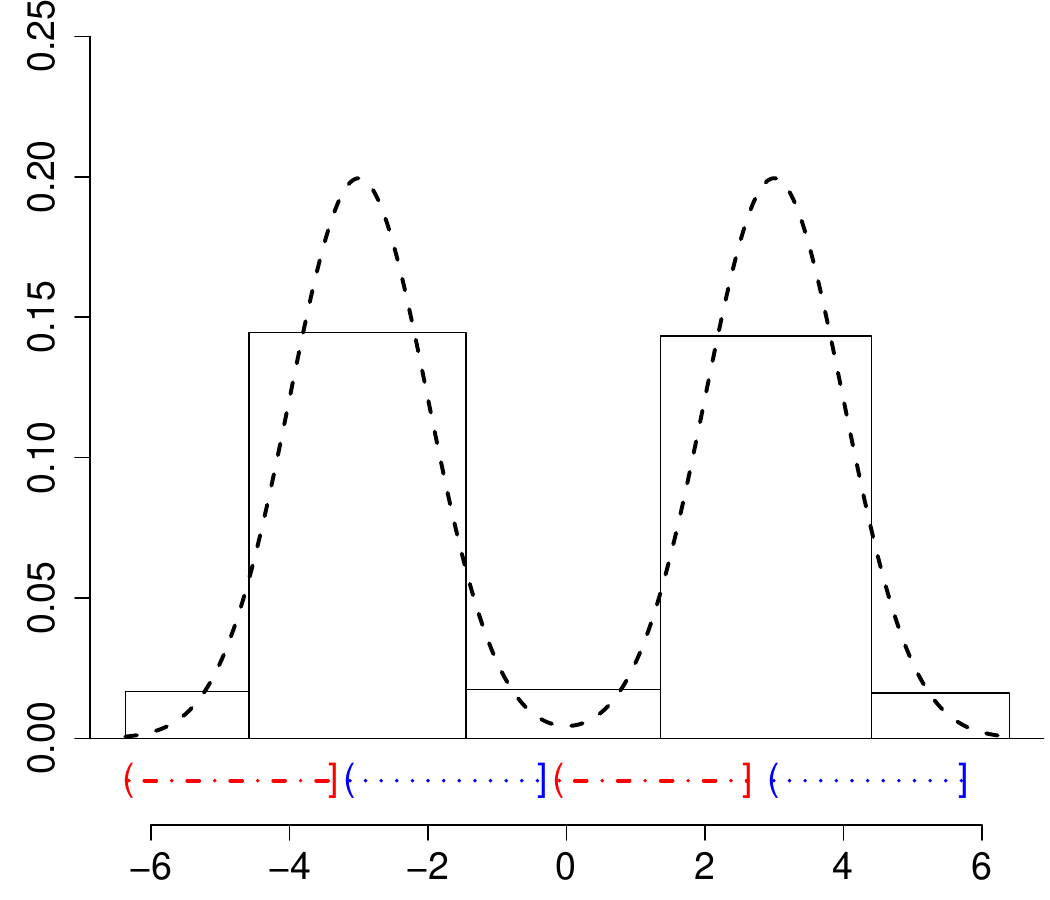} 
\end{tabular}
\caption{Illustration of the essential histogram. (a) the empirical distribution of {900} observations from the Gaussian mixture $0.5\mathcal{N}(-3,1)+0.5\mathcal{N}(3,1)$; (b) the essential histogram with $\alpha = 0.1$ (solid) and the true density (dashed); in the lower part, intervals indicating regions that contain a point of increase (dot-dash) or decrease (dotted). }
\label{fig:empIntro}
\end{figure}

The essential histogram is fairly general, since we make no assumption on $F$. In particular, it also applies to distributions with discrete components, a common feature in real datasets \citep[see e.g.][]{Unw15}. Figure~\ref{fig:discrete} gives two illustration examples: one is the example in \citet[Fig.~4]{DenMal09}, a mixture of three distributions $0.775\,\mathcal{N}(0,1) + 0.15\, \delta_7 + 0.075\,\mathcal{U}(0,10)$; the other is the duration times in the geyser dataset \citep{AzBo90}. As is shown, the essential histogram  estimates the probability of both continuous and discrete components over all scales rather well, and reveals the true shape of the underlying distribution functions. 

\begin{figure}[!t]
\centering
\begin{tabular}{lll}
{\small (a) }& {\small (b)} & {\small (c)} \\
\includegraphics[width=0.3\textwidth]{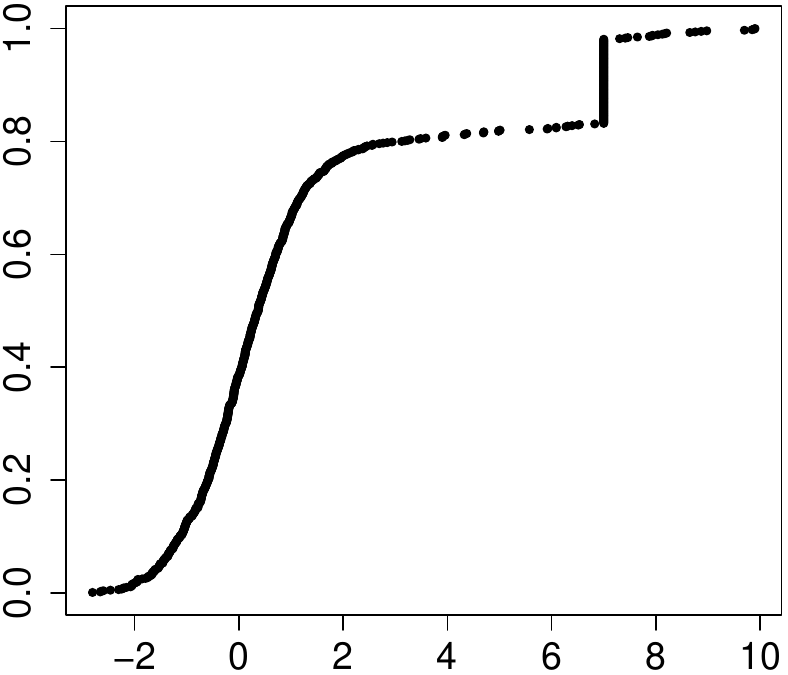} &
\includegraphics[width=0.3\textwidth]{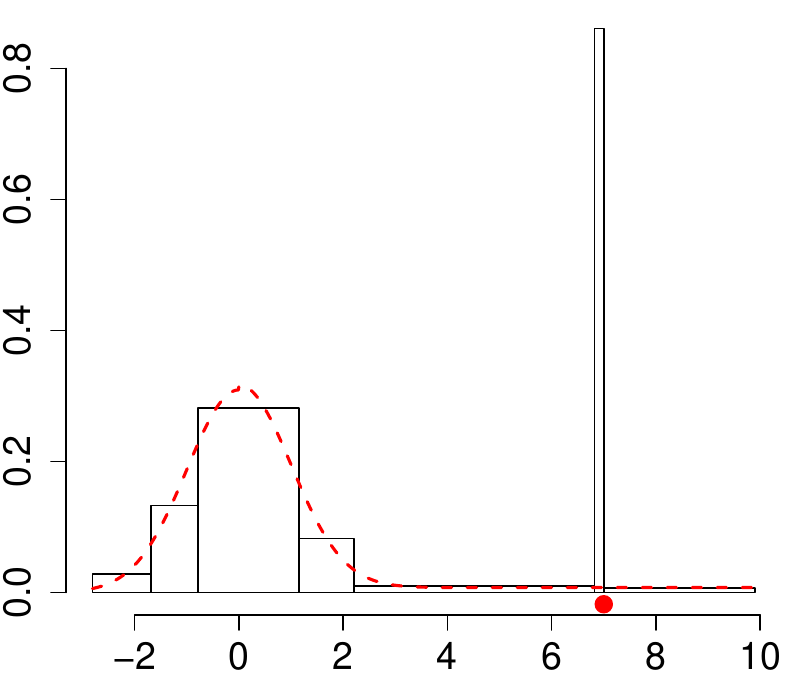} & 
\includegraphics[width=0.3\textwidth]{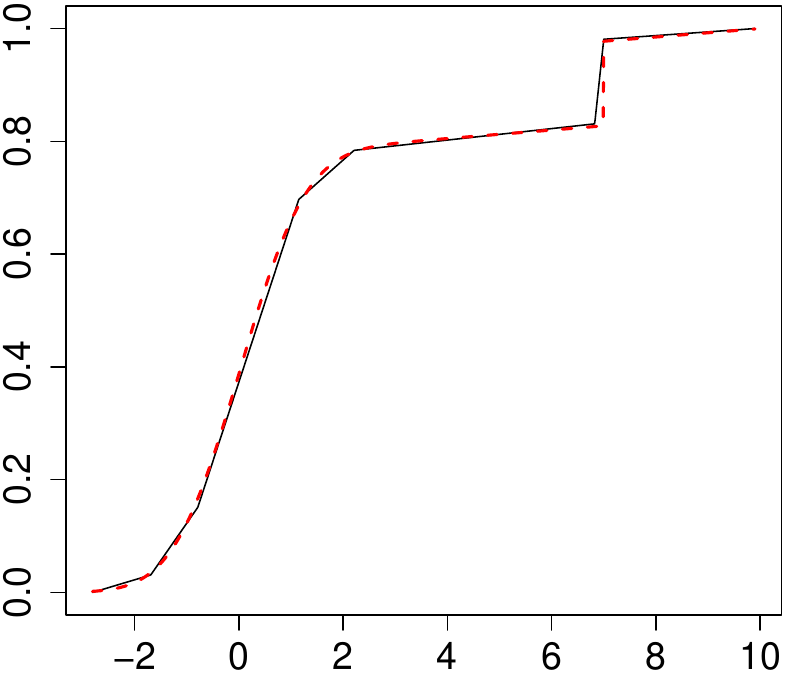}
\end{tabular}
\begin{tabular}{ll}
{\small (d) }& {\small (e)} \\
\includegraphics[width=0.3\textwidth]{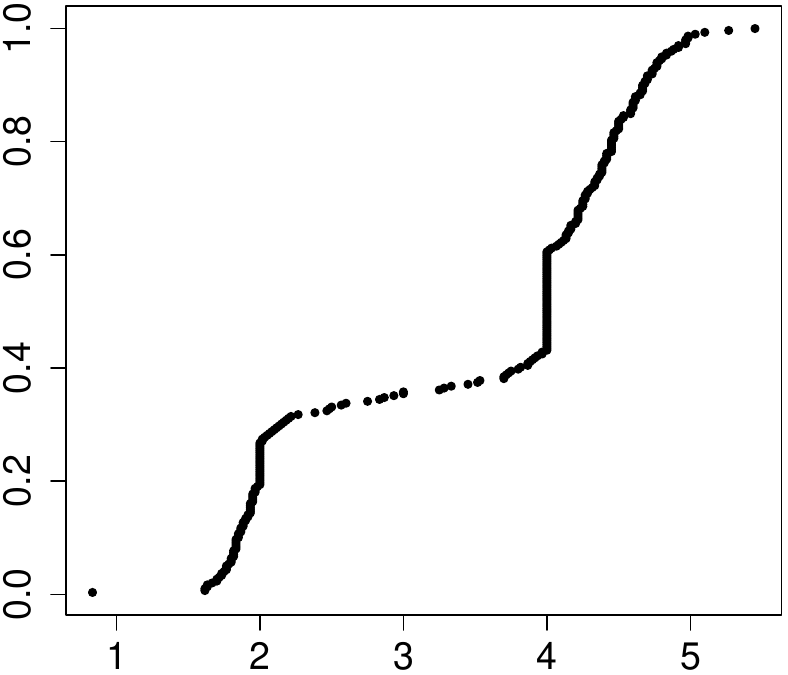} &
\includegraphics[width=0.3\textwidth]{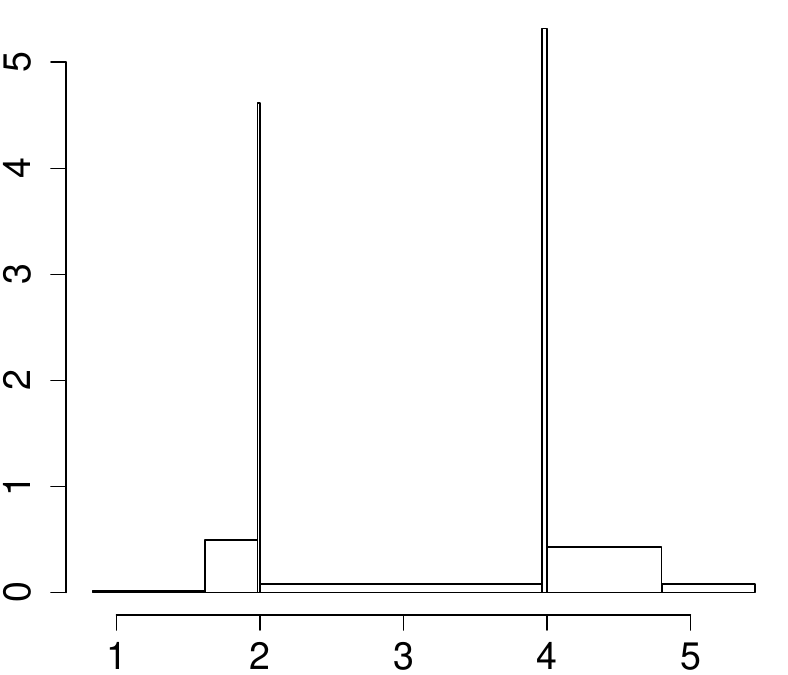} 
\end{tabular}
\caption{Illustration on discontinuous distribution functions. Denby \& Mallows example (sample size $n=1,000$): 
(a) empirical distribution function, (b) the essential histogram (default $\alpha = 0.5$) and the true density (dashed, with a dot indicating the point mass $\delta_7$), and (c) their distribution functions (dashed for the truth).  Geyser data: (d) empirical distribution function of duration times (mins), (e) the essential histogram (default $\alpha = 0.5$).}
\label{fig:discrete}
\end{figure}

The construction of the confidence set is based on the multiscale likelihood ratio test introduced by \cite{RivWal13}, and we show here that this test results in the optimal detection of certain features in the data. \cite{FrMuSi14} use such a multiscale likelihood ratio test for inference on change-points in a regression setting and they {employ} the idea of selecting the function in the confidence set that has the fewest jumps. In the context of the histogram, this approach produces breakpoints only at locations where the evidence in the data requires those in order to show significant features and to provide good probability estimates. Hence the methodology will not put any breakpoints in regions where the density is close to flat. This built-in parsimony is what one would expect from an automatic method for constructing a histogram, see also the comments about open research problems in \cite{DenMal09}. The taut string method of \cite{DavKov04} can be interpreted as producing a histogram {(although not satisfying the first goal of histogram)} that has the smallest number of modes within a confidence ball given by the (periodic) Kolmogorov metric. It is known that the Kolmogorov metric will not result in good probability estimates for intervals unless they have large probability content \citep{DueWel14}. This procedure does not aim at parsimony of bins and will typically produce many more bins than the essential histogram {(although {often} providing visually appealing solutions, and estimating the number of modes well, see \S\ref{examples})}, while the essential histogram automatically results in parsimony of bins and thus also of modes as explained above.

\section{A confidence set for the distribution function} \label{confset}
The empirical distribution function $F_n$ of $n$ independent and identically distributed univariate observations $X_1,\ldots,X_n$ is in a certain sense an optimal estimator of the underlying distribution function $F$,  see \citet{DvoKieWolf56}. While it is straightforward to convert $F_n$ into a histogram distribution function, see \citet[p.86]{ShoWel86}, the resulting histogram with $n$ breakpoints at the observations will generally not be useful for the data visualization of the data as it is much too rough. The premise of this paper is that it is typically possible to remove a large fraction of these breakpoints and still have an estimator that is just as good as $F_n$ for estimating probabilities {$F(I) = \int_I dF$} of arbitrary intervals $I$. This is clearly plausible for local stretches where $F$ has a density that is flat, but it will be seen that also for more general $F$ it is typically possible to reduce the number of breakpoints considerably without incurring a significant error in estimating $F(I)$ or loss of power for detecting important features of $F$. This motivates our proposal for constructing a histogram by choosing the histogram distribution function with the fewest breakpoints that is still optimal for the latter tasks. As the resulting histogram will typically be parsimonious, this construction achieves the goal of providing a simple visualization of the data that optimally addresses the inferential and exploratory tasks of histograms.

The first step in this construction consists of deriving  a confidence set of distribution functions that have the same performance as $F_n$ for estimating probabilities $F(I)$. The idea is to apply certain likelihood ratio tests on a judiciously chosen set of intervals and then to invert this family of tests, i.e., to define a $(1-\alp)$-confidence region for $F$ as those 
distribution functions that pass the totality of these tests:
\begin{multline}\label{eq:Cn}
C_n(\alp)\ =\ \Bl\{ \mbox{distribution function } H: {\Bl(2 \lLR \bigl(H(I),F_n(I)\bigr)\Br)^{1/2}}  \\ \leq
\pen(F_n(I))+ \kappa_n(\alp)\ \ \mbox{for all }
I \in \JJ \Br\}.
\end{multline}
Here 
\begin{equation*}
\lLR \bigl(H(I),F_n(I)\bigr)= n F_n(I) \log\Bl( \frac{F_n(I)}{H(I)}\Br)
+n(1-F_n(I)) \log\Bl( \frac{1 - F_n(I)}{1 - H(I)}\Br)
\end{equation*} 
is the log-likelihood ratio statistic for testing $F(I)=H(I)$, 
\begin{equation}\label{eq:pen}
\pen(F_n(I)) = \Bl({2\log \bigl(\frac{e}{F_n(I) (1-F_n(I))}\bigr)}\Br)^{1/2}
\end{equation} is the scale penalty, and $\kappa_n(\alp)$ is the $(1-\alp)$-quantile of the distribution of
\begin{equation}\label{eq:msStat}
T_n\ =\ \max_{I \in \JJ} \Bl\{\Bl(2 \lLR \bigl(F(I),F_n(I)\bigr)\Br)^{1/2} -
\pen(F_n(I))\Br\}
\end{equation}
with $\JJ$ being a collection of intervals:
\begin{equation}\label{eqJJ}
\begin{aligned}
\JJ &= \bigcup_{l =2}^{l_{max}} \JJ(l),\;\;\; \mbox{ where }
\;l_{max}=\Bl\lfloor \log_2 \frac{n}{\log n}\Br\rfloor \;\;
\mbox{ and}\\
\JJ(l) &= \Bl\{(X_{(j)},X_{(k)}]:\,
j,k \in \{1+i d_{l},
i\in \N_0 \} {\cap\mathcal{D}} \ \mbox{, } m_{l}<k-j\leq 2m_{l}\Br\},\\
&\;\;\;\;\mbox{ where } m_{l}=n2^{-l},\; d_{l}=
\Bl\lceil \frac{m_{l}}{6 {l}^{1/2}}\Br\rceil, {\text{ and }\mathcal{D} = \bigl\{i\, :\, X_{(i)} \neq X_{(i+1)}\bigr\}.} 
\end{aligned}
\end{equation}
This collection of intervals was introduced in \cite{Wal10} to approximate the collection of all intervals on the line in a computationally efficient manner: \cite{RivWal13} show that the above multiscale likelihood ratio statistic can be computed in $O(n \log n)$ steps while the collection is still rich enough to guarantee optimal detection in certain scanning problems. Here we  show in \S\ref{optimalprobs} that every $H \in C_n(\alp)$ has the same asymptotic estimation error as $F_n$ for probabilities $F(I)$. Moreover, we show in \S\ref{optimalfeatures} that every $H \in C_n(\alp)$ is optimal for the detection of  certain features which are relevant for the exploratory purpose of the histogram. In particular, these optimality properties hold for the parsimonious histogram distribution function that we compute in \S\ref{computation} in the second step of our construction of the essential histogram.

\section{Computing the essential histogram}  \label{computation}

\subsection{Computationally feasible relaxation}{
For a given partition of the real line into intervals  $I_0,I_1, \ldots, I_K $, we define the {histogram of $F$} as the density $h(x)= \sum_{i=0}^K F(I_i) \ind_{I_i}(x) /\abs{I_i}$, where $\abs{I_i}$ is the Lebesgue measure of $I_i$. The histogram $h$ can be recovered from its distribution function $H$ as the left-hand derivative of $H$.} In the second step of our construction we will find a histogram in $C_n(\alpha)$ in \eqref{eq:Cn} with the least number of bins. This computation requires the solution of a nonconvex combinatorial optimization problem and is practically infeasible for most real world applications. However, it is possible to compute the exact solution of a slight relaxation (still nonconvex) of the original optimization problem in almost linear run time, see \S\ref{subAlg} and \S\ref{sp:cmp}. This optimization problem~is
\begin{equation}\label{eq:adj_ess_hist}
 \min  \quad \nbin (H)  \qquad \text{ subject to }\quad H \in \tilde{C}_{n}(\alp).
\end{equation}
Here $\tilde{C}_n(\alp)$ is the superset of the histogram distribution functions in $C_n(\alp)$ that results if one evaluates the likelihood ratio tests only on those intervals where the candidate density is constant:
\begin{multline}\label{eq:tldCn}
\tilde{C}_n(\alpha) = \Bl\{\mbox{histogram distribution function } H: 
\Bl(2 \lLR \bigl(H(I),F_n(I)\bigr)\Br)^{1/2} \leq \pen(F_n(I)) \\ +\kappa_n(\alpha)
\mbox{ for each } I\in \JJ \mbox{ where the left-hand derivative } H^{\prime} \mbox{ is constant} \Br\},
\end{multline}
and $\nbin(H)$ is the number of bins of the density of $H$. In general, solutions to~\eqref{eq:adj_ess_hist} are not unique. In that case we will pick 
$\hat{H}$ with density $\hat h = \sum_{k = 0}^K {\abs{I_k}}^{-1}{F_n(I_k)}1_{I_k}$, which
maximizes the following negative entropy {(up to a factor of $n$)}
\begin{equation*}
\sum_{k = 0}^K n F_n(I_k)\log \left(\frac{F_n(I_k)}{\abs{I_k}}\right).
\end{equation*}
This is the log-likelihood if we assume the data are distributed according to $\hat{H}$. Thus, we select the one that explains data best in terms of likelihood among all 
solutions of~\eqref{eq:adj_ess_hist}.  We refer to this solution as the {essential histogram}.

Since $\tilde{C}_n(\alp)$ is a superset of the histogram distribution functions in $C_n(\alp)$, the
minimization problem \eqref{eq:adj_ess_hist} over histogram distribution functions $H \in \tilde{C}_{n}(\alp)$
will result in a solution that may have fewer bins than the minimizer
over $C_n(\alp)$, which is a beneficial side effect. In turn, $\tilde{C}_{n}(\alp)$
involves fewer goodness of fit constraints, which may result in some loss in
efficiency in inference. In the next sections, the theoretical results and the simulations show that this loss is not significant. Moreover, such computational relaxation still allows to derive guaranteed finite sample confidence
statements about certain features of the distribution.

\subsection{Numerical computation}\label{subAlg}

{For brevity, we focus here on the main ideas, and defer the technical details to \S\ref{sp:cmp} in the appendix. The implementation is provided in the R-package {essHist} on CRAN. }

{Computation of the threshold $\kappa_n(\alp)$ in~\eqref{eq:Cn}:} 
In case of continuous $F$, the distribution of $T_n$ is independent of $F$, so $\kappa_n(\alpha)$ can be determined by setting $F$ as e.g.~uniform, which leads to a confidence level exact at $1-\alpha$.  In general case, where $F$ can be discontinuous, the distribution of  $T_n$ may depend on the unknown $F$. However, it is always stochastically bounded from above by a universal distribution, defined via a slight variant of $T_n$ with $F$ being uniform. In particular, this implies that there exists some $\kappa^*_n(\alpha)$, satisfying $\kappa_n(\alpha) \le \kappa_n^*(\alpha) $ and $\sup_n\kappa_n^*(\alpha) < \infty$, see Lemma~\ref{le:ks}. This ensures that, if using $\kappa_n^*(\alpha)$ instead of $\kappa_n(\alpha)$ as the threshold in~\eqref{eq:Cn}, the confidence level is always $\ge1-\alpha$, and all our theoretical results  remain valid. The choice of $\kappa_n^*(\alpha)$, compared to $\kappa_n(\alpha)$, makes the inference slightly conservative, but it is not consequential for the empirical performance of the essential histogram. In practice, we will choose $\kappa_n^*(\alpha)$ as the threshold in~\eqref{eq:Cn} when there are tied observations; otherwise, we treat $F$ as continuous, and estimate the threshold $\kappa_n(\alpha)$ by setting $F$ as uniform. In all experiments in the paper, $\kappa_n(\alpha)$ or $\kappa_n^*(\alpha)$ is estimated by  5,000 Monte-Carlo simulations, which needs to be done only once for a fixed sample size $n$, {and can be approximated for large $n$, see \S\ref{ss:qnt}}

{Computation of the essential histogram:} By $X_{(1)}, \ldots, X_{(n)}$ we denote the order statistics of observations $X_1, \ldots, X_n$. We treat each $X_{(i)}$ as a node in a graph, and set the edge length between nodes $X_{(i)}$ and $X_{(j)}$ as the minimal number of blocks of a step function on $(X_{(i)}, X_{(j)}]$ that satisfies the multiscale constraint~\eqref{eq:tldCn}.  Then the computation of the essential histogram is to find the shortest path between $X_{(1)}$ and $X_{(n)}$, which can be exactly computed by dynamic programming algorithms, see e.g.~\citet{Dij59}, with computation complexity $O(n^3)$. To improve computational speed, we exploit an accelerated dynamic program by incorporating pruning ideas~\citep[see e.g.][]{KillFeaEck12,FrMuSi14,MHRF17,HRFB17}. The constraint that the estimator itself should be a histogram has been incorporated into the dynamic programming algorithm. The resulting accelerated dynamic program is significantly faster than the standard dynamic program, and most of the time has nearly linear computation complexity in terms of sample size, with the worst case computation complexity being quadratic up to a log factor (which happens very rarely). This is confirmed by its empirical time complexity, which is almost linear (cf.~Fig.~\ref{fig:claw_tm}). Moreover, the memory complexity is always $O(n)$.

\section{Optimal estimation of probabilities}  \label{optimalprobs}

For ease of exposition, we assume here and in \S\ref{optimalfeatures} that the underlying distribution function $F$ is continuous. We stress, however, that our methodology is designed for arbitrary and possibly discontinuous $F$. Recall that the distribution of $T_n$ under any $F$ is stochastically bounded from above by a universal distribution. Thus, the theoretical guarantee in Theorems~\ref{thmA}, \ref{thmFeatureInfer} and \ref{thmD2} carry over to any discontinuous $F$ with natural modifications, and so does the upper bound in Theorem~\ref{thmC1}  (first part). In contrast, the lower bounds, i.e., in Theorems~\ref{thmC1} (second part),~\ref{thmD} and~\ref{thmB}, require the assumption of continuity on $F$ to distinguish from e.g., the pure deterministic case.  Some further optimality results and all the proofs are in \S\ref{optConfInt} and \S\ref{proofs} in the appendix.

Now we investigate how well $H \in C_n(\alp)$ performs with regard to the first goal of the histogram, namely estimating probabilities $F(I)$ for intervals $I$. To this end, for probabilities of size $p \in (0,1)$, we introduce the simultaneous standardized estimation error of $H$ as
\begin{equation}\label{eq:dp}
d_p (F,H)\ =\ \sup_{\mbox{intervals }I: F(I)=p} \frac{|H(I)-p|}{\bigl({p(1-p)}\bigr)^{1/2}}.
\end{equation}
Note that $d_p (F,H) = d_{1-p} (F,H)$. Thus, it suffices to consider $p \in (0,1/2]$ for $d_p (F,H)$. 

The first result establishes a benchmark for this task by deriving the performance of the empirical distribution function $F_n$. It shows that $d_p(F, F_n)$ is very close to $\left({2 \log (e/p_n) /n}\right)^{1/2}$.
\begin{theorem}  \label{thmC1}
For $B_n \ra \infty$ arbitrarily slowly {as $n\ra\infty$}, it holds uniformly in $F$ that
$$
\Pr_F \Bl( {n}^{1/2}\ d_p (F,F_n) \leq \bigl({2 \log \frac{e}{p}}\bigr)^{1/2} +B_n\ \mbox{ for all }
p \in \Bl[\frac{\log^2n}{n},\frac{1}{2}\Br] \Br) \ra 1.
$$
Furthermore, if $ p_n \ge n^{-1}{\log^2n}$, and $p_n \ra 0$, then, uniformly in $F$,
$$
\Pr_F \Bl( {n}^{1/2}\ d_{p_n} (F,F_n) \geq \bigl({2 \log \frac{e}{p_n}}\bigr)^{1/2} -B_n \Br) \ra 1.
$$
\end{theorem}

In fact, no estimator can improve on the $\bigl({2 \log (e/p_n) /n}\bigr)^{1/2}$ bound, as explained in the proof of Theorem~\ref{thmC1}. Thus, $F_n$ provides an optimal estimator for the collection $(F(I) )_I$. The next theorem shows that the distribution functions $H \in  C_n(\alp)$ nearly match this performance: the first part shows that if $H_n$ is a fixed sequence of distribution functions such that $d_{p_n}(F,H_n)$ is slightly larger than this bound, then with high probability $H_n \not\in C_n(\alp)$. However, since we optimize over $C_n(\alp)$ to find the simplest $H \in C_n(\alp)$, we need to bound the worst-case estimation error over all $H \in C_n(\alp)$. The second part of Theorem~\ref{thmA} shows that this worst-case error is at most twice the optimal bound. One readily checks that Theorem~\ref{thmA} holds also for $\tilde{C}_n(\alp)$ in place of $C_n(\alp)$ if in the definition of $d_p(F,H)$ we only consider intervals $I$ where the density of $H$ is constant.

\begin{theorem}  \label{thmA}
Let $B_n \ra \infty$ and $\eps_n = B_n \bigl(\log ({e}/{p_n}) \bigr)^{-1/2}$.
Then for $p_n \in \bigl(n^{-1}{\log^2 n}, {1}/{2}\bigr)$
$$
\sup_{H: d_{p_n}(F,H) > (1+\eps_n)\bigl({\frac{2}{n} \log \frac{e}{p_n}}\bigr)^{1/2}}
\Pr_F \Bl( H \in C_n(\alp) \Br) \ra 0 \ \ \mbox{ uniformly in $F$};
$$
Moreover, it holds uniformly in $F$ that
$$
\Pr_F\biggl( d_{p_n}(F,H) > (2+\eps_n)\bigl({\frac{2}{n} \log \frac{e}{p_n}}\bigr)^{1/2} \mbox{ for some }
H \in C_n(\alp) ,p_n \in \Bl(\frac{\log^2 n}{n},\frac{1}{2}\Br) \biggr) \ra 0.
$$
\end{theorem}

The loss of a factor $2$ is not consequential when compared to popular histogram rules: Proposition~\ref{thmC2} gives the performance of a histogram that uses $k_n$ equally sized bins. If one chooses $k_n \asymp n^{1/3}$ bins as recommended by the common rules in the literature, then ${n}^{1/2} d_{p_n}(F,H_n)$ blows up at the rate $n^{1/3}$ for some rather typical continuous $F$ and $p_n={(4k_n)^{-1}}$, while the benchmark given by $F_n$, and the worst-case error over $H \in C_n(\alp)$, grow very slowly  at a rate of $({\log n})^{1/2}$. A similar result obtains if one uses bins with equal probability content.

\begin{proposition}  \label{thmC2}
Let $H_n$ denote the distribution function of a histogram that partitions $[0,1]$ into $k_n$
equally sized bins. Then there is a continuous $F$ such that for $p_n={(4k_n)^{-1}}$
and odd $k_n$, 
$$
{n}^{1/2} \ d_{p_n}(F,H_n) \ \geq \ \frac{1}{2} \bigl({n p_n}\bigr)^{1/2}.
$$
\end{proposition}

If one is willing to make higher order smoothness assumptions on $F$, then it can be shown that the performance of these common histogram rules gets much closer to the benchmark. One key advantage of our proposed histogram is that it essentially attains the benchmark in every case by automatically adapting to the local smoothness. At the same time, some $H$ in $C_n(\alp)$ will typically have many fewer than the $n$ bins produced by $F_n$: If the underlying density is locally close to flat, then the multiscale likelihood ratio test will not exclude a candidate $H$ that has no breakpoints in that local region. Thus the $H \in C_n(\alp)$ with the fewest bins gives a simple visualization of the data while still guaranteeing essentially optimal estimation of~$(F(I) )_I$.

The optimality results for estimating $F(I)$ in Theorems~\ref{thmA} and \ref{thmB} carry over to estimating the average density $\bar{f}(I)=F(I)/|I|$ by simply dividing the inequalities by $|I|$, see \S\ref{optimalfeatures}. We note that the construction of $C_n(\alp)$ via log likelihood ratio statistic $\lLR (H(I),F_n(I))$ rather than, say, the standardized binomial statistic ${n^{1/2}}{|H(I) -F_n(I)|}{\bigl({H(I)(1-H(I))}\bigr)^{-1/2}}$ is crucial for these optimality results, see the discussion in \S\ref{optConfInt}. That section also shows that $C_n(\alp)$ is an optimal confidence region for $F$ when $d_p(F,H)$ is interpreted as a distance between $F$ and~$H$.

\section{Optimal detection of features}    \label{optimalfeatures}

Besides estimating probabilities, another important purpose of a histogram is to show important features of the distribution, such as increases or modes of the density. An important aspect of the essential histogram is that the significance level of the confidence set $\tilde{C}_n(\alp)$ automatically carries over to certain features of the essential histogram, thus making it possible to give finite sample confidence statements about features of {$\bar{f}(I) = F(I)/\abs{I}$, which provides a measure of the average density over $I$ without any smoothness assumptions on $F$.} This is a noteworthy advantage of the essential histogram that is not shared by many other histogram rules. Such confidence statements about features of $\bar{f}$ can be derived from the following simultaneous confidence statement about $\bar{f}$:

\begin{theorem}\label{thmFeatureInfer}
Let $c_n(I) =\pen(F_n(I))  + \kappa_n(\alpha)$ with $\pen(F_n(I))$ in ~\eqref{eq:pen}, 
and $$r_n(I) =\frac{2c_n(I)}{\abs{I}}\left(\left({\frac{F_n(I)(1-F_n(I))}{n}}\right)^{1/2}+
\frac{c_n(I)}{2{n}}\right).$$ 
Then with confidence at least $1-\alp$
\be  \label{signpattern}
\Bigl|\bar{f}(I)-\bar{h}(I)\Bigr| \ \leq \ r_n(I)
\ee
simultaneously for all $I \in \JJ$ and all $H \in \tilde{C}_n(\alpha)$ whose density
$h$ is constant on $I$.
\end{theorem}

This simultaneous confidence statement can be used, for example, to establish finite sample lower confidence bounds on the number of modes and troughs of $\bar{f}$: It follows from (\ref{signpattern}) that with confidence at least $1-\alp$,  $\bar{f}(I)-\bar{f}(J)$ must have the same sign as $\bar{h}(I)-\bar{h}(J)$ whenever $\bigl|\bar{h}(I)-\bar{h}(J)\bigr|\geq r_n(I) + r_n(J)$. Therefore, if one can find intervals $I_1 < J_1 \le I_2 < J_2 \le \cdots \le I_m < J_m$ (where the inequalities are understood elementwise) such that $(-1)^{k+1} (\bar{h}(I_k)-\bar{h}(J_k)) > r_n(I_k) + r_n(J_k)$ for $k=1,\ldots,m$, then one can conclude with confidence at least $1-\alp$ that $(-1)^{k+1} (\bar{f}(I_k)-\bar{f}(J_k)) >0$, hence $\bar{f}$ has at least $\lfloor {m}/{2} \rfloor +1$ modes and  $\lfloor {m}/{2} \rfloor$ troughs. If $F$ has a density $f$, then $\bar{f}(I) -\bar{f}(J)>0$ implies $f(x)>f(y)$ for some $x \in I, y \in J$, so this confidence bound then applies to the density $f$ as well. See Fig.~\ref{fig:empIntro} for an illustration.

We now show that the essential histogram is even optimal in reproducing such increases and decreases, in the sense that it will show an increase if the size of the increase in the underlying distribution is just above the threshold below which detection is asymptotically not possible. Since we are considering general distribution functions $F$ and we do not want to make any smoothness assumptions, we will quantify the size of an increase via $\bar{f}$. We consider a set  $\II_n(c)$ of distribution functions which have an increase in
$\bar{f}$ whose size is parametrized by $c>0$:
\begin{align}\label{Ist}
\II_n(c)=&\bigg\{ F:\mbox{ there exist} \mbox{ disjoint intervals $I_1<I_2$ s.t. }
  \frac{\log^2n}{n}<F(I_i) \leq p_n \\
&\mbox{ for } i  = 1,2, \nn  \mbox{ and }
  \bar{f}(I_2)-\bar{f}(I_1) > c\ \sum_{i=1}^2{\frac{\bigl(2F(I_i)(1-F(I_i))\log \frac{e}{F(I_i)}\bigr)^{1/2}}{n^{1/2}|I_i|}}
\bigg\},  
\end{align}
where $p_n \in (2n^{-1}{\log^2n}, {1}/{2})$ is any given sequence, which for simplicity we omit from the notation $\II_n(c) =\II_n(c,p_n)$. Theorem~\ref{thmD} shows that it is not possible to reliably detect an increase in $\bar{f}$ if $F \in \II_n(1-\eps_n)$ with $\eps_n \downarrow 0$ slowly enough, as no test to this effect can have nontrivial asymptotic power.  In contrast, the first part of Theorem~\ref{thmD2} establishes that with asymptotic probability one the essential histogram will show the increase if $F \in \II_n(1+\eps_n)$. This result clearly also applies to the simultaneous reproduction of a finite number of increases/decreases and hence to the reproduction of modes. Thus the essential histogram has the desirable property that it will show increases and modes of $\bar{f}$ once the evidence in the data is strong enough to make their detection possible in principle. Conversely, one needs to keep in mind that the presence of a feature such as an increase in the essential histogram does not automatically imply that the feature is present in $\bar{f}$: Such an inferential confidence statement requires that the essential histogram shows an increase that exceeds a certain size, as detailed in Theorem~\ref{thmFeatureInfer} and the subsequent exposition. The second part of Theorem~\ref{thmD2} shows that this condition is met if $F \in \II_n(3+\eps_n)$, losing only a factor of 3 on the optimal bound.  This mirrors the result on the estimation of probabilities in Theorem~\ref{thmA}, where a similar loss was found not to be consequential. Therefore, not only has the essential histogram the advantage that it can provide confidence statements about certain features of $\bar{f}$, but when used as such an inferential tool, the essential histogram is even rate optimal.

\begin{theorem} \label{thmD}
Let $X_1, \ldots, X_n  $ be independent samples from $ F$, and $\bs{X} = (X_1, \ldots, X_n)$. 
Assume $\phi_n(\bs{X})$ is any test with level
$\alp \in (0,1)$ under $H_0: \bar{f}$ is non-increasing, in the sense that
$\bar{f}(I_1) \ge \bar{f}(I_2)$ for all disjoint intervals $I_1 < I_2$.
If $\eps_n \in (0,1)$ with $\eps_n ({\log e/p_n})^{1/2} \rightarrow \infty$, then
$$
\inf_{F \in \II_n(1-\eps_n)}\Ex_F \phi_n (\bs{X})\ =\ \alp +o(1).
$$
\end{theorem}

\begin{theorem}\label{thmD2} 
If $\eps_n >0$ with $\eps_n ({\log e/p_n})^{1/2} \rightarrow \infty$, then
\begin{multline*}
\inf_{F \in \II_n(1+\eps_n)} \Pr_F \Bl(\mbox{every $H\in \tilde{C}_n(\alp)$ whose density
$H'$ is constant on $I_1$ and $I_2$}\\ \mbox{has
a point of increase of $H'$ in the convex hull of $(I_1 \cup I_2)$} \Br) \ \ra \ 1.
\end{multline*}
If $\eps_n >0$ with $\eps_n ({\log e/p_n})^{1/2} \rightarrow \infty$, then
\begin{multline*}
\inf_{F \in \II_n(3+\eps_n)} \Pr_F \Bl(\mbox{for every $H\in \tilde{C}_n(\alp)$ whose density 
$H'$ is constant on $I_1$ and $I_2$}\\ \mbox{the confidence statement (\ref{signpattern})
allows to conclude $\bar{f}(I_2) > \bar{f}(I_1)$ } \Br) \ \ra \ 1.
\end{multline*}
\end{theorem}

Furthermore, in the case where the underlying distribution itself
 is a histogram (i.e. has a piecewise constant density), 
we have an explicit control on the number of modes: 
\begin{theorem}\label{thmHistDensity}
Assume that  distribution function $F$ has a piecewise constant density $f = \sum_{k = 0}^K c_k 1_{(\tau_k,\tau_{k+1}]}$, with $-\infty < \tau_0 < \cdots < \tau_{K+1} < \infty$. Then for the essential histogram $h$ (with distribution function $H$) in~\eqref{eq:adj_ess_hist} it controls overestimating the number of bins
\[
 \sup_F \,\Pr_F \Bl(\nbin (H) > \nbin(F)\Br) \le \alpha.
\]
Furthermore, we define
\begin{equation}\label{eq:difficulty}
\gamma  \equiv \gamma(f) = \lambda_f \min\{\underline\theta_f, \Delta_f\},
\end{equation}
with $\underline\theta_f = \min_k c_k$, $\lambda_f = \min_k(\tau_k -\tau_{k-1})$, and $\Delta_f = \min_k\abs{c_k - c_{k-1}}$, and assume that significance level $\alpha_n \gtrsim n^{-\nu}$ for some $\nu >0$, and $\gamma_n \equiv \gamma(f_n) \ge c ({\log n/n})^{1/2}$ for some small enough $c = c(\nu)$ and a sequence of piecewise constant densities $\{f_n\}_{n \ge 1}$ with distribution functions $\{F_n\}_{n\ge 1}$. Then, for some generic $C$, it controls underestimating the number of bins, 
\[
\Pr_{F_n} \Bl(\nbin(H_n) < \nbin(F_n)\Br) \le CK_n \exp\left(-Cn\gamma_n^2\right) \qquad\text{ for } n \ge n_0;
\]
and it controls the number of modes and troughs, for $n \ge n_0$,
\begin{multline*}
\Pr_{F_n} \Bl(h_n \text{ and }f_n \text{ have the same number of modes and troughs}\Br) 
\ge 1 - \alpha_n -  C K_n \exp\left(-C n \gamma_n^2\right).
\end{multline*}
\end{theorem}

In Theorem~\ref{thmHistDensity}, the constants $c$, $C$ and $n_0$ are known explicitly, see Proposition~\ref{propHistDensity}. Thus, a sufficient condition for the {consistent} estimation of the number of bins, modes and troughs is  
\[
\gamma_n \gtrsim \frac{b_n + ({\log K_n})^{1/2} + ({\log n})^{1/2}}{{n}^{1/2}},
\]
for some $b_n \to \infty$, which can be arbitrarily slow. Further, we stress that $\gamma$ in \eqref{eq:difficulty} quantifies the underlying difficulty in estimating the numbers of modes and troughs, and the number of bins.

\section{Simulation study}    \label{examples}

\subsection{{Comparison study}}\label{ss:compare}

Now we consider various simulation scenarios that reflect a range of difficulties in density estimation and data exploration. For comparison, we include the classical histograms with equal widths of bins {\citep{Pea895}} and with equal areas of blocks {\citep{Sco92}}, and also a {more recent} multiscale density estimator by \cite{DavKov04}. The number of bins for the classical histograms is selected by the~\cite{Stu26}'s rule (default in R), and the asymptotically optimal rule in~\citet{Sco92}. Both are computed by the built-in function {\tt hist} in R.  The \DK estimator has a similar flavor as the essential histogram, defined as a solution to a variational problem under a certain multiscale constraint, but it computes only an approximate solution using taut strings together with some heuristical adjustments (e.g.~local squeezing), and hence statistical error guarantee or confidence statements appear to be difficult. It is computed by the function \code{pmden} with default parameters in the R-package {ftnonpar} on CRAN. We only report visual results here, and provide the detailed comparison in terms of mean integrated squared error, skewness, and the number of modes, etc., to \S\ref{sp:sim} in the appendix. 

{Uniform density:} Observations are from the uniform distribution $\mathcal{U}(0,1)$. The comparison result is given in Fig.~\ref{fig:uniform} and Table~\ref{tab:uniform}. The essential histogram with small significance levels ($\alp \le 0.5$) performs best as it recovers the true density almost perfectly, while with large $\alp$ (e.g.~$\alp = 0.9$), similar to the \DK estimator, it tends to include false bins. By sharp contrast, the classical histograms overall perform worst, and report many false bins $\bigl(\nbin(\hat F)-\nbin(F)\bigr)$ and thus false modes, which become even worse as the sample size increases (cf.~Table~\ref{tab:uniform}).

\begin{figure}[!t]
\centering
\begin{tabular}{llll}
{\small (a) }& {\small (b)} &{\small (c) }& {\small (d)} \\
\includegraphics[width=0.23\textwidth]{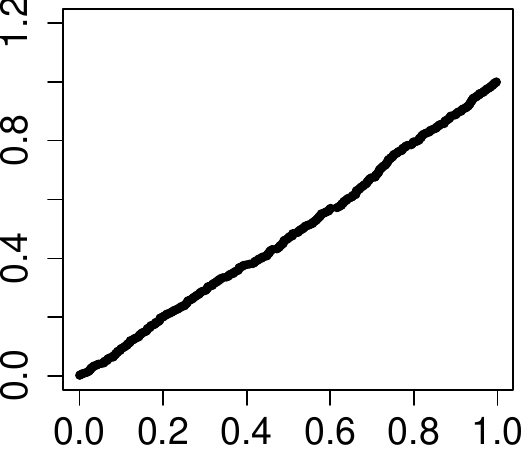} &
\includegraphics[width=0.23\textwidth]{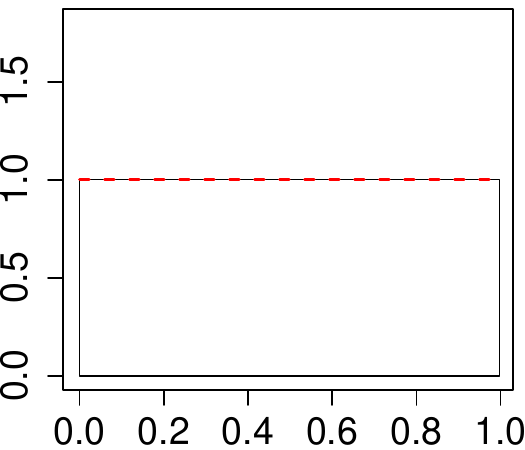} &
\includegraphics[width=0.23\textwidth]{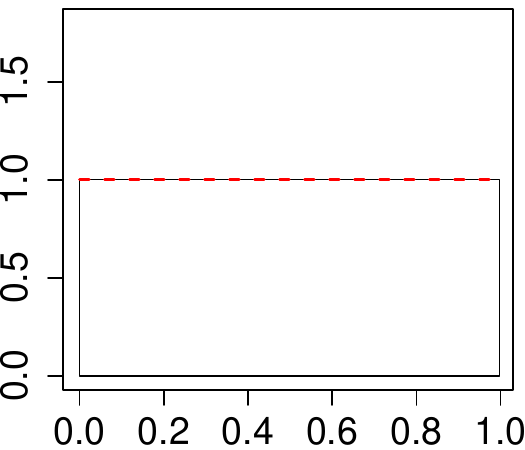} &
\includegraphics[width=0.23\textwidth]{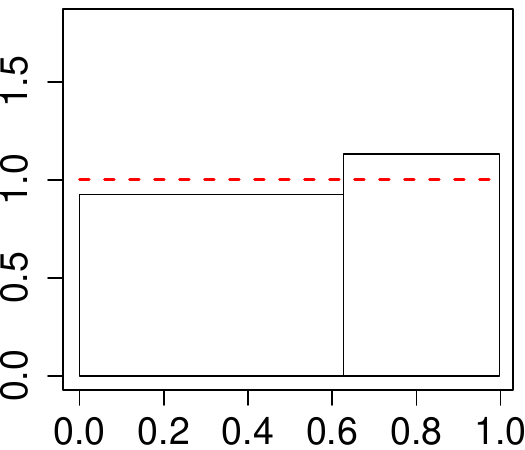}\\
{\small (e)} &{\small  (f)} &{\small (g) }&{\small  (h) }\\
\includegraphics[width=0.23\textwidth]{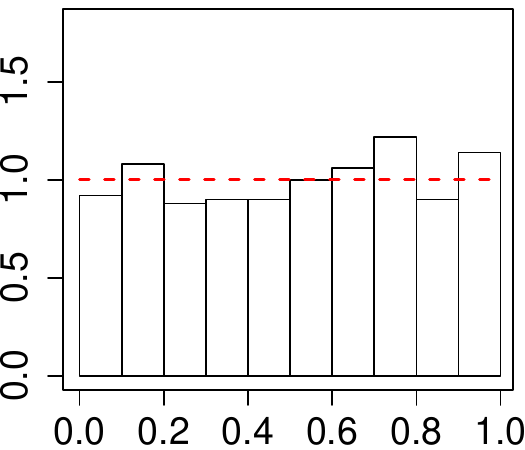} &
\includegraphics[width=0.23\textwidth]{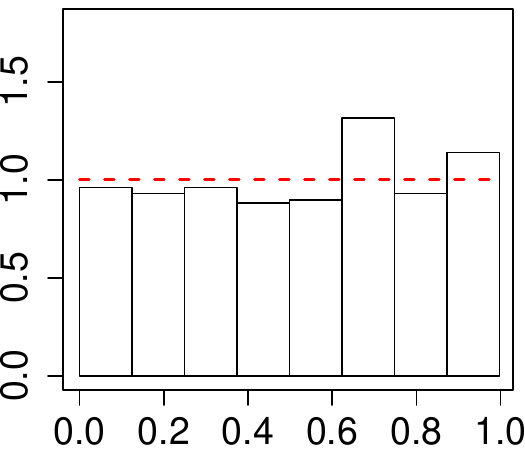} &
\includegraphics[width=0.23\textwidth]{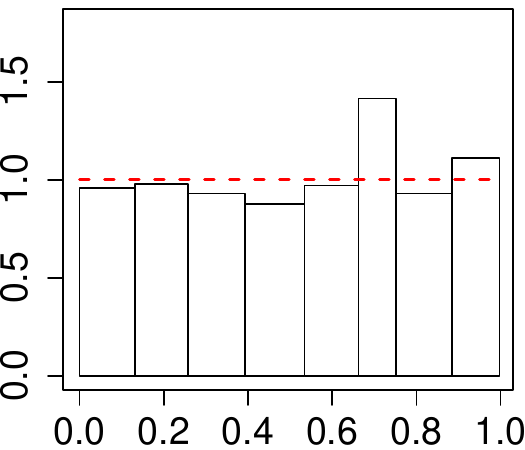} &
\includegraphics[width=0.23\textwidth]{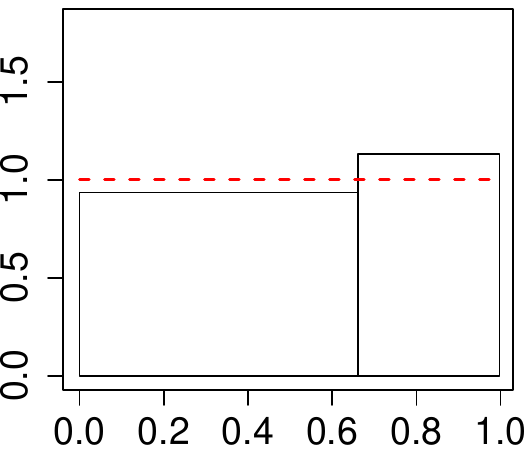} 
\end{tabular}
\caption{Uniform density: (a) the empirical distribution function; (b)-(d) the essential histograms with $\alpha = 0.1, 0.5$ and $0.9$;  (e)-(f) the equal bin width histograms with the Sturges' and Scott's rules; (g) the equal block area histogram with the Scott's rule; (h) the \DK estimator. In each panel, the true density (dashed) and sample size $n =500$.}
\label{fig:uniform}
\end{figure}

{Monotone density:} Figure \ref{fig:exponential} and Table~\ref{tab:monotone} are about the exponential distribution with unit mean. The essential histogram is better than other methods from both density estimation and feature detection perspectives, while requiring the fewest bins, which eases the interpretation of the data. The \DK estimator performs comparably well, but sometimes distorts the true shape (e.g.~the artificial spike in Fig.~\ref{fig:exponential}h). Similar to the previous example, the classical histograms are less competitive, and tend to  include more false modes as the sample size increases. In addition, the comparison results on other monotone densities (not shown) are similar to this example.

\begin{figure}[!t]
\centering
\begin{tabular}{llll}
{\small (a) }& {\small (b)} &{\small (c) }& {\small (d)} \\
\includegraphics[width=0.23\textwidth]{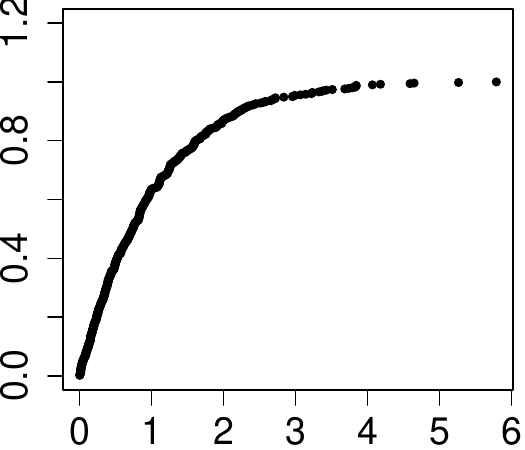} &
\includegraphics[width=0.23\textwidth]{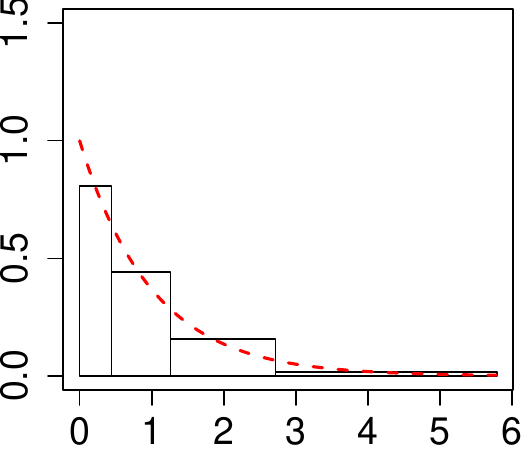} &
\includegraphics[width=0.23\textwidth]{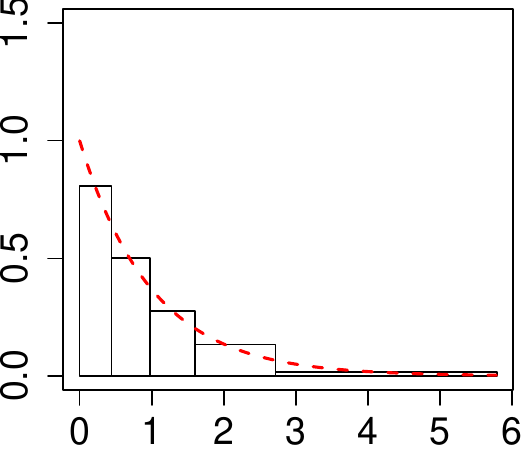} &
\includegraphics[width=0.23\textwidth]{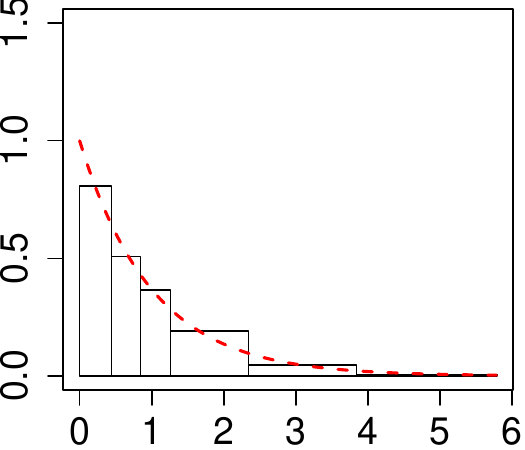}\\
{\small (e)} &{\small  (f)} &{\small (g) }&{\small  (h) }\\
\includegraphics[width=0.23\textwidth]{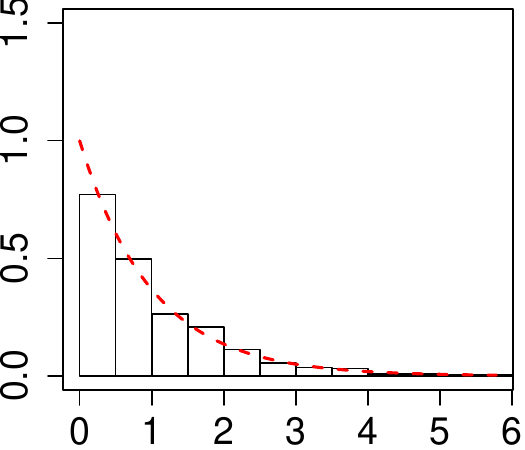} &
\includegraphics[width=0.23\textwidth]{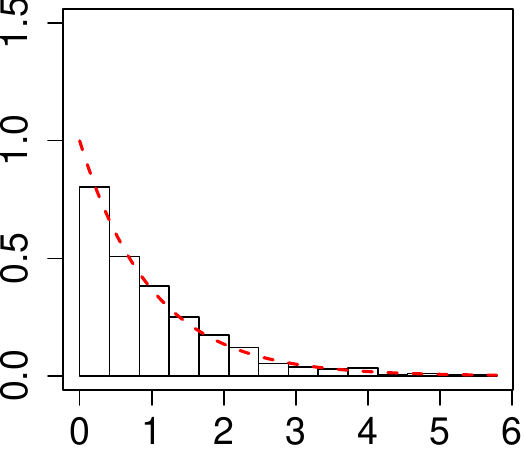} &
\includegraphics[width=0.23\textwidth]{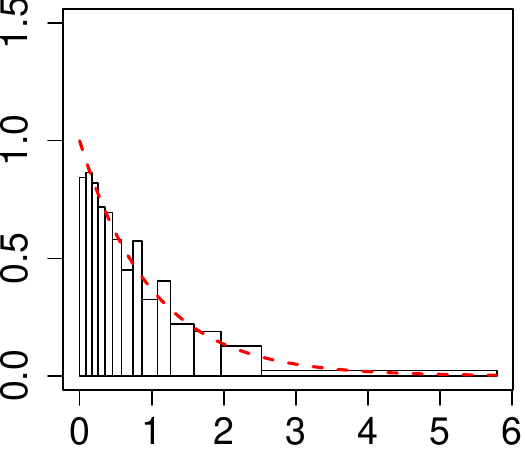} &
\includegraphics[width=0.23\textwidth]{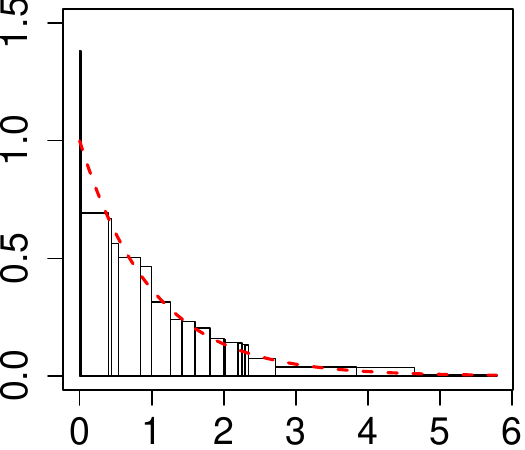} 
\end{tabular}
\caption{Exponential density: (a)-(h) are the same as in Fig.~\ref{fig:uniform}; sample size $n =500$.}
\label{fig:exponential}
\end{figure}

\begin{figure}[!h]
\centering
\begin{tabular}{llll}
{\small (a) }& {\small (b)} &{\small (c) }& {\small (d)} \\
\includegraphics[width=0.23\textwidth]{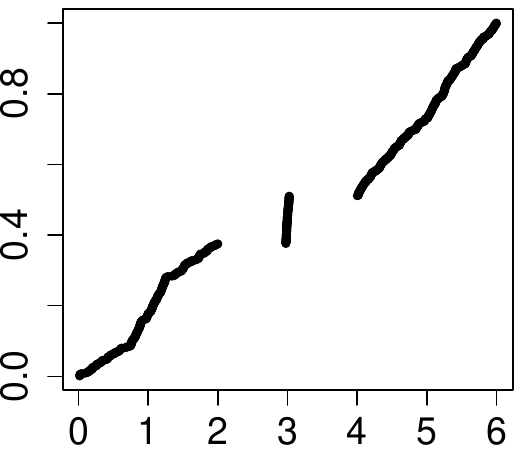} &
\includegraphics[width=0.23\textwidth]{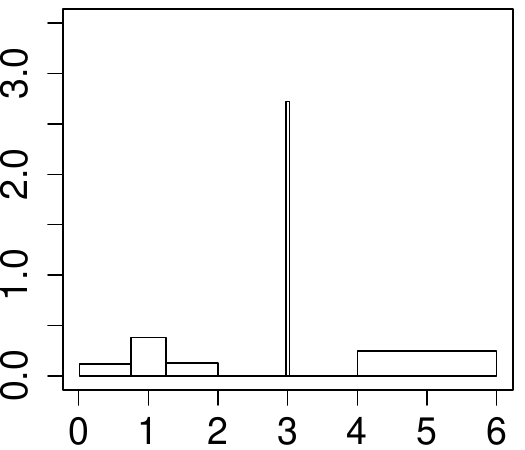} &
\includegraphics[width=0.23\textwidth]{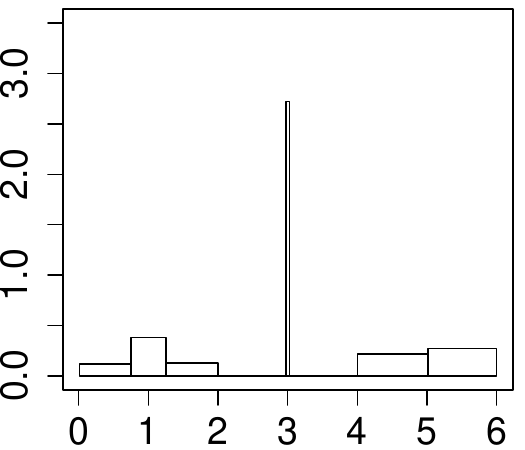} &
\includegraphics[width=0.23\textwidth]{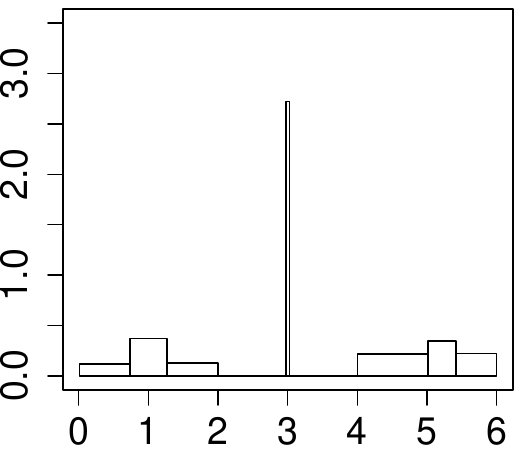}\\
{\small (e)} &{\small  (f)} &{\small (g) }&{\small  (h) }\\
\includegraphics[width=0.23\textwidth]{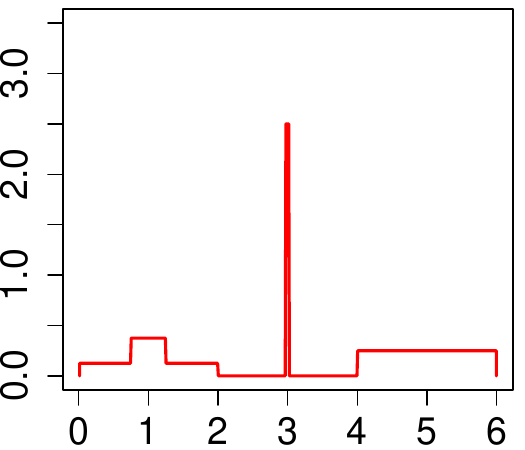} &
\includegraphics[width=0.23\textwidth]{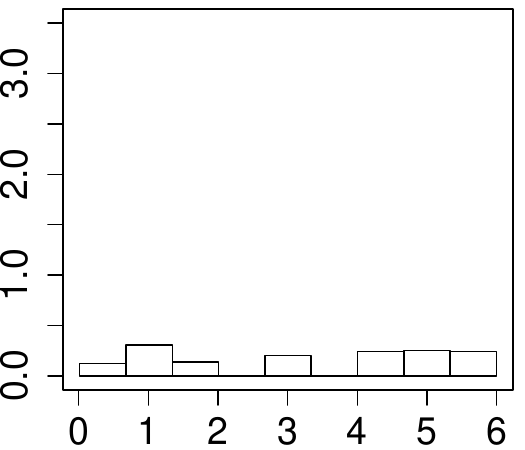} &
\includegraphics[width=0.23\textwidth]{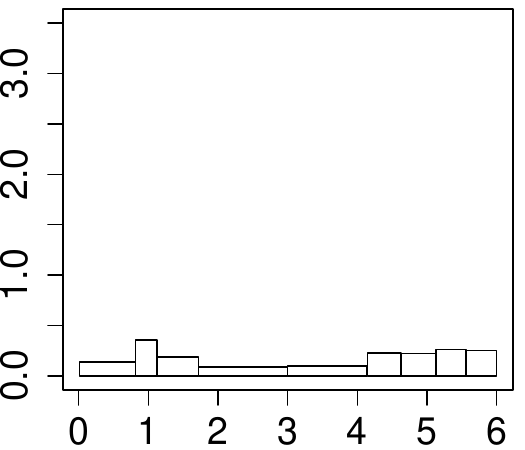} &
\includegraphics[width=0.23\textwidth]{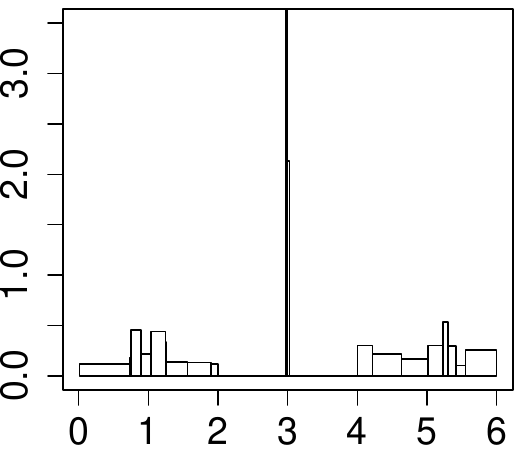} 
\end{tabular}
\caption{Histogram density: (a)-(d) and (f)-(h) are the same as in Fig.~\ref{fig:uniform}; (e) the true density; sample size $n =800$.}
\label{fig:mix_unif}
\end{figure}

{Histogram density:}
This example is about the distribution $\frac{1}{4}\mathcal{U}(0, 2) + \frac{1}{8}\mathcal{U}(0.75, 1.25) + \frac{1}{8}\mathcal{U}(2.975, 3.025) + \frac{1}{2}\mathcal{U}(4, 6)$, consisting of three different regions: an ordinary one mode region, a sharp spike region, and a flat region. The comparison is given in Fig.~\ref{fig:mix_unif}, Tables~\ref{tab:mix_unif_mode} and~\ref{tab:mix_unif_bin}. The essential histogram with a wide range of significance levels performs substantially better than all the others. It recovers all three regions of the true density fairly well, and greatly outperforms other methods with respect to the detection of correct number of bins, which reflects the theoretical finding in Theorem~\ref{thmHistDensity}. For a fixed sample size, the essential histogram tends to introduce slightly more false bins, and slightly more false modes, for larger significance levels $\alpha$. By contrast, the DK estimator often noticeably over-estimates the height of the spike, and introduces many distinct modes in the one mode and the flat regions. Its ability in identifying the true number of modes first improves, but later deteriorates as the sample size increases. The classical histograms again perform worst and seriously flatten the central spike.

{Claw density:}
The comparison on the claw density from~\cite{MarWan92} is given in Figs.~\ref{fig:claw} and~\ref{fig:claw_tm}, and Tables~\ref{tab:claw},~\ref{tab:claw_nbin} and~\ref{tab:claw_skewness}. 
The essential histogram performs well in both mode detection and density estimation for large sample sizes $n$ {or} high significance levels $\alpha$. For a fixed $n$, it recovers more {details of the density} from the data as $\alpha$ increases, at the expense of statistical confidence. This reveals the ability of the essential histogram as a potential exploratory tool for the analysis of data, {and we suggest to view the nominal level $\alpha$ as a screening parameter.} Small $\alpha$ provides reliable confidence statements in Theorems~\ref{thmFeatureInfer} and~\ref{thmD2};  a large $\alpha$ typically leads to a better recovery e.g., in mode detection. For a fixed $\alpha$, the performance of essential histogram improves as $n$ increases, {which supports the theoretical finding in Theorem~\ref{thmD2}.}  Also the essential histogram needs the fewest bins to detect the correct number of modes (Table~\ref{tab:claw_nbin}). Empirically, solutions in a range of $\alpha$ between 0.5 and 0.9 always look very similar (Fig.~\ref{fig:claw}) revealing a certain stability if estimation is the primary goal.  Moreover, the essential histogram recovers the shape of the truth in such a reliable way that the skewness of estimated histograms almost coincide with that of the truth (Table~\ref{tab:claw_skewness}).   
The \DK estimator is among the best in mode detection, while it slowly starts to include false modes as $n$ increases. However, it performs not so well in estimating the height of each mode, and the number of bins within each peak varies to a large extent (Fig.~\ref{fig:claw}). The latter potentially leads to misinterpretation of the data (e.g.,~one might wrongly infer that the peaks are of completely different shape). In addition, the \DK estimator gives the largest number of bins among all methods, which further complicates the interpretation of the data.  
For classical histograms, the Scott's rule is better than the Sturges' rule in both mode detection and skewness preservation, but it tends to report more (both true and false) modes as $n$ increases. The equal bin width histogram gives better estimation at tail region (low density), while the equal block area one is more preferable in the central region (high density). 
Regarding computation time, the essential histogram is the slowest, while being still affordable: e.g.,~it just takes around 1 second for 3,000 observations (Fig.~\ref{fig:claw_tm}). Seemingly, the computation time is of the same order for all methods, i.e., linearly increasing in~$n$.

\begin{figure}[!h]
\centering
\begin{tabular}{ll}
{\small (a)} & {\small (b)}  \\
\includegraphics[width=0.45\textwidth]{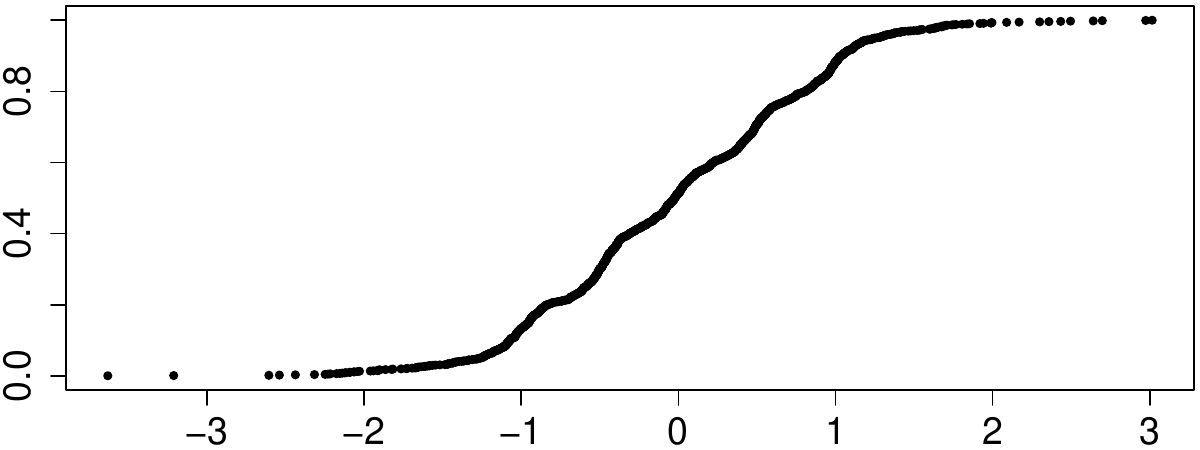} &
\includegraphics[width=0.45\textwidth]{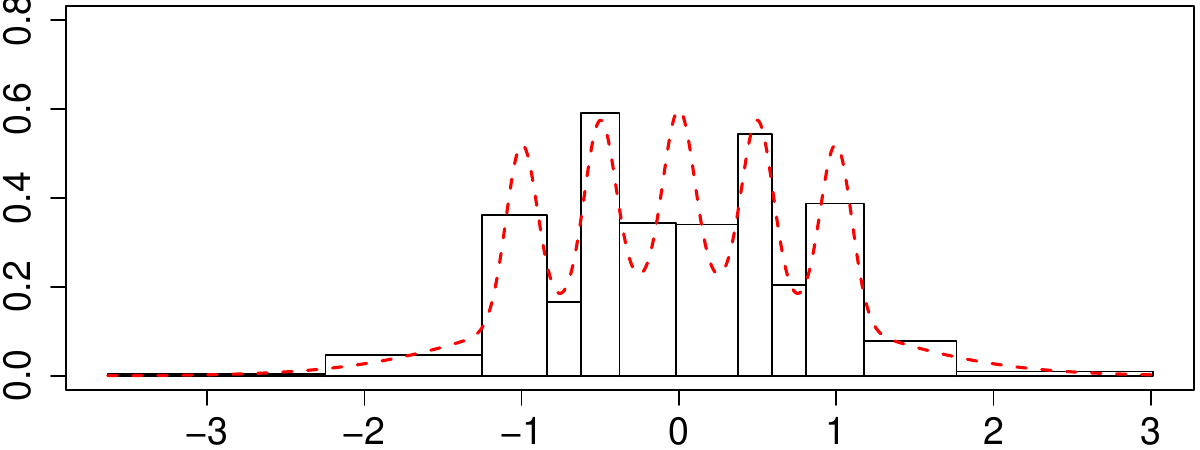}\\
{\small (c)} & {\small (d)} \\
\includegraphics[width=0.45\textwidth]{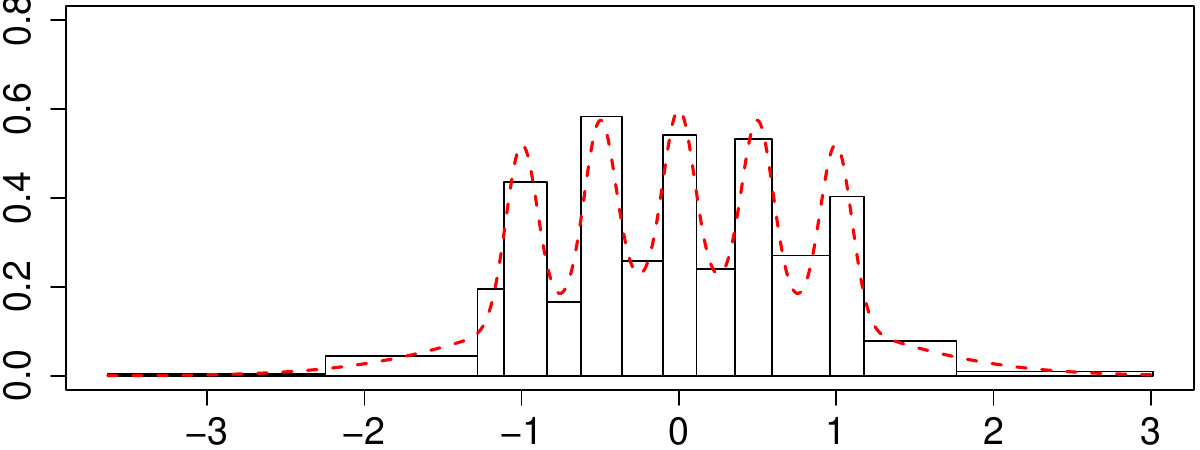} &
\includegraphics[width=0.45\textwidth]{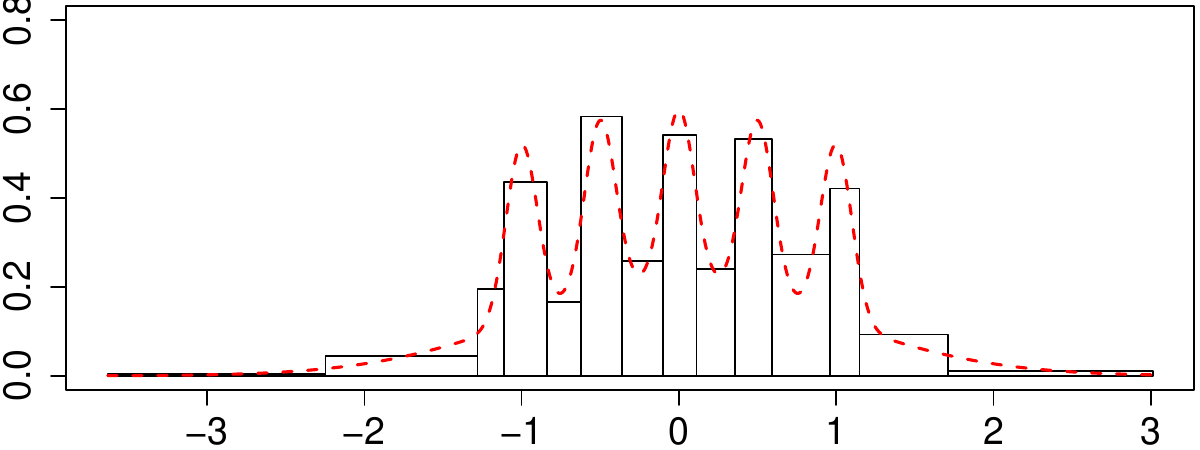} \\
{\small (e)} & {\small (f)} \\
\includegraphics[width=0.45\textwidth]{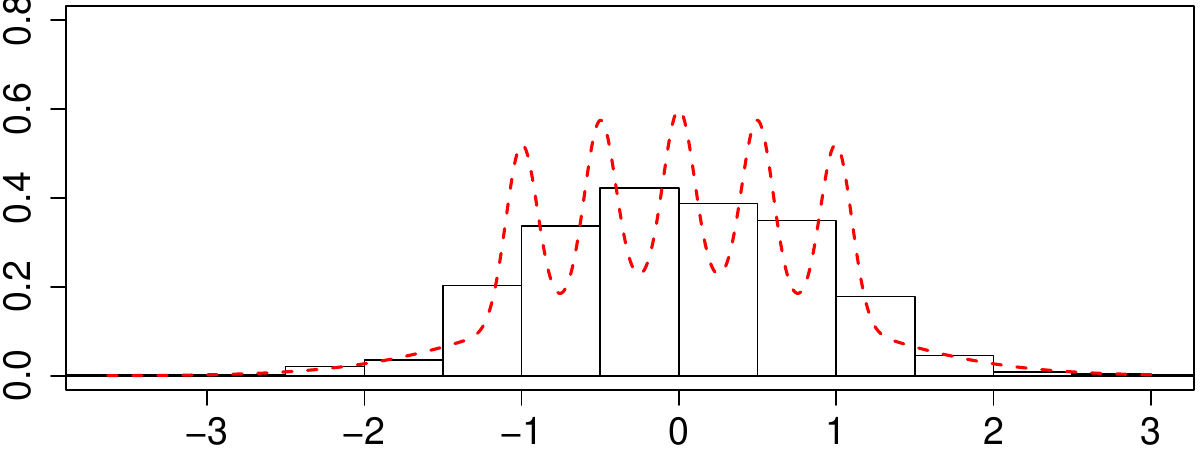} &
\includegraphics[width=0.45\textwidth]{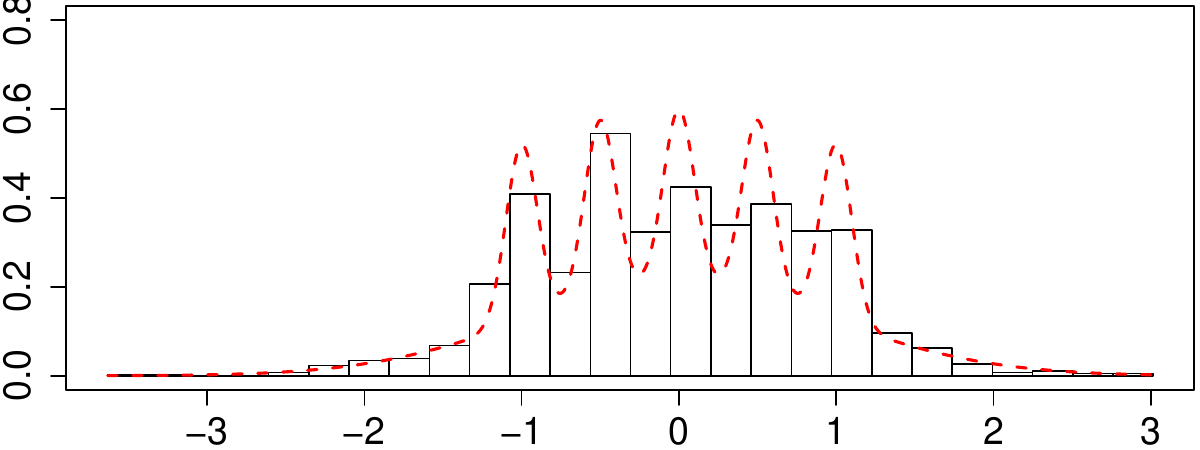}\\
{\small (g)} & {\small (h)} \\
\includegraphics[width=0.45\textwidth]{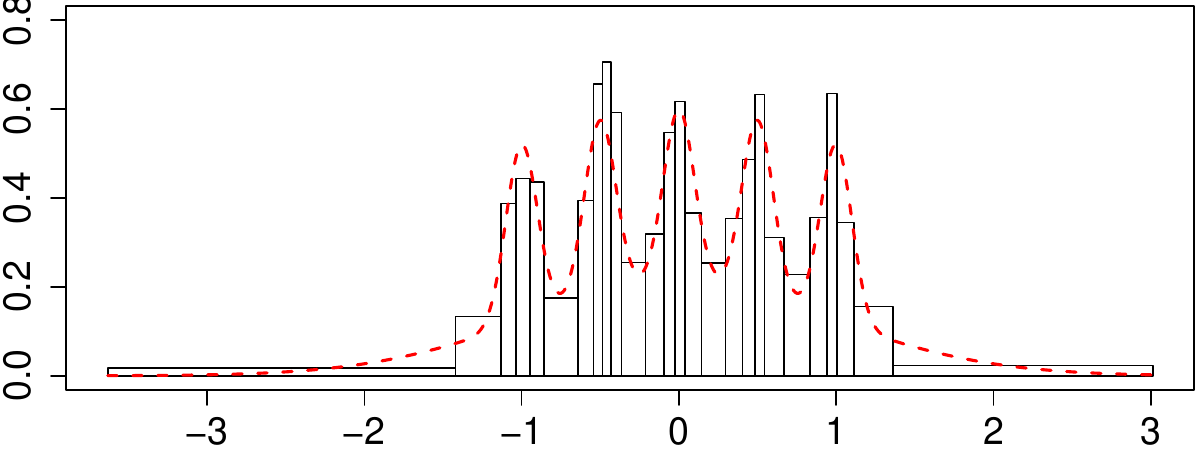} &
\includegraphics[width=0.45\textwidth]{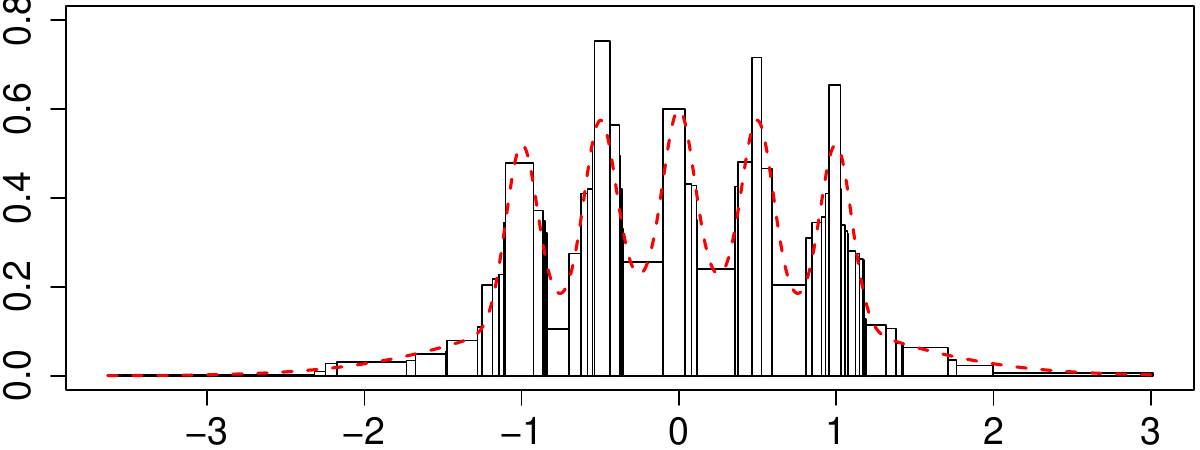} 
\end{tabular}
\caption{Claw density: (a)-(h) are the same as in Fig.~\ref{fig:uniform}; sample size $n =1,500$.}
\label{fig:claw}
\end{figure}

\begin{figure}[!h]
\centering
\includegraphics[width=0.8\textwidth]{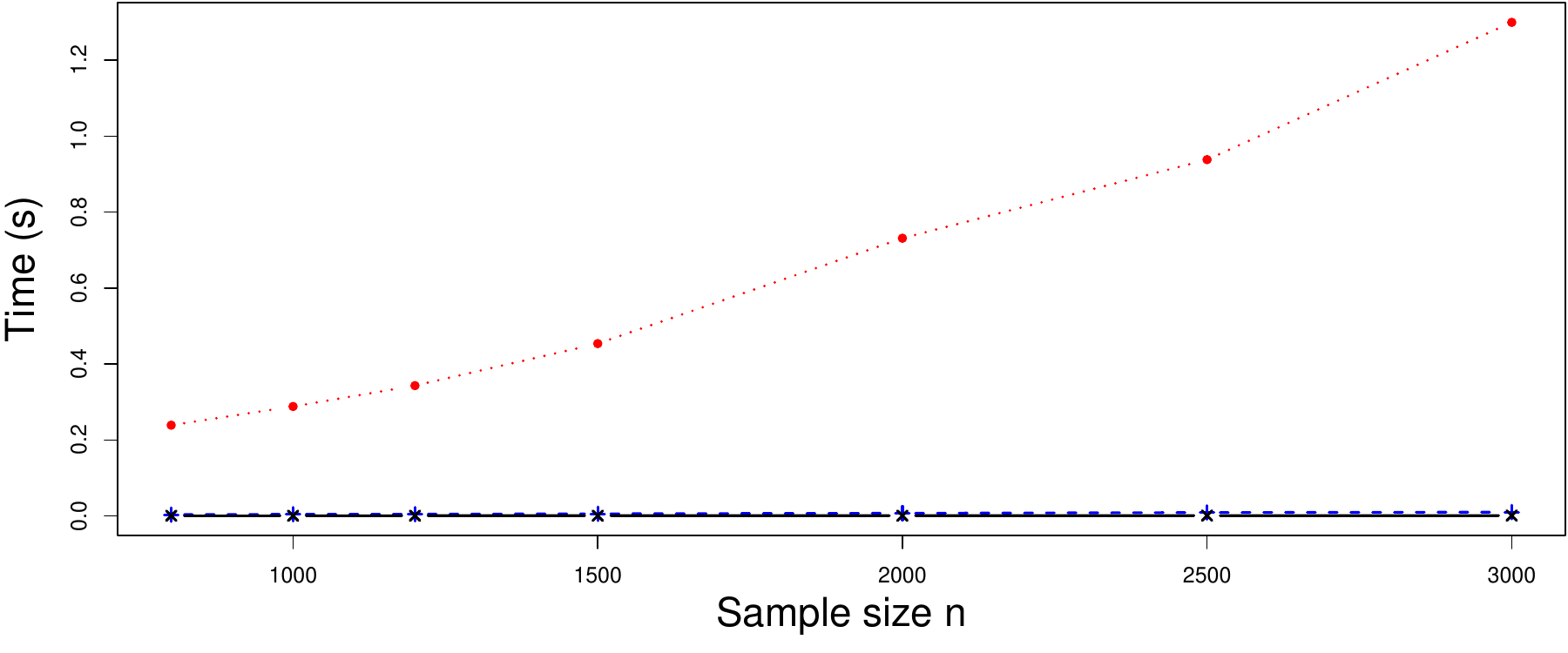}
\caption{Computation time in the claw example (a signle 3.3 GHz processor with two cores; 8GB memory) averaged over 500 runs. The classical histograms (solid), the \DK (dashed) and the essential histogram (dotted). \label{fig:claw_tm}}
\end{figure}

{Harp density:} We now consider the Gaussian mixture density, $0.2 \mathcal{N}(0, 0.5) + 0.2  \mathcal{N}(5, 1) + 0.2  \mathcal{N}(15, 2) + 0.2  \mathcal{N}(30, 4) + 0.2  \mathcal{N}(60, 8)$, termed the {harp density} due to the similarity in shape (cf.~Fig.~\ref{fig:harp}). It encodes the difficulty to have modes at several scales, increasingly more difficult to detect from left to right.  The comparison is shown in Fig.~\ref{fig:harp}, and Tables~\ref{tab:harp_mode},~\ref{tab:harp_mise},~\ref{tab:harp_kol} and~\ref{tab:harp_skew}. 
The essential histogram with various significance levels $\alpha$ is overall the best in recovering the shape of the true density (Fig.~\ref{fig:harp}), and also quantitatively regarding skewness (Table~\ref{tab:harp_skew}). Concerning mode detection, the essential histogram with larger $\alpha$ usually performs better, at the expense of lower confidence about the inference. For large sample sizes ($n \ge 1500$), the essential histogram with different $\alpha$ will eventually identify the correct number of modes (Table~\ref{tab:harp_mode}).  Further, the essential histogram outperforms all other methods in estimation error measured by the Kolmogorov metric, and is only slightly worse than DK in terms of the mean integrated squared error (Tables~\ref{tab:harp_mise} and~\ref{tab:harp_kol}). The \DK estimator is again the best in mode detection, but it has a tendency to bias the exact shapes and locations of modes (see e.g.~the local maxima near 30 in Fig.~\ref{fig:harp}); It also significantly underestimates the skewness of the truth (Table \ref{tab:harp_skew}). The classical histograms are generally less competitive; visually, the equal bin width histograms perform better in the region $[40, 60]$, while the equal block area histogram is better in $[0, 40]$, see again Fig.~\ref{fig:harp}. Moreover, the equal block area histogram is preferred in~mode detection and estimation error, but the equal bin width histogram is favored in~skewness preservation. This dilemma in deciding between these two type of histograms reflects the underlying difficulty of the problem.

\begin{figure}[!t]
\centering
\begin{tabular}{ll}
{\small (a)} & {\small (b)}  \\
\includegraphics[width=0.45\textwidth]{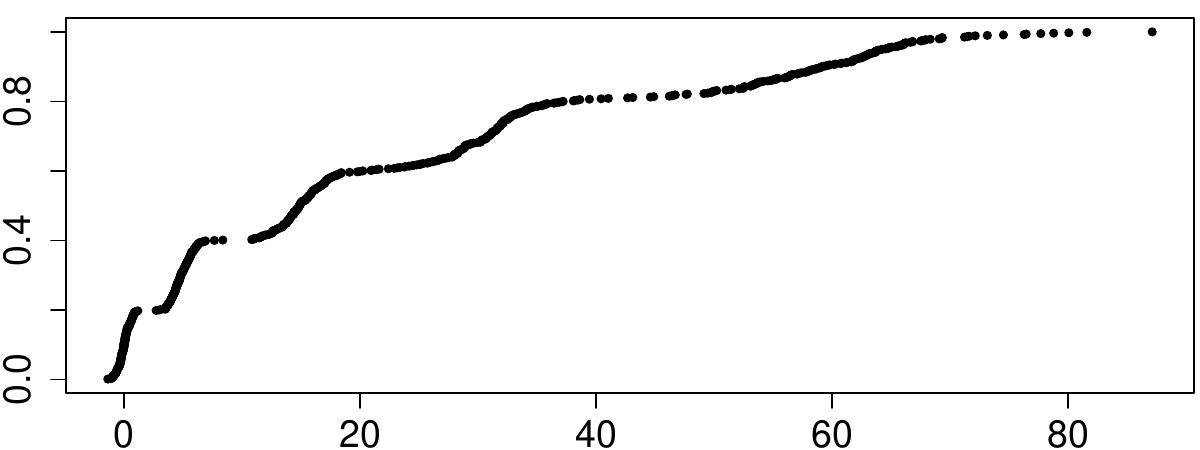} &
\includegraphics[width=0.45\textwidth]{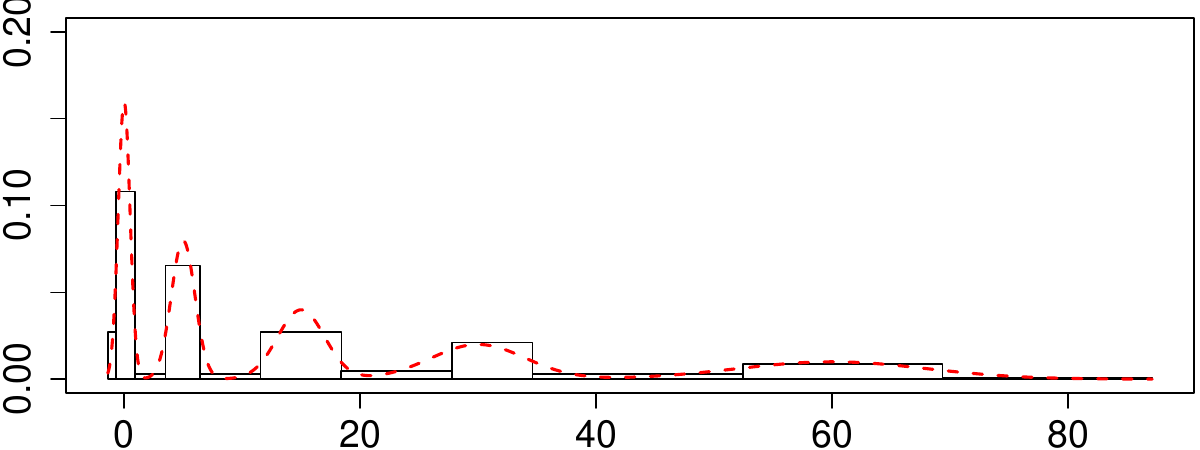}\\
{\small (c)} & {\small (d)} \\
\includegraphics[width=0.45\textwidth]{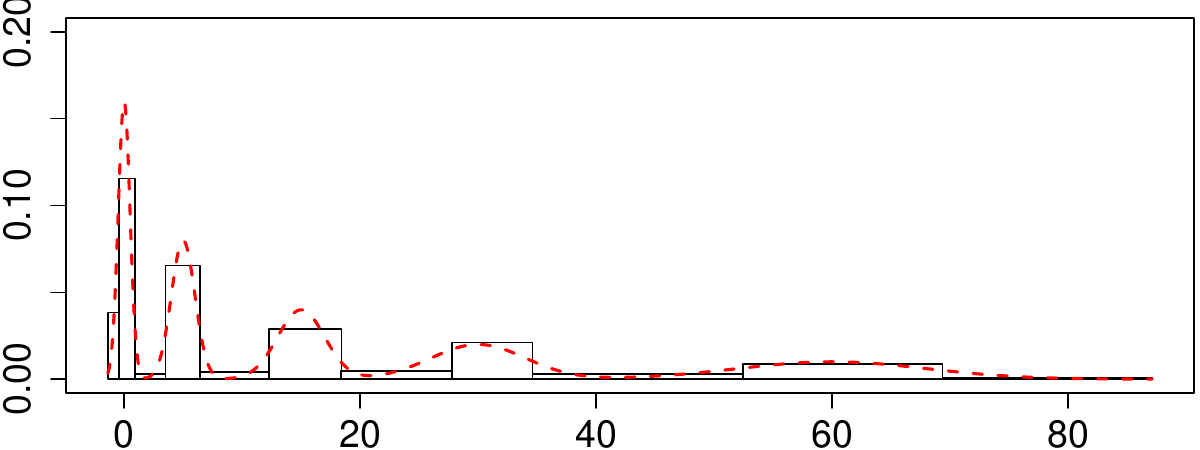} &
\includegraphics[width=0.45\textwidth]{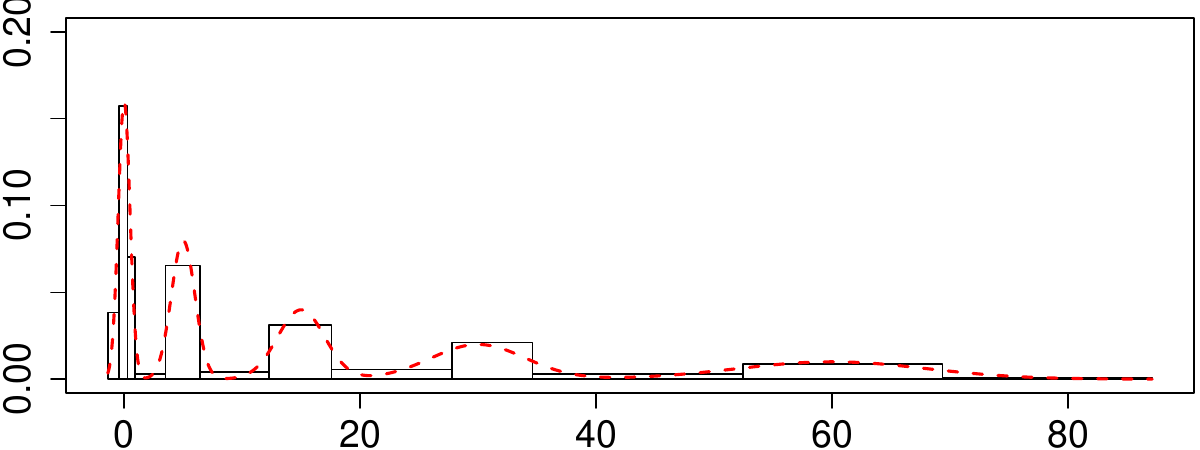} \\
{\small (e)} & {\small (f)} \\
\includegraphics[width=0.45\textwidth]{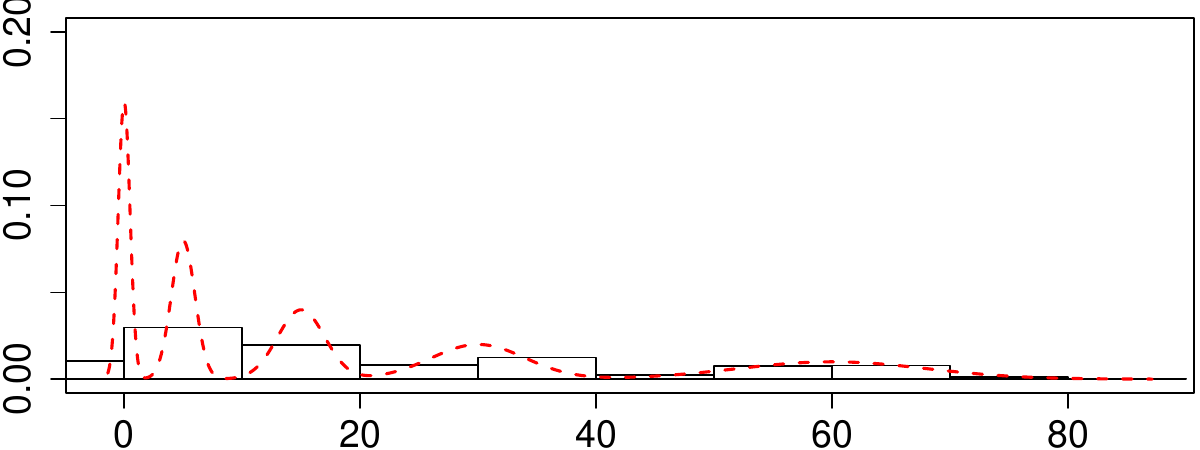} &
\includegraphics[width=0.45\textwidth]{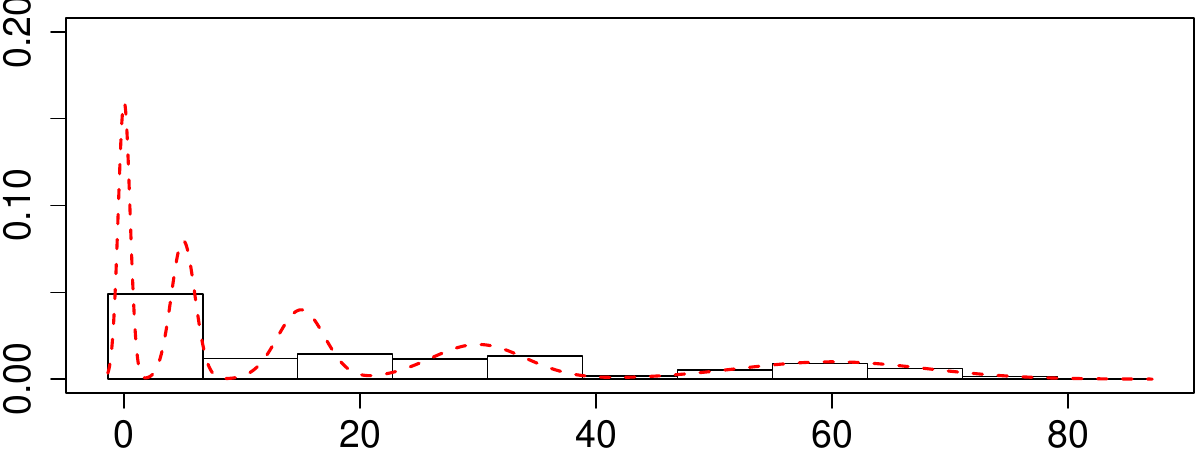}\\
{\small (g)} & {\small (h)} \\
\includegraphics[width=0.45\textwidth]{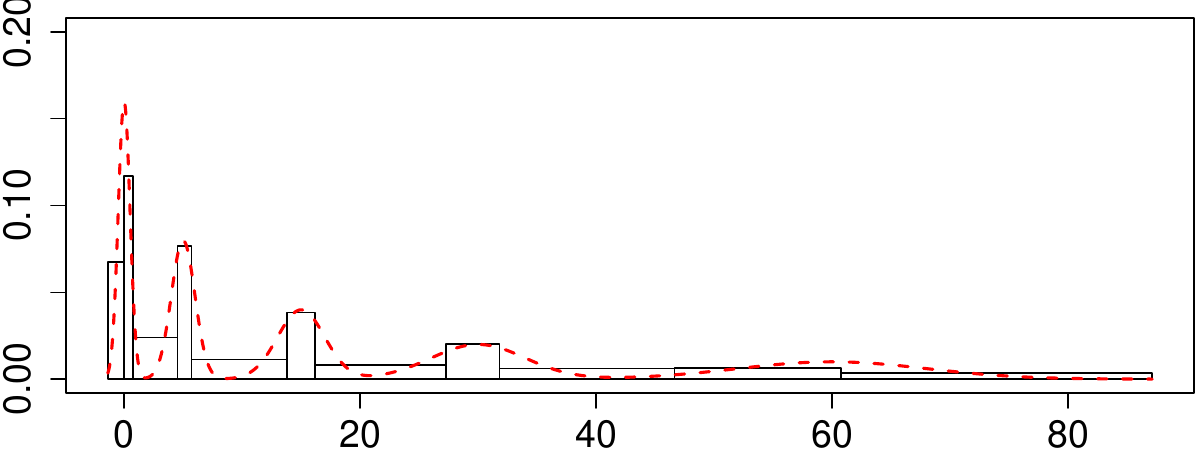} &
\includegraphics[width=0.45\textwidth]{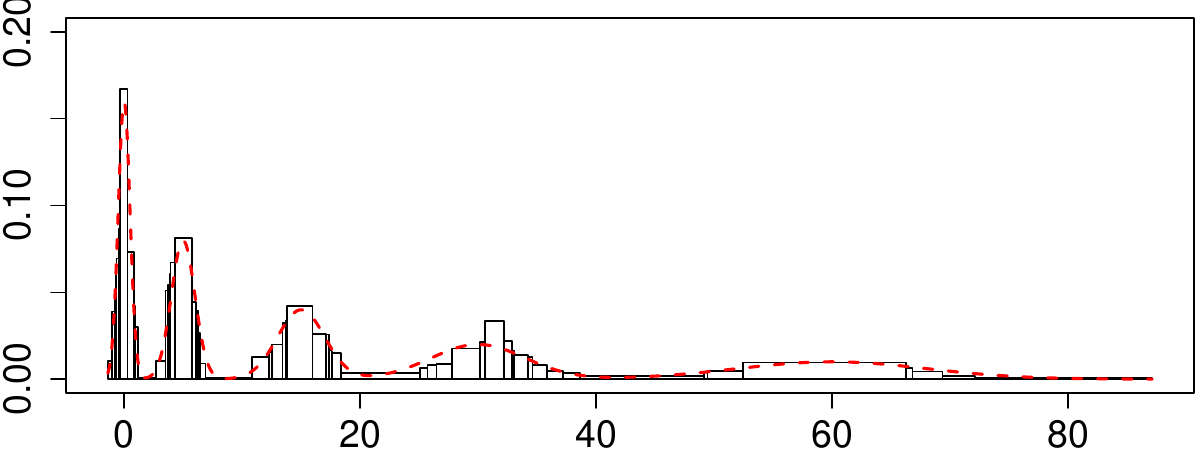} 
\end{tabular}
\caption{Harp density: (a)-(h) are the same as in Fig.~\ref{fig:uniform}; sample size $n = 800$.}
\label{fig:harp}
\end{figure}

{Heavy tails:} Figure~\ref{fig:cauchy}, Tables~\ref{tab:cauchy} and~\ref{tab:cauchy_nbin} give the comparison on the standard Cauchy density $f(x) = 1/ \bigl(\pi(1+x^2)\bigr)$, a typical one with heavy tails. Overall,  the essential histogram and the \DK estimator outperform the classical histograms. Both perform nearly perfectly in mode detection. For density estimation, the essential histogram recovers the truth quite well with only a few bins, while the \DK estimator tends to include many unnecessary slim bins, and sometimes overestimates the peak of the truth. Further, the essential histogram is the most robust against outliers, as indicated by the little changes of number of bins (Table~\ref{tab:cauchy_nbin}). For classical histograms, the one with equal bin width detects the major features, but may largely overestimate the true peak. By contrast, the one with equal block area completely distorts the shape of the truth, although still identifies the correct number of modes with moderate frequency. 

\begin{figure}[!h]
\centering
\begin{tabular}{llll}
{\small (a) }& {\small (b)} &{\small (c) }& {\small (d)} \\
\includegraphics[width=0.23\textwidth]{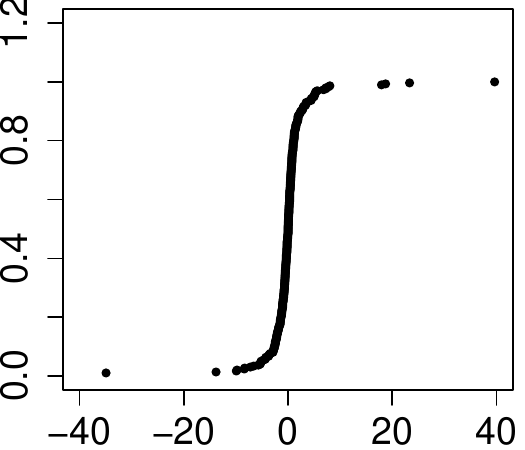} &
\includegraphics[width=0.23\textwidth]{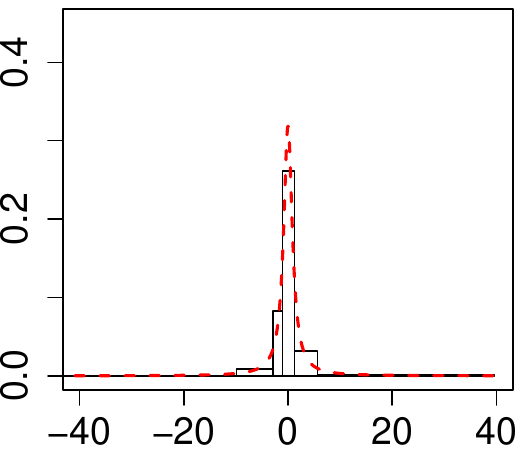} &
\includegraphics[width=0.23\textwidth]{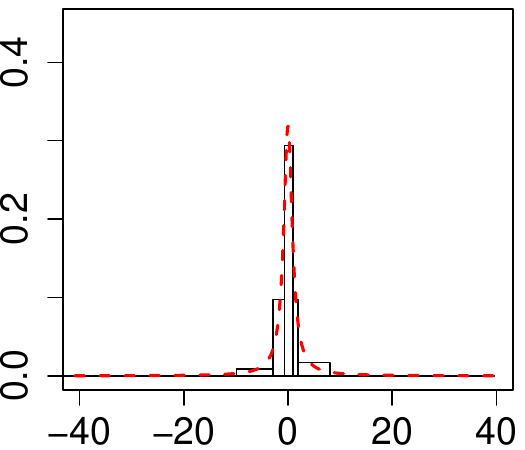} &
\includegraphics[width=0.23\textwidth]{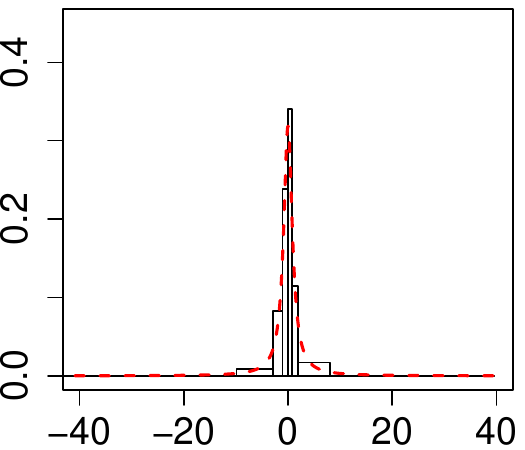}\\
{\small (e)} &{\small  (f)} &{\small (g) }&{\small  (h) }\\
\includegraphics[width=0.23\textwidth]{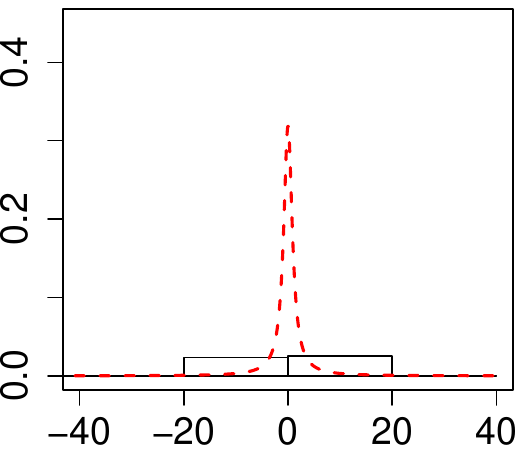} &
\includegraphics[width=0.23\textwidth]{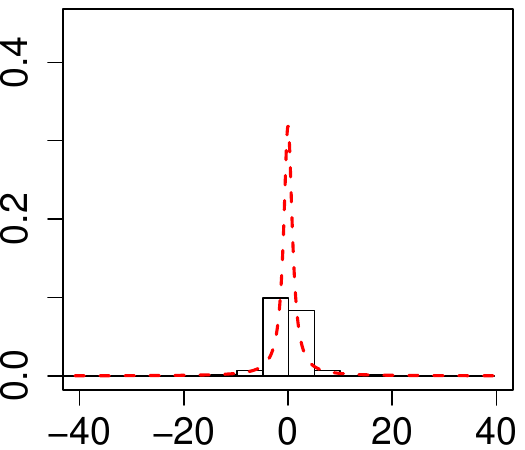} &
\includegraphics[width=0.23\textwidth]{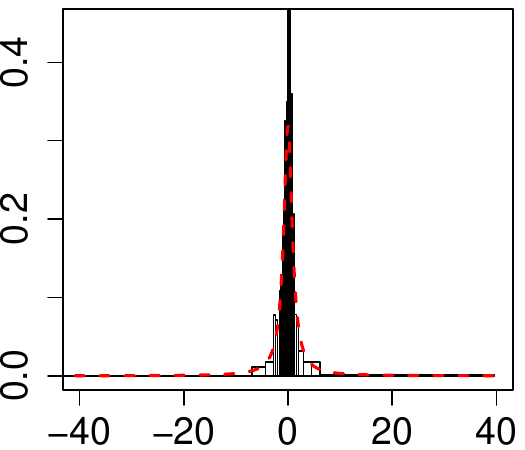} &
\includegraphics[width=0.23\textwidth]{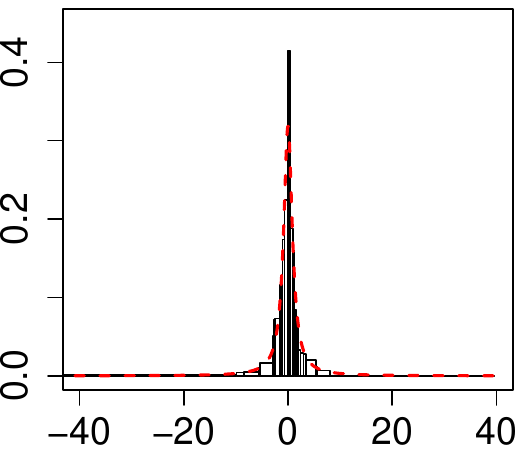} 
\end{tabular}
\caption{Cauchy density: (a)-(h) are the same as in Fig.~\ref{fig:uniform}; sample size $n =300$.}
\label{fig:cauchy}
\end{figure}

\begin{figure}[!h]
\centering
\begin{tabular}{l}
{\small (a)} \\
\includegraphics[width=0.9\textwidth]{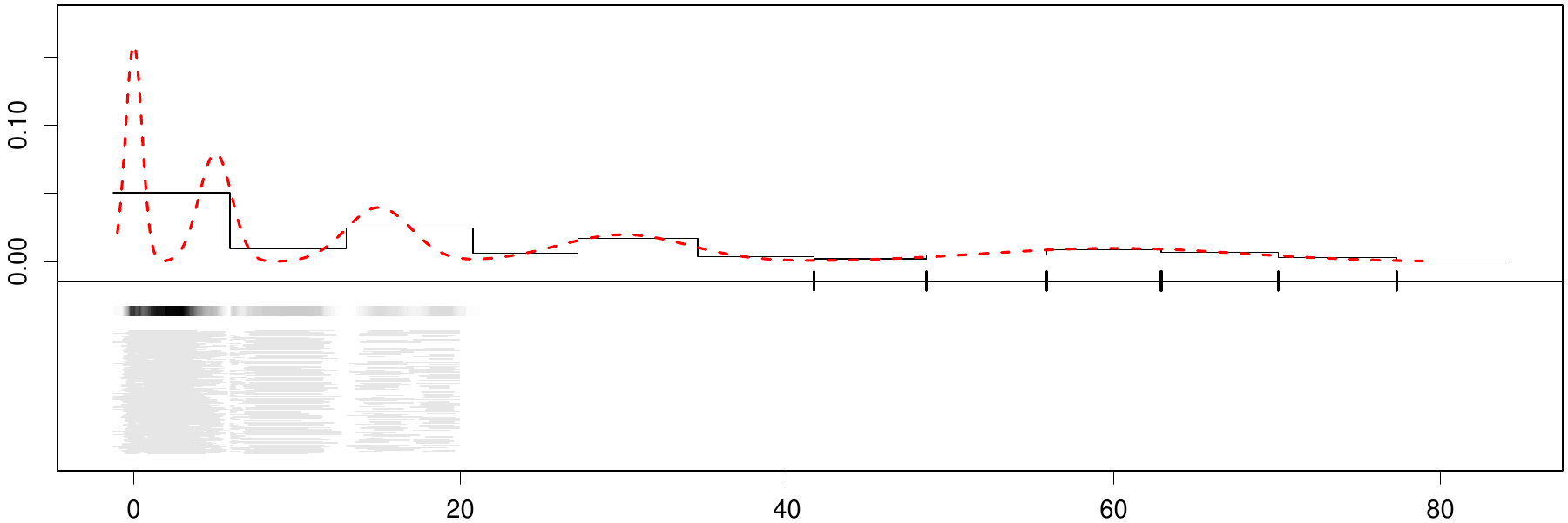}\\
{\small (b)} \\
\includegraphics[width=0.9\textwidth]{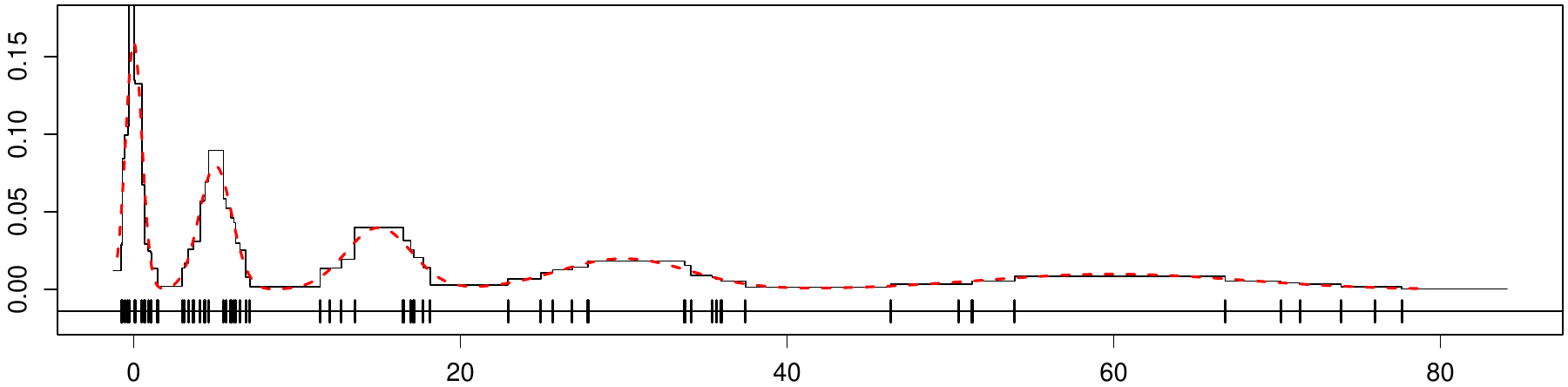} 
\end{tabular}
\caption{Evaluation for the harp example: (a) the equal bin width histogram with the Scott's rule; (b) the \DK estimator. 
In each panel, sample size $n =10^3$ and the truth (dashed); The lower part are intervals where the violation of local constraints occurs; The gray scale bar (middle) summarizes such violations with the darkness reflecting the number of violation intervals covering a location; Short vertical lines (middle) mark removable breakpoints. }
\label{fig:eval_msConst}
\end{figure}

\subsection{{Multiscale constraint as an evaluation tool}}\label{ss:evaTool}

The multiscale constraint $\tilde{C}_n(\alpha)$ in~\eqref{eq:adj_ess_hist} can be beneficial to any histogram estimator  $\hat\mu$ as an evaluation tool. We can e.g., check, for each $I$ in $\mathcal{J}$, where $\hat\mu$ is constant, whether the corresponding local constraint 
$
 \left({2 \lLR \left(\int_I \hat\mu(x)dx,F_n(I)\right)}\right)^{1/2} - \pen(F_n(I))  \leq \kappa_n(\alpha)
$ is fulfilled. The collection of all intervals where the local constraints are violated shows whether and where  $\hat\mu$ misses important features (i.e.~false negatives). Further, $\tilde{C}_n(\alpha)$ can be used to find superfluous breakpoints (i.e.~false positives). To this end, we consider whether merging the two nearby estimated segments still satisfies $\tilde{C}_n(\alpha)$. If it is the case, the breakpoint there is said to be removable. Though all the removable breakpoints cannot be simultaneously removed, any sub-collection of removable breakpoints, such that every two are not end points of a common segment, are simultaneously removable. The evaluation in terms of violation intervals and removable breakpoints is simultaneously valid with confidence $1-\alpha$, see Theorems~\ref{thmFeatureInfer} and \ref{thmD2}. An example is in Fig.~\ref{fig:eval_msConst}: The classical histogram recovers modes of medium size well, but misses the spiky modes, and reports redundant breakpoints for the wide spread modes; The \DK estimator has no violation intervals (thus missing no modes), but gives many unnecessary breakpoints. 

\section*{Acknowledgement}
We thank the editors and reviewers for constructive comments, and are grateful to Anthony Unwin for helpful comments and pointing us to several data sets which led to an improvement of our methodology. Li, Sieling and Munk are supported by the Deutsche Forschungsgemeinschaft (DFG, German Research Foundation) CRC 803, Z02 and under Germany's Excellence Strategy -- EXC 2067/1 - 390729940, and Walther by the U.S.~National Science Foundation grants DMS 1220311 and DMS 1501767.

\begin{appendices}

\section[Optimality of the confidence region]{Optimality of the confidence region $C_n(\alp)$}  \label{optConfInt}

The first part of Theorem~\ref{thmA}
and Theorem~\ref{thmB} show that $C_n(\alp)$ in \eqref{eq:Cn} is an optimal confidence region for continuous $F$ with
respect to the distance $d_p$ in \eqref{eq:dp} for arbitrary $p$: The first part of Theorem~\ref{thmA} shows
that with probability converging to one, $C_n(\alp)$ will exclude $H$ with
$d_p(F,H) \geq (1+\eps_n) \bigl({2 \log (e/p)/n}\bigr)^{1/2}$, where $\eps_n \downarrow 0$ sufficiently
slowly. In the case of small $p$, Theorem~\ref{thmB} shows that if $1+\eps_n$ is replaced by
$1-\eps_n$, then no test can distinguish $F$ and $H$ with nontrivial power.
In the case of larger $p$, i.e. when $p$ stays bounded away from zero,
the condition of the first part of Theorem~\ref{thmA} becomes $d_p(F,H) \geq n^{-1/2}B_n$
with $B_n \ra \infty$. On the other hand, a contiguity argument as in the proof of
Theorem~4.1(c) in \citet{DueWal08} shows that for any test to have
asymptotic power 1 against a sequence $H^n$ requires $d_p(F,H^n)=n^{-1/2}B_n$
with $B_n \ra \infty$.

\begin{theorem}  \label{thmB}
Let $\phi_n(\bs{X})$, with $\bs{X} =(X_1, \ldots, X_n)$,  be any test with level $\alp \in (0,1)$ under $H_0:\
X_1,\ldots,X_n$ are independent samples from a continuous distribution function $F$. If $(\log n)^2/{n} \leq p_n \ra 0$ and
$\eps_n \in (0,1)$ such that $\eps_n ({\log e/p_n})^{1/2} \ra \infty$, then
$$
\inf_{H:\ d_{p_n}(F,H) \geq (1-\eps_n) \left(\frac{2}{n} \log \frac{e}{p_n}\right)^{1/2}}
\Ex_H \phi_n(\bs{X})\ =\ \alp +o(1) \, .
$$
\end{theorem}

\begin{remark} 
The price for simultaneously considering all $H \in C_n(\alp)$
in the second part of Theorem~\ref{thmA}, as opposed to a fixed sequence $H=H_n$ in the first part,
is a doubling of the distance $d_{p_n}(F,H)$: For a fixed sequence of intervals $I=I_n$,
the standardized distance between $F(I)$ and $F_n(I)$ becomes negligible
compared to the radius $\bigl({2 \log e/F_n(I)}\bigr)^{1/2}$ of the confidence ball around
$F_n(I)$. But if one needs to consider all intervals simultaneously, then for the
worst-case interval $I$ the
standardized distance between $F_n(I)$ and $F(I)$ is also about $\bigl(2 \log e/F_n
(I)\bigr)^{1/2}$.

The proof of Theorem~\ref{thmA} shows that the first part holds even for smaller intervals,
namely for $p_n \in [2 \log (n)/n,1/2)$, provided that also $F(I) > H(I)$.
If $F(I)<H(I)$, then (\ref{RES}) requires a different bound. For example,
if $p_n=k \log (n)/n$, $k\geq 2$, then (\ref{RES}) requires
\be  \label{SMALLI}
d_{p_n} (F,H) \ >\ (1+\eps_n) ( {2}^{1/2} +{{k}^{-1/2}})
\left(\frac{1}{n} \log \frac{e}{p_n}\right)^{1/2}\,,
\ee
and it is not clear whether this result can be improved. Note that Theorem~\ref{thmB}
does not provide a lower bound for scales of order $\log (n)/n$.

The construction of $C_n(\alp)$ via the log likelihood ratio statistic
$\lLR (H(I),F_n(I))$ rather than,
say, the standardized binomial statistic ${n}^{1/2}{|H(I) -F_n(I)|}{\bigl({
H(I)(1-H(I))}\bigr)^{-1/2}}$ is crucial for these optimality results: While the tail of
${n}^{1/2}{|F(I) -F_n(I)|}{\bigl({F(I)(1-F(I))}\bigr)^{-1/2}}$ is close to sub-Gaussian, it does
vary with $F(I)$ and becomes increasingly heavy as $F(I)$ decreases to 0, see~\citet[Chapter~11.1]{ShoWel86}. It is thus not clear how to construct
a penalty that is effective in combining the evidence  on the various scales $p_n$.
For example, if $F(I)=k \log(n)/n$ for some fixed $k>0$, then the penalty
$\bigl({2 \log \cdots}\bigr)^{1/2}$ in the definition of $C_n(\alp)$ would not be sufficiently
large for the standardized binomial statistic and therefore the optimality result
(\ref{RESI}) would not hold, at least in the case $F(I)>H(I)$.
\end{remark}

\section{Proofs} \label{proofs}

{
To ease understanding, we recommend interested readers to read also~\cite{RivWal13}. In all the proofs, we will make it explicit where the continuity assumption on $F$ is used. Whenever the assertions (more precisely, the first part in Theorem~\ref{thmC1}, and Theorems~\ref{thmA}, \ref{thmFeatureInfer} and \ref{thmD2}) also hold for general, possibly discontinuous $F$,  we will detail how such an extension is possible.} Recall that $\pen(F_n(I))$ is defined in \eqref{eq:pen}.  We will make use of the following

\begin{lemma}   \label{quadapprox}
$H \in C_n(\alp)$ implies for all $I \in \JJ$:
\begin{subequations}
\begin{align}
{n}^{1/2}\frac{|H(I)-F_n(I)|}{\max(H(I),F_n(I))} & \leq c_n(I)\,, \label{eq:quadProxA}\\
{n}^{1/2}\frac{|H(I)-F_n(I)|}{{M}^{1/2}} & \leq c_n(I)\qquad \text{  if } \max(F_n(I),H(I)) \leq \frac{1}{2}\,, \label{eq:quadProxB}\\
{n}^{1/2}\frac{|H(I)-F_n(I)|}{{m}^{1/2}} & \leq c_n(I) +\frac{c_n^2(I)}{2{(mn)^{1/2}}}\,, \label{eq:quadProxC}
\end{align}
\end{subequations}
where $c_n(I)=\pen(F_n(I)) +\kappa_n(\alp)$,
$M=\max \bigl(F_n(I)(1-F_n(I)), H(I)(1-H(I))\bigr)$, and
$m=\min \bigl(F_n(I)(1-F_n(I)), H(I)(1-H(I))\bigr)$.

Likewise, Lemma~\ref{quadapprox} hold for all $H \in \tilde{C}_n(\alp)$ and $I \in \JJ$
where $H^{\prime}$ is constant.
\end{lemma}

\begin{proof}
We note that Taylor's theorem implies for $H \in C_n(\alp)$
\begin{equation}  \label{Lstar}
\begin{split}
\frac{c_n^2(I)}{2n} \ & \geq \ \frac{\lLR (H(I),F_n(I))}{n} \\
 & = \ \frac{(H(I)-F_n(I)))^2}{2 \xi (1-\xi)} \ \mbox{ for some $\xi$ between $H(I)$ and $F_n(I)$}\\
& \geq \ \begin{cases}
{(H(I)-F_n(I))^2}/{\bigl(2\max(H(I),F_n(I))\bigr)}\\
{(H(I)-F_n(I))^2}/{(2M)} & \text{if $\max(H(I),F_n(I)) \leq {1}/{2}$}
\end{cases}
\,,
\end{split}
\end{equation}
since $\xi (1-\xi)$ is increasing for $\xi \in (0,1/2]$, proving \eqref{eq:quadProxA} and \eqref{eq:quadProxB}.
The assertion \eqref{eq:quadProxC} follows by applying (K.5) in \cite{DueWel14}
to \eqref{Lstar}. 
\end{proof}

\begin{proof}[of Theorem~\ref{thmC1}]
{Note that the law of  $d_{p}(F,F_n)$ for an arbitrary $F$ is bounded above by the law of $d_{p}(F,F_n)$ for a continuous $F$, and that the latter does not depend on $F$.} Thus, we may assume $X_1,\ldots,X_n$ are independent samples from $F=U[0,1]$. The statement of the theorem is closely
related to the modulus of continuity of the uniform empirical process, see~\citet[Chapter 14.2]{ShoWel86}. Unfortunately, the available results appear not strong enough
to cover the case where $B_n \ra \infty$ slowly. Therefore we employ the Hungarian construction
together with elementary calculations and recent results about Brownian motion.

By \citet[Chapter 12.3]{ShoWel86}, there exists a sequence $W_n$ of Brownian motions {on the same probability space}
such that
$$
\limsup_n \frac{{n}^{1/2}}{\log n} \sup_{I \subset [0,1]} \Bl| {n}^{1/2} \Bl(F_n(I)-|I|\Br)
-B_n(I) \Br| \leq M < \infty \ a.s.,
$$
where $B_n(t)=W_n(t)-t W_n(1)$, {a standard Brownian bridge.} Writing

\begin{equation*}
\begin{split}
R_n & =  \sup_{I \subset [0,1]:\ |I| \in [\frac{\log^2n}{n},\frac{1}{2})}
  \left| \frac{{n}^{1/2} (F_n(I)-|I|)- W_n(I)}{\bigl({|I|(1-|I|)}\bigr)^{1/2}} \right| \\
& \leq  \sup_{I \subset [0,1]} {\frac{(2n)^{1/2}}{\log n}}
  \ \Bl| {n}^{1/2} (F_n(I)-|I|) -B_n(I) \Br| +|W_n(1)| \,,\\
\end{split}
\end{equation*}
we obtain $R_n=O_p(1)$ and
$$
\Bl| {n}^{1/2} \ d_{p}(F,F_n) - \sup_{I:|I|=p} \frac{|W_n(I)|}{\bigl({p(1-p)}\bigr)^{1/2}} \Br| \leq R_n
\ \mbox{ for all } p \in \Bl[\frac{\log^2n}{n},\frac{1}{2}\Br)\,.
$$

The first statement of the theorem now follows from 
\citet[Theorem~2.1, also Section~6.1]{DueSpo01} together with $(1-p)^{-1/2} \leq 1+2p$ and the fact that
$p ({2 \log e/p})^{1/2}$ stays bounded in $p$.

For the second claim we note that
$$
\sup_{I:|I|=p_n} \frac{|W_n(I)|}{\bigl({p_n(1-p_n)}\bigr)^{1/2}}\ \geq \ \max_{i=1,\ldots,m_n} N_i\,,
$$
where the $N_i=\bigl(W_n(i p_n)-W_n((i-1)p_n)\bigr) p_n^{-1/2}$ are independent, and identically distributed from $\mathcal{N}(0,1)$ and $m_n=\lfloor 1/p_n \rfloor$. Without loss of generality, we may assume that $B_n$ is such that $\lam_n =
({2 \log e/p_n})^{1/2} -B_n/2 \ra \infty$. Mill's ratio gives
\begin{equation*}
\begin{split}
\Pr \Bl( \max_{i=1,\ldots,m_n} & N_i \leq \lam_n \Br) \\
& \leq \Bl( 1- \frac{1}{4\lam_n} \exp (-\lam_n^2/2) \Br)^{m_n} \\
& \leq \exp \Bl\{ \frac{-m_n \frac{p_n}{e} \exp \Bl\{B_n \Bl(\bigl({2 \log e/p_n}\bigr)^{1/2}-B_n/4\Br)/2
\Br\}}{4 \lam_n} \Br\} \\
& \leq \exp \Bl\{- \frac{\exp \Bl( B_n ({2 \log e/p_n})^{1/2} /4\Br)}{5e ({2 \log e/p_n})^{1/2}} \Br\}
=o(1)\,.
\end{split}
\end{equation*}

Theorem~7.1(b) in \cite{Due03} suggests that the second statement of the theorem holds for
any collection of estimators $(H_n(I))_I$, so it is not possible to improve on the performance
of $F_n$. We will not engage in the lengthy technical work required to establish that
claim. 
\end{proof}

\begin{proof}[of Theorem~\ref{thmA}]
To avoid lengthy technical work we will prove
the theorem using $\JJ = \{$ all intervals in $\R \}$ in the definition of $C_n(\alp)$.
The technical work in \cite{RivWal13} shows that the approximating set of
intervals used in \S\ref{confset} is fine enough so that the optimality results
continue to hold with that approximating set.

To prove the first part we will show that for arbitrary $p_n \in [2 \log (n)/n,1/2)$,
$H=H_n$ and $I=I_n$ with $F(I)=p_n$ and
\begin{multline}  \label{RES}
{n}^{1/2} \frac{|H(I)-p_n|}{\bigl({p_n(1-p_n)}\bigr)^{1/2}}\ >\\  (1+\eps_n)
\Bl({2}^{1/2}
 +{\left\{\frac{\log (e/p_n)}{np_n(1-p_n)}\right\}^{1/2}} \ind(F(I)<H(I))\Br)
\left({\log \frac{e}{p_n}}\right)^{1/2}\,,
\end{multline}
we have
\be  \label{RESI}
\Pr_F(H \in C_n(\alp)) \ \ra \  0\ \ \mbox{ uniformly in $F$}\,.
\ee
The first claim follows since for $p_n \in (\log^2(n)/n,1/2)$ we have
$$
(1+\eps_n) \left({\frac{\log e/p_n}{np_n(1-p_n)}}\right)^{1/2} \ \leq \
(1+\eps_n) (\log n)^{-1/2} \ =\ o(\eps_n)\,,
$$
since $\eps_n \gg (\log e/p_n)^{-1/2} \geq (\log n)^{-1/2}$.

On the other hand, if $p_n=k \log(n)/n$ for some $k \geq 2$, then
$$
\left({\frac{\log e/p_n}{np_n(1-p_n)}}\right)^{1/2} \ =\ k^{-1/2} +o\Bl((\log n)^{-1/2}\Br)\,,
$$
yielding (\ref{SMALLI}) in the case where $F(I)<H(I)$.

To prove (\ref{RESI}) set $c_n=(1+\eps_n) ({2 \log e/p_n})^{1/2}$. Then the inequality
(\ref{RES}) reads
\be  \label{A0}
{n}^{1/2} \frac{|H(I)-p_n|}{\bigl({p_n(1-p_n)}\bigr)^{1/2}} > c_n + \frac{c_n^2}{2(1+\eps_n)
\bigl({np_n(1-p_n)}\bigr)^{1/2}}\ \ind(F(I)<H(I))\,.
\ee
We have $b_n=\min (B_n,{(\log n)^{1/2}}) \ra \infty$ by the assumption of the theorem,
and we define the event  $\AA_n =\{{n}^{1/2} |F_n(I)-F(I)| \leq \bigl({b_nF(I)(1-F(I))}\bigr)^{1/2}\}$.
We will show that on the event $\AA_n$, (\ref{A0}) implies for sufficiently large $n$ 
\be  \label{A1}
{n}^{1/2} \frac{|H(I)-F_n(I)|}{\bigl({F_n(I) (1-F_n(I))}\bigr)^{1/2}}
\geq \tilde{c}_n +\frac{\tilde{c}_n^2}{2\bigl({nF_n(I)(1-F_n(I))}\bigr)^{1/2}}
\ \ind(F_n(I)<H(I))\ ,
\ee
uniformly in $H$ and $I$, where $\tilde{c}_n=\pen(F_n(I)) +b_n^{1/2}/4$. Hence Lemma~\ref{quadapprox}~\eqref{eq:quadProxB} and~\eqref{eq:quadProxC} give
\begin{equation*}
\begin{split}
&\Pr_F(H \in C_n(\alp)) \\  
\leq\ & \ \Pr_F \biggl(\frac{|H(I)-F_n(I)|}{\bigl({F_n(I) (1-F_n(I))}\bigr)^{1/2}} \leq \tilde{c}_n +\frac{\tilde{c}_n^2}{2\bigl({nF_n(I)(1-F_n(I))}\bigr)^{1/2}}\ind(F_n(I)<H(I)) \biggr)\\
& \ \ \ \ \mbox{ once } {b_n}^{1/2}/4 > \kappa_n(\alp) \\
\leq \ & \Pr_F(\AA_n^c)\ \ \ \text{ eventually}\\
\leq \ & \frac{1}{b_n}\,,
\end{split}
\end{equation*}
by Chebychev's inequality, and the above conclusions are uniform in $F,H$ and $I$. Thus
\eqref{RESI} follows since $\kappa_n(\alp)=O(1)$ by \citet[Proposition~1]{RivWal13} if $F$ is continuous. {The above argument holds for general (possibly discontinuous)  $F$ as well; In that case, we will use $\kappa_n^*(\alpha)$ instead of $\kappa_n(\alpha)$, which is also $O(1)$ by Lemma~\ref{le:ks}. }

It remains to prove (\ref{A1}). On $\AA_n$ we have
\be  \label{A2}
\biggl| \frac{F_n(I)}{F(I)} -1\biggr|\ \leq \ \left({\frac{b_n(1-F(I))}{nF(I)}}\right)^{1/2}
\ \leq \ \left({\frac{b_n}{\log n}}\right)^{1/2}\,,
\ee
since $F(I) \geq 2 \log(n)/n$, and the same bound applies to $(1-F_n(I))/(1-F(I))$
since $F(I) \leq 1/2$. Hence
$$
\left| \left({\frac{F_n(I)(1-F_n(I))}{p_n(1-p_n)}}\right)^{1/2}-1\right| \ \leq \ \left({\frac{b_n}{\log n}}\right)^{1/2}\,.
$$
So on $\AA_n$:
\begin{equation}   \label{A3}
\begin{split}
& {n}^{1/2}\frac{|H(I)-F_n(I)|}{\bigl({F_n(I)(1-F_n(I))}\bigr)^{1/2}} \\  
\geq & \ {n}^{1/2}\frac{|H(I)-F(I)|-\sqrt{F(I)(1-F(I))b_n/n}}{\bigl({p_n(1-p_n)\bigr)^{1/2}}(1+
\sqrt{b_n/\log n})} \\
\geq &\ {n}^{1/2}\frac{|H(I)-F(I)|}{\sqrt{p_n(1-p_n)}} \Bl(1-\sqrt{\frac{b_n}{
\log n}}\Br) -{b_n}^{1/2}\\
\geq & \ \biggl( c_n+\frac{c_n^2\ \ind(F(I)<H(I))}{2(1+\eps_n)\sqrt{np_n(1-p_n)}}
\biggr) \Bl(1-\sqrt{\frac{b_n}{\log n}}\Br) -{b_n}^{1/2}\,,
\end{split}
\ee
by (\ref{A0}). {Next for sufficiently large $n$
\begin{equation*}
\begin{split}
2\log \frac{e}{p_n} \ & \geq \  2\log \frac{e}{F_n(I)(1-F_n(I))}
+ 2\log \frac{F_n(I)}{2F(I)} \\
& \geq \ 2\log \frac{e}{F_n(I)(1-F_n(I))} - 4\log 2 \\
& = \ \pen(F_n(I))^2 - 4\log 2\,,
\end{split}
\end{equation*}
since $F_n(I) \in (p_n/2, \ 1/2)$ eventually.} Hence
\begin{equation*}
\begin{split}
c_n \ & = \  {\sqrt{2\log\frac{e}{p_n}} +  2^{1/2}B_n}\\
& \geq \tilde{c}_n +b_n/4\ \ \ \text{eventually,}\\
\end{split}
\end{equation*}
and
\begin{align*}
\frac{c_n^2}{1+\eps_n} \ & = 2 \log\frac{e}{p_n} +2B_n\sqrt{\log\frac{e}{p_n}}\ \\
& \geq \ 2\log\frac{e}{p_n} + \frac{B_n}{2}\sqrt{\log \frac{e}{p_n}} +\frac{3}{2}B_n\ \ \text{eventually}\\
& \geq \ \tilde{c}_n^2\ \ \text{eventually} .
\end{align*}

Finally, (\ref{A2}) yields $(np_n(1-p_n))^{-1/2} \geq (nF_n(I)(1-F_n(I)))^{-1/2}
(1-\sqrt{b_n/\log n})$ eventually. Thus (\ref{A3}) gives
\begin{align*}
{n}^{1/2} & \frac{|H(I)-F_n(I)|}{\sqrt{F_n(I)(1-F_n(I))}}  \\
& \geq \ \biggl\{\tilde{c}_n +\frac{b_n}{4} +\frac{\tilde{c}_n^2\ \ind(F(I)<H(I))}{
2\sqrt{nF_n(I)(1-F_n(I))}}\biggr\} \Bl(1-\sqrt{\frac{b_n}{\log n}}\Br) -b_n^{1/2} \ \ \ \text{eventually}\\
& \geq \ \tilde{c}_n +\frac{\tilde{c}_n^2\ \ind(F_n(I)<H(I))}{
2\sqrt{nF_n(I)(1-F_n(I))}} \ \ \ \text{eventually,}
\end{align*}
since
\begin{align*}
\biggl(\tilde{c}_n +\frac{\tilde{c}_n^2 }{\sqrt{nF_n(I)(1-F_n(I))}}\biggr)
\sqrt{\frac{b_n}{\log n}} \ & \leq \ \Bl(\sqrt{2 \log n} +b_n^{1/2}
+\frac{6 \log n +b_n}{\sqrt{\log n}} \Br) \sqrt{\frac{b_n}{\log n}}\\
& =\ ({2}^{1/2} + 6 +o(1)) b_n^{1/2},
\end{align*}
and because $\AA_n$ and (\ref{A0}) imply that $F_n(I)<H(I)$  if and only if  $F(I) <H(I)$. 
(\ref{A1}) is proved.

To prove the second part we will consider the event
$$
\BB_n = \Bl\{\frac{n^{1/2}|F_n(I)-F(I)|}{\sqrt{F(I)(1-F(I))}}
\leq \sqrt{2\log \frac{e}{F(I)}} + b_n^{1/2} \mbox{ for all $I$ with }
\frac{\log^2 n}{n} \leq F(I) \leq \frac{1}{2}\Br\}
$$
 in lieu of $\AA_n$. Then $\Pr_F(\BB_n) \ra 1$ uniformly in $F$ by Theorem~\ref{thmC1}.
Now we proceed analogously as in the proof of the first part:
Suppose that there exist $H,p_n \in (\log^2(n)/n,1/2)$ and $I$ with $F(I)=p_n$
satisfying
\be  \label{A0II}
n^{1/2}\frac{|H(I)-p_n|}{\sqrt{p_n(1-p_n)}} \ >\ 2c_n
\ee
where $c_n=(1+\eps_n)\bigl\{2\log (e/p_n)\big\}^{1/2}$ as in the first part.
We will show that on the event $\BB_n$, (\ref{A0II}) implies, for $n$ large enough, 
\be  \label{A1II}
n^{1/2} \frac{|H(I)-F_n(I)|}{\sqrt{F_n(I)(1-F_n(I))}} >
\tilde{c}_n +\frac{\tilde{c}_n^2}{2\sqrt{nF_n(I)(1-F_n(I))}}
\ind(F_n(I)<H(I))\,,
\ee
uniformly in $H$ and $I$, where $\tilde{c}_n =\pen(F_n(I)) +b_n^{1/2}/4$ as in the first part. Hence
\begin{align*}
\Pr_F & \Bl( (\ref{A0II}) \mbox{ holds for some } H \in C_n(\alp), p_n \in
\Bl( \frac{\log^2n}{n},\frac{1}{2}\Br) \mbox{ and $I$ with } F(I)=p_n\Br)\\
& \leq \ \Pr_F \biggl( n^{1/2} \frac{|H(I)-F_n(I)|}{\sqrt{F_n(I)(1-F_n(I))}} >
\tilde{c}_n +\frac{\tilde{c}_n^2}{2\sqrt{nF_n(I)(1-F_n(I))}}
\ind(F_n(I)<H(I)) \\ 
& \hspace{2cm} \mbox{ for some } H \in C_n(\alp), p_n \in \Bl(\frac{\log^2n}{n},\frac{1}{2}\Br)\biggr)
  + \Pr_F(\BB_n^c)\ \ \text{eventually}\\
& =\ \Pr_F(\BB_n^c)\ \ \ \ \ \mbox{ once } b_n^{1/2}/4 > \kappa_n(\alp)\ \ \
 \ \ \mbox{ by Lemma~\ref{quadapprox}~\eqref{eq:quadProxB} and~\eqref{eq:quadProxC}}\\
& =\ o(1)\ \ \ \ \mbox{ uniformly in $F$},
\end{align*}
proving the second part. {(As in the first part, we use $\kappa_n^*(\alpha)$ instead of $\kappa_n(\alpha)$ for general $F$.)}

It remains to prove that on $\BB_n$, (\ref{A0II}) implies (\ref{A1II}):
On $\BB_n$ we have
\begin{equation*} 
\Bl| \frac{F_n(I)}{F(I)} -1 \Br| \ \leq \ \frac{\sqrt{2\log \bigl({e}/{F(I)}\bigr)}
+b_n^{1/2}}{\log n}\ \leq \ \frac{2}{\sqrt{\log n}}\ \ \ \text{eventually},
\end{equation*}
since $F(I)\geq \log^2(n)/n$, and the same bound holds for $(1-F_n(I))/(1-F(I))$,
hence
$$
\left| \left({\frac{F_n(I)(1-F_n(I))}{p_n(1-p_n)}}\right)^{1/2}-1 \right| \ \leq \ \frac{2}{\sqrt{\log n}}\,.
$$
Thus on $\BB_n$
\begin{equation}   \label{Astar}
\begin{aligned}
&n^{1/2} \frac{|H(I)-F_n(I)|}{\sqrt{F_n(I)(1-F_n(I))}} \\
 \geq & \
  n^{1/2} \frac{|H(I)-F(I)|-\Bl(\sqrt{2\log ({e}/{p_n})}+b_n^{1/2}\Br)
  \sqrt{{p_n(1-p_n)}/{n}}}{\sqrt{p_n(1-p_n)}\bigl(1+2({\log n})^{-1/2}\bigr)} \\
 \geq &\ n^{1/2} \frac{|H(I)-p_n|}{\sqrt{p_n(1-p_n)}} \Bl(1-\frac{2}{
\sqrt{\log n}} \Br) - \Bl(\sqrt{2 \log ({e}/{p_n})}+b_n^{1/2} \Br)\\
 \geq& \ c_n +\frac{1}{2} B_n\ \ \ \mbox{ by (\ref{A0II})},
\end{aligned}
\end{equation}
since $\sqrt{\log e/p_n} \leq \sqrt{\log n}$. As in the proof of the first part, one finds
$c_n \geq \tilde{c}_n +b_n/4$ eventually. Moreover, on $\BB_n$ we have
$F_n(I) \geq F(I)(1-2({\log n})^{-1/2}) \geq  \log^2(n)/(2n)$, hence
$$
\frac{\tilde{c}_n^2}{\sqrt{nF_n(I)(1-F_n(I))}}\ \leq \ \frac{6 \log n+b_n}{
{2}^{-1} \log n} \ \leq \ 13\ \ \ \text{eventually},
$$
and so (\ref{Astar}) is no smaller than
$$
\tilde{c}_n +\frac{\tilde{c}_n^2}{2 \sqrt{nF_n(I)(1-F_n(I))}}
\ \ind(F_n(I) < H(I))\ \ \ \text{eventually},
$$
completing the proof.
\end{proof}

\medskip

\begin{proof}[of Proposition~\ref{thmC2}]
Let $F$ have density 
$$
f(x)=\frac{3}{2} \ind \bigl(x \in [0,\frac{1}{2})\bigr) +\frac{1}{2}\ind\bigl(x \in
[\frac{1}{2},1]\bigr).
$$ If $k_n$ is odd, then the $(k_n+1)/2$th bin is $\bigl[{1}/{2} -
{1}/{(2k_n)}, {1}/{2}+{1}/{(2k_n)}\bigr)$. Denote the height of the histogram (i.e.~the slope of $H_n$) on that bin by $h$. Set
$$
I =
\begin{cases}
 \Bl[{1}/{2} -{1}/{(6 k_n)},{1}/{2} \Br) &  \text{if $h \leq {3}/{4}$},\\
 \Bl[{1}/{2}, {1}/{2} +{1}/{(2 k_n)} \Br) & \text{if $h > {3}/{4}$}.
\end{cases}
$$
Then $p_n=F(I)={1}/{(4k_n)}$ and $|F(I)-H_n(I)| \geq {1}/{(8k_n)} = {p_n}/{2}$,
hence
$$
n^{1/2}\ d_{p_n} (F,H_n) \ \geq \ n^{1/2} \frac{|F(I)-H_n(I)|}{\sqrt{F(I)(1-F(I))}}
               \ \geq \ \frac{1}{2} \sqrt{n p_n}. 
$$
$ $ \end{proof}

\begin{proof}[of Theorem~\ref{thmFeatureInfer}]
By construction of $C_n(\alp)$ we have $\Pr(F \in C_n(\alp)) \geq 1-\alp$, {(which also holds true for general $F$, if we set $\kappa_n^*(\alpha)$ as the threshold, see Lemma~\ref{le:ks}).} Hence
Lemma~\ref{quadapprox}~\eqref{eq:quadProxC} gives
$$
\Pr \Bl(\bigl|{\bar{f}(I) -{\abs{I}}^{-1}{F_n(I)}}\bigr| \le \frac{1}{2} 
r_n(I) \text{ for all } I \in \JJ \Br) \ge 1-\alpha.
$$
Furthermore, Lemma~\ref{quadapprox}~\eqref{eq:quadProxC} yields
$\bigl|{\bar{h}(I) -{\abs{I}}^{-1}{F_n(I)}}\bigr| \le  
r_n(I)/2$ for all $I \in \JJ$ and every $H \in \tilde{C}_n(\alp) $ 
whose density $h$ is constant on $I$. 
\end{proof}

\begin{proof}[of Theorem~\ref{thmD}]
Let $\eps_n \in (0,1)$ and $p_n \in
\bigl(2 n^{-1}{\log^2 n}, {1}/{2} \bigr)$ be arbitrary sequences with $\eps_n
\sqrt{\log e/p_n} \ra \infty$. In particular, $p_n \ra 0$. 
We proceed as in the proof of Theorem~\ref{thmB} and will construct $m_n=\lfloor
{1}/{2p_n} \rfloor ^2$ densities $f_{njk} \in \II_n(1-\eps_n)$ such that
\be  \label{Dst}
\lim_{n \ra \infty} \Ex \Bl| \frac{1}{m_n} \sum_{j,k=1}^{m_n^{1/2}} L_{njk}-1\Br|\ =\ 0,
\ee
where $L_{njk}=\prod_{i=1}^n f_{njk}(X_i)$ with $X_i$ independent $U(0,1)$. Unfortunately,
the truncation argument used in the proof of Theorem~\ref{thmB} will not go through
as the covariances of the $L_{njk}$ are not small enough. Instead, we follow an idea
in the proof of Theorem~3.1(a) in \cite{DueSpo01} and write $f_{njk}(x)=g_{nj}(x) h_{nk}(x)$,
where for $j,k=1,\ldots,m_n^{1/2}$:
\begin{align*}
g_{nj}(x) &=1_{[0,1]}(x)-c_n1_{I_{nj}}(x)+a_n 1_{[0,{1}/{2})\setminus I_{nj}}(x),\\
h_{nk}(x) &=1_{[0,1]}(x)+c_n1_{J_{nk}}(x)-a_n 1_{[{1}/{2},1)\setminus J_{nk}}(x),
\end{align*}
and $c_n=(1-\eps_n) \{{(1-p_n)}/{np_n}\}^{1/2}
\{2 \log ({2e}/{p_n})\}^{1/2}$, $a_n=c_n p_n/({1}/{2}-p_n)$, 
$I_{nj}=[(j-1)p_n,jp_n)$, $J_{nk}=[{1}/{2} +(k-1)p_n, {1}/{2} +kp_n)$.
One readily checks that $f_{njk},g_{nj}$ and $h_{nk}$ are densities. 
For each pair $(j,k)$:
\begin{equation*}
\begin{split}
\bar{f}_{njk}(J_{nk}) - \bar{f}_{njk}(I_{nj}) \ & =\ 2c_n \ =\ 2(1-\eps_n) 
\left({\frac{p_n(1-p_n)}{n}}\right)^{1/2} \frac{\sqrt{2 \log ({2e}/{p_n})}}{p_n} \\
& \geq (1-\eps_n) \Bl [ \left\{\frac{F_{njk}(J_{nk})(1-F_{njk}(J_{nk}))}{n}\right\}^{1/2}
\frac{\bigl\{2 \log \bigl({e}/{F_{njk}(J_{nk})}\bigr)\bigr\}^{1/2}}{|J_{nk}|} \\
& \ +\left\{\frac{F_{njk}(I_{nj})(1-F_{njk}(I_{nj}))}{n}\right\}^{1/2}
\frac{\bigl\{2 \log \bigl({e}/{F_{njk}(I_{nj})}\bigr)\bigr\}^{1/2}}{|I_{nj}|} \Br],
\end{split}
\end{equation*}
since $F_{njk}(J_{nk})=(1+c_n)p_n$, $F_{njk}(I_{nj})=(1-c_n)p_n \geq p_n/2$, and the
function $g(t)=\sqrt{t(1-t)}$ is concave for $t \in (0,1)$.
Thus $f_{njik} \in \II_n(1-\eps_n)$. (While $F_{njk}(J_{nk})$ is larger than $p_n$,
one can easily bound it by, say, $2p_n$, and changing the definition of $\II_n(c)$ to this
end will not affect the conclusions of the theorem.) Moreover, $p_n \geq
2\log^2(n) /n$ implies $np_n \geq 2\bigl(\log (e/p_n)\bigr)^2$, hence $\Delta_{\infty} =
\sup_{x \in [0,1]} |g_{nj}(x)-1| =c_n \leq \sqrt{2 \log (2e/p_n)/(np_n)}
\leq \bigl(\log (e/p_n)\bigr)^{-1/2} \leq (\frac{1}{2}\log m_n)^{-1/2}$, while
$\Delta_2^2 = \int_0^1 (g_{nj}-1)^2 =p_n c_n^2 +({1}/{2}-p_n)a_n^2 \leq
2(1+2p_n)(1-\eps_n)^2  \log (2e/p_n)/n \leq 2(1-\eps_n/2) \log (2e/p_n)/n$ since
$p_n \leq \eps_n/4$. Therefore we obtain as in the proof of Theorem~\ref{thmB}:
\be  \label{Dstst}
\lim_{n \ra \infty} \Ex \Bl| \frac{1}{m_n^{1/2}} \sum_{j=1}^{m_n^{1/2}} L_{g_{nj}}-1\Br|\ =\ 0,
\ee
where $L_{g_{nj}}=\prod_{i=1}^n g_{nj}(X_i)$, and the same result holds for 
$L_{h_{nk}}$. Conditional on $N=\#\{i:X_i <{1}/{2}\}$,
$\sum_j L_{g_{nj}}$ and $\sum_k L_{h_{nk}}$ are independent and the $\{X_i:X_i <{1}/{2}\}$
are independent $U(0,{1}/{2})$, hence (\ref{Dstst}) gives
$$
\Ex \Bl( \Bl| \frac{1}{m_n^{1/2}} \sum_{j=1}^{m_n^{1/2}} L_{g_{nj}} -1\Br| \mid N\Br)
\ \ind \Bl(N \in \bigl(\frac{n}{4},\frac{3n}{4}\bigr)\Br) \ \ra\ 0\ \ \ \mbox{a.s.},
$$
and likewise for the average of the $L_{h_{nk}}$. Thus
\be  \label{Dststst}
\begin{split}
&\Ex \Bl| \frac{1}{m_n} \sum_{j,k=1}^{m_n^{1/2}} (L_{g_{nj}} -1)(L_{h_{nk}}-1) \Br|  \\
&\leq \Ex \left( \Ex \Bl( \Bl| \frac{1}{m_n^{1/2}} \sum_{j=1}^{m_n^{1/2}} L_{g_{nj}} -1\Br| 
  \mid N\Br) \Ex \Bl( \Bl| \frac{1}{m_n^{1/2}} \sum_{k=1}^{m_n^{1/2}} L_{h_{nk}} -1\Br|
  \mid N\Br) \ind\Bl(N \in \bigl(\frac{n}{4},\frac{3n}{4}\bigr)\Br) \right) \\
  & \hspace{2cm} + 2^2 \Pr \Bl( N 
  \not\in \bigl(\frac{n}{4},\frac{3n}{4}\bigr)\Br)  \\
& \ra 0 \,,
\end{split}
\ee
by bounded convergence, since $\Ex \bigl(\big|L_{g_{nj}} -1\big| \mid N\bigr) \leq \Ex 
\bigl(L_{g_{nj}}\mid N\bigr) +1=2$. Hence
$$
\Ex \Bl| \frac{1}{m_n} \sum_{j,k=1}^{m_n^{1/2}} L_{njk} -1 \Br| =
\Ex \Bl| \frac{1}{m_n} \sum_{j,k=1}^{m_n^{1/2}} \Bl\{(L_{g_{nj}} -1)(L_{h_{nk}}-1)
+ (L_{g_{nj}} -1)+(L_{h_{nk}}-1) \Br\} \Br|
$$
converges to zero by (\ref{Dstst}) and (\ref{Dststst}), proving (\ref{Dst}).
\end{proof}

\begin{proof}[of Theorem~\ref{thmD2}]
For $k \in \{0,1\}$, let $F \in \II_n(1+2k+\eps_n)$,
so there exist $I_1 <I_2$ for which (\ref{Ist}) holds with $c=1+2k+\eps_n$
and $F(I_i) \geq {\log ^2 (n)}/{n}$, $i=1,2$. It is readily checked that
these lower bounds on $F(I_i)$ imply on the event
$$
\AA_n=\Bl\{  n^{1/2} \frac{|F_n(I_i)-F(I_i)|}{\sqrt{F(I_i)
(1-F(I_i))}} \leq \frac{\eps_n}{2} \sqrt{2 \log \frac{e}{F(I_i)}}
\text{ for } i=1,2\Br\},
$$
it holds, for large $n$ enough, 
\be  \label{Ihel}
(1+\frac{\eps_n}{6}) \left\{\frac{F(I_i)(1-F(I_i))}{n}\right\}^{1/2} 
\frac{\sqrt{2 \log ({e}/{F(I_i)})}}{|I_i|}
\ \geq \ \frac{1}{2} r_n(I_i)\quad  i=1,2, 
\ee
where $r_n(I)$ is defined in Theorem~\ref{thmFeatureInfer}. 
{(For general $F$, we replace $\kappa_n(\alpha)$ by $\kappa_n^*(\alpha)$)}.

Let $H \in \tilde{C}_n(\alpha)$ with $h$ being constant on $I_1$ and $I_2$,
and suppose
\be  \label{Iass}
\bar{h}(I_1)-\bar{h}(I_2)\ \geq \ -k \Bl(r_n(I_1)+r_n(I_2)\Br)
\ee
holds. Then on $\AA_n$:
\be \nn
\begin{split}
& \frac{H(I_1)-F_n(I_1)}{|I_1|} + \frac{F_n(I_2)-H(I_2)}{|I_2|} \\
& \geq \bar{h}(I_1) -\bar{h}(I_2) +\bar{f}(I_2) -\bar{f}(I_1) -\frac{\eps_n}{2}
\sum_{i=1}^2 \left\{\frac{F(I_i)(1-F(I_i))}{n}\right\}^{1/2} \frac{\sqrt{2 \log \bigl({e}{F(I_i)}\bigr)}}{|I_i|} \\
& > -k(r_n(I_1)+r_n(I_2)) + (1+2k+\frac{\eps_n}{2}) \sum_{i=1}^2 
  \left\{\frac{F(I_i)(1-F(I_i))}{n}\right\}^{1/2} \frac{\sqrt{2 \log \bigl({e}/{F(I_i)}\bigr)}}{|I_i|} \\
  & \qquad \qquad  \mbox{ by \eqref{Iass} and \eqref{Ist}} \\
& \geq \frac{1}{2} (r_n(I_1)+r_n(I_2)) \ \ \text{eventually}\ \ \mbox{ by (\ref{Ihel}),}
\end{split}
\ee
which gives the contradiction  $h \not\in \tilde{C}_n(\alpha)$ by Lemma~\ref{quadapprox}~\eqref{eq:quadProxC}.
Therefore (\ref{Iass}) can only hold on $\AA_n^c$.
(As in the proof of Theorem~\ref{thmA} we assumed $\JJ=\{$ all real intervals $\}$ and 
refer to \cite{RivWal13} for the technical work that can be used to show
 that the conclusion also obtains
with the approximating set used in \S\ref{confset}.)

If $h$ is nonincreasing on conv$(I_1 \cup I_2)$, then (\ref{Iass}) holds with $k=0$
and so for $F \in \II_n(1+\eps_n)$:
\be \nn
\begin{split}
\Pr_F \Bl( & \mbox{there exists $H \in \tilde{C}_n(\alp)$ whose density is 
  constant on $I_1$} \\
  & \hspace{2cm} \mbox{and $I_2$ and nonincreasing on conv$(I_1 \cup I_2)$} \Br) \\
& \ \leq \Pr_F(\AA_n^c ) \ \leq \ \frac{8}{2\eps_n^2  \log ({e}/{p_n})} \ \ra 0 \ \ \mbox{ uniformly in $F$}\,,
\end{split}
\ee
by Chebychev's inequality, proving the first part. The second part follows analogously since
(\ref{signpattern}) together with $\bar{f}(I_2) \leq \bar{f}(I_1)$ implies that (\ref{Iass})
holds with $k=1$. 
\end{proof}

\begin{proof}[of Theorem~\ref{thmHistDensity}]
It follows immediately from Proposition~\ref{propHistDensity} below and the fact that  $\kappa_n(\alpha) \le \sqrt{2\log(C/\alpha)}$ for some constant $C$, see~\cite{RivWal13}.
\end{proof}

\begin{proposition}\label{propHistDensity}
Under the same notation as Theorem \ref{thmHistDensity},  it controls overestimating the number of bins uniformly over all $F$'s
\[
 \Pr_F \Bl(\nbin (H) > \nbin(F)\Br) \le \alpha;
\]
and it controls underestimating the number of bins, for  $n \ge16\log n / (\lambda\underline\theta)$,
\[
\Pr_F \Bl(\nbin(H) < \nbin(F)\Br) \le 4K\left(2\exp\left(-\frac{1}{128}n\lambda^2\underline\theta^2\right) + \exp\left(-\frac{1}{72}n\lambda^2\left(\frac{\Delta}{2}-\frac{12\delta_n}{\lambda n^{1/2}}\right)_+^2\right)\right),
\]
with $\delta_n = \sqrt{2\log {(256 e)}/{(\lambda^2\underline\theta^2)}} + \kappa_n(\alpha);$
Moreover, it controls the number of modes and troughs, for $n  \ge 32\log n / (\lambda\underline\theta)$, with $\tilde{\delta}_n = \sqrt{2\log {(1024 e)}{(\lambda^2\underline\theta^2)}} + \kappa_n(\alpha)$,
\begin{multline*}
\Pr_F \Bl(h \text{ and }f \text{ have the same number of modes and troughs}\Br) \\
\ge 1 - \alpha -  12K \left(2\exp\left(-\frac{1}{512}n\lambda^2\underline\theta^2\right) + \exp\left(-\frac{1}{288}n\lambda^2\left(\frac{\Delta}{2}-\frac{24\tilde{\delta}_n}{\lambda n^{1/2}}\right)_+^2\right)\right).
\end{multline*}
\end{proposition}

\begin{remark} Note that the assertions in Proposition~\ref{propHistDensity} also hold for sequences of $f = f_n$ with $\underline\theta = \underline\theta_n$, $\lambda = \lambda_n$, $\Delta = \Delta_n$, $K =K_n$, and $\alpha = \alpha_n$.
\end{remark}

\begin{proof}[of Proposition \ref{propHistDensity}]
For the first part, by the definition of the essential histogram in~\eqref{eq:adj_ess_hist}, we have
\[
\Pr_F\Bl(\nbin(H) \le \nbin(F)\Br) \ge \Pr_F\Bl(F \in \tilde{C}_n(\alpha)\Br) \ge 1-\alpha.
\]
For the second and the third parts, we {use arguments similar to}~\citet[Theorem 7.10]{FrMuSi14}, but with notable differences due to the use of the reduced system $\JJ$. We will frequently use the following inequality, which comes as an application of the Hoeffding's inequality, 
\begin{equation}\label{eqEmpF}
\Pr_F\Bl(\abs{F(I) - F_n(I)} \ge x \Br) \le 2\exp(-2n x^2).
\end{equation}
The detail is as follows: For the second part, let $m_k$ be the mid-point of $I_k$, $I_k^{-} = I_k \cap (-\infty, m_k]$, and $I_k^{+} = I_k \cap (m_k, \infty)$. For a fixed $k$, we have
\begin{align*}
&\Pr_F\Bl(h \text{ is constant on } (m_{k-1}, m_k]\Br) \\
 = \ &\Pr_F\Bl(h \equiv c \ge \frac{c_{k-1}+c_k}{2} \text{ on } (m_{k-1}, m_k]\text{ for some constant } c \Br) \\
& {}\qquad + \Pr_F\Bl(h \equiv c < \frac{c_{k-1}+c_k}{2} \text{ on } (m_{k-1}, m_k]\text{ for some constant } c \Br)  \\
 \le\ & \Pr_F\Bl(\abs{\bar{h}(I) - \bar{f}(I)} \ge \frac{\Delta}{2},\, h \text{ is constant on } I \equiv I_{k-1}^+\Br)  \\
& {}\qquad + \Pr_F\Bl(\abs{\bar{h}(I) - \bar{f}(I)} \ge \frac{\Delta}{2},\, h \text{ is constant on } I \equiv I_{k}^-\Br). 
\end{align*}
By symmetry we only need to consider the first term in the r.h.s.~of the above equation, where $I = I_{k-1}^+$. By the construction of $\JJ$ in~\eqref{eqJJ}, it holds that for any $I$ with $F_n(I) \ge 6\log (n) / n$ there is an interval $J \subseteq I$ and $J \in \JJ$ such that $F_n(J) \ge F_n(I)/3$. Conditioned on $\abs{F_n (I) - F(I)} \le  \lambda \underline\theta/16$ and $\abs{F_n (J) - F(J)} \le  \lambda \underline\theta/16$, we have for $n \ge 16\log n / (\lambda\underline\theta)$
\[
F(J) \ \ge \ F_n(J) - \frac{1}{16}\lambda\underline\theta \ \ge \ \frac{1}{3}F_n(I) - \frac{1}{16}\lambda\underline\theta \ \ge \ \frac{1}{3}F(I) - \frac{1}{12}\lambda\underline\theta \ \ge \ \frac{1}{6}F(I),  
\]
which implies $\abs{J} \ge \abs{I}/6 \ge \lambda/12$. Thus,  for $n \ge 16\log n/(\lambda\underline\theta)$
\begin{align*}
&\Pr_F\Bl(\abs{\bar{h}(I) - \bar{f}(I)} \ge \frac{\Delta}{2},\, h \text{ is constant on } I \equiv I_{k-1}^+\Br) \\
\le & \,\Pr_F\Bl(\abs{F_n (I) - F(I)} \ge \frac{1}{16} \lambda \underline\theta\Br) + \Pr_F\Bl(\abs{F_n (J) - F(J)} \ge \frac{1}{16} \lambda \underline\theta\Br) \\
& {}\qquad+ \Pr_F\Bl(\abs{\bar{h}(J) - \bar{f}(J)} \ge \frac{\Delta}{2},\, h\text{ is constant on }J, \, \abs{J} \ge \frac{1}{12}\lambda, \abs{F_n (J) - F(J)} \le \frac{1}{16} \lambda \underline\theta\Br) \\
\le & \, 2\exp\left(-\frac{1}{128}n\lambda^2\underline\theta^2\right) + 2\exp\left(-\frac{1}{128}n\lambda^2\underline\theta^2\right)  \\
& {}\qquad+  \Pr_F\Bl(\abs{\bar{h}(J) - \bar{f}(J)} \ge \frac{\Delta}{2},\, h\text{ is constant on }J, \, \abs{J} \ge \frac{1}{12}\lambda,\, \abs{F_n (J) - F(J)} \le \frac{1}{16} \lambda \underline\theta\Br)  \\
& {}\qquad \qquad\text{by~\eqref{eqEmpF}} \\
\le &\, 4\exp\left(-\frac{1}{128}n\lambda^2\underline\theta^2\right)  + \Pr_F\Bl(\abs{\bar{h}(J) - \bar{f}(J)} \ge \frac{\Delta}{2} - \frac{12\delta_n}{\lambda n^{1/2}}\Br)\qquad\text{by Lemma~\ref{quadapprox}~\eqref{eq:quadProxA}} \\
\le & \, 4\exp\left(-\frac{1}{128}n\lambda^2\underline\theta^2\right) + 2 \exp\left(-\frac{1}{72}n\lambda^2\left(\frac{\Delta}{2}-\frac{12\delta_n}{\lambda n^{1/2}}\right)_+^2\right) \qquad\text{by~\eqref{eqEmpF}}.
\end{align*}
The same bound holds for $\Pr_F\bigl(\abs{\bar{h}(I) - \bar{f}(I)} \ge {\Delta}/{2},\, h \text{ is constant on } I \equiv I_{k}^-\bigr)$ due to symmetry. Therefore, we have  for $n \ge 16\log n/(\lambda\underline\theta)$
\begin{align*}
\Pr_F\Bl(\nbin(H) < \nbin (F)\Br) &\le \Pr_F\Bl(h \text{ is constant on } (m_{k-1}, m_k] \text{ for some } k\Br) \\
&\le 4K\left(2\exp\left(-\frac{1}{128}n\lambda^2\underline\theta^2\right) + \exp\left(-\frac{1}{72}n\lambda^2\left(\frac{\Delta}{2}-\frac{12\delta_n}{\lambda n^{1/2}}\right)_+^2\right)\right).
\end{align*}
For the third part, we further divide $I^+_k$ (or $I^-_k$) into two subintervals $I^+_{k,1}$, $I^+_{k,2}$ (or $I^-_{k,1}$, $I^-_{k,2}$) of equal lengths. For any fixed $k$, it holds that 
\be \nn
\begin{split}
\Pr_F\Bl( & h \text{ has exactly one jump, but has a different trend from } f \text{ on } (m_{k-1}, m_k]\Br)\\
 & \le \Pr_F\Bl(\abs{\bar{h}(I) - \bar{f}(I)} \ge \frac{\Delta}{2}, \text{ and } h \text{ is constant on } I \equiv I^+_{k,1}  \Br)\\ 
&+ \Pr_F\Bl(\abs{\bar{h}(I) - \bar{f}(I)} \ge \frac{\Delta}{2}, \text{ and } h \text{ is constant on } I \equiv I^+_{k,2}  \Br)\\
& + \Pr_F\Bl(\abs{\bar{h}(I) - \bar{f}(I)} \ge \frac{\Delta}{2}, \text{ and } h \text{ is constant on } I \equiv I^-_{k,1}  \Br) \\
&+ \Pr_F\Bl(\abs{\bar{h}(I) - \bar{f}(I)} \ge \frac{\Delta}{2}, \text{ and } h \text{ is constant on } I \equiv I^-_{k,2}  \Br).\\ 
\end{split}
\ee

Each term above can be bounded in a similar way as in the second part, which leads to
\begin{multline*}
\Pr_F\Bl(h \text{ has exactly one jump, but has a different trend from } f \text{ on } (m_{k-1}, m_k]\Br)\\ \le
16\exp\left(-\frac{1}{512}n\lambda^2\underline\theta^2\right) + 8 \exp\left(-\frac{1}{288}n\lambda^2\left(\frac{\Delta}{2}-\frac{24\tilde{\delta}_n}{\lambda n^{1/2}}\right)_+^2\right)\qquad\text{ for } n \ge \frac{32\log n}{\lambda\underline\theta}.
\end{multline*} 
It follows from (i) and (ii) that for  for $n \ge 16\log n / (\lambda\underline\theta)$
\begin{multline*}
\Pr_F \Bl( h \text{ has exactly one jump in each }(m_{k-1}, m_k]\Br) \ge 1-\alpha \\
- 4K\left(2\exp\left(-\frac{1}{128}n\lambda^2\underline\theta^2\right) + \exp\left(-\frac{1}{72}n\lambda^2\left(\frac{\Delta}{2}-\frac{12\delta_n}{\lambda n^{1/2}}\right)_+^2\right)\right).
\end{multline*}
Thus, for $n \ge 32\log n / (\lambda\underline\theta)$
\begin{align*}
&\, \Pr_F \Bl(h \text{ and }f \text{ have the same number of modes and troughs}\Br) \\
\ge & \, \Pr_F \Bl( h \text{ has exactly one jump, and has the same trend as $f$, on each }(m_{k-1}, m_k]\Br) \\
\ge & \,  1 - \alpha -  12K \left(2\exp\left(-\frac{1}{512}n\lambda^2\underline\theta^2\right) + \exp\left(-\frac{1}{288}n\lambda^2\left(\frac{\Delta}{2}-\frac{24\tilde{\delta}_n}{\lambda n^{1/2}}\right)_+^2\right)\right).
\end{align*}
$ $ \end{proof}

\begin{proof}[of Theorem~\ref{thmB}]
Using the probability integral transformation we may assume
$F=U[0,1]$. For $j=1,\ldots,m_n=\lfloor 1/p_n \rfloor$ define the densities 
$f_{nj}(x)=(1+c_n) \ind_{I_{nj}}(x) + (1-a_n)\ind_{[0,1]\setminus I_{nj}}(x)$, where
$I_{nj}=[(j-1)p_n, jp_n)$, $a_n=c_n p_n/(1-p_n)$ and 
$c_n= \sqrt{{(1-p_n)}/{np_n}} (1-\eps_n) \sqrt{2 \log (e/p_n)}$. 
Then $d_{p_n}(F,F_{nj}) = |F_{nj}(I_{nj})-F(I_{nj})|
( p_n(1-p_n))^{-1/2} = (1-\eps_n) \sqrt{2 \log (e/p_n)/n}$. 
The claim of the theorem will follow as in the proof of Theorem~4.1(b)
in \cite{DueWal08} once we show that
\be  \label{6stst}
\lim_{n \ra \infty} \Ex_F \Bl| \frac{1}{m_n} \sum_{j=1}^{m_n} L_{nj} -1 \Br| \ =\ 0,
\ee
where $L_{nj}=\prod_{i=1}^n f_{nj}(X_i)$. Since the sets $\{f_{nj} \neq 1\}$
are not disjoint, Lemma~7.4 of \cite{DueWal08} is not applicable to prove (\ref{6stst})
and we have to account for the covariances of the $L_{nj}$. For $i \neq j$ we obtain
$\int_0^1 f_{ni}(x) f_{nj}(x) dx =1 - \bigl({p_nc_n}/{(1-p_n)}\bigr)^2$,
hence $\Ex_F(L_{ni}L_{nj}) =\left\{1 - \bigl({p_nc_n}/{(1-p_n)}\bigr)^2\right\}^n < 1$.
Using $\Ex_F L_{nj}=1$ and H\"{o}lder's inequality gives for any $\eps >0$
\begin{align*}
& \Ex_F \Bl| \frac{1}{m_n} \sum_{j=1}^{m_n} L_{nj} -1 \Br| \\
& \leq \left\{ {\rm var}_F \Bl( \frac{1}{m_n} \sum_{j=1}^{m_n} L_{nj} \ind(L_{nj} \leq \eps m_n)\Br)\right\}^{1/2}
  +\frac{2}{m_n} \sum_{j=1}^{m_n} \Ex_F\bigl( L_{nj}\ind(L_{nj} > \eps m_n)\bigr) \\
& \leq \left\{\frac{1}{m_n} \Ex_F \Bl(L_{n1}^2 \ind(L_{n1} \leq \eps m_n)\Br)
  + \frac{1}{m_n^2} \sum_{i<j} {\rm cov} \Bl(L_{ni}\ind(L_{ni} \leq \eps m_n), L_{nj}\ind(L_{nj} 
  \leq \eps m_n)\Br)\right\}^{1/2}\\
&{}\qquad \ \ + 2\, \Ex_F\bigl( L_{n1} \ind(L_{n1} > \eps m_n)\bigr)\\
& \leq \left\{\frac{1}{m_n} \Ex_F (\eps m_n L_{n1}) + \Ex_F L_{n1}L_{n2}
  -\Bl( \Ex_F L_{n1} \ind(L_{n1} \leq \eps m_n)\Br)^2\right\}^{1/2} + 2\,\Ex_F \bigl(L_{n1} \ind(L_{n1} > \eps m_n)\bigr)\\
& \leq \left\{\eps +1 -\Bl(1-\Ex_F L_{n1} \ind(L_{n1} > \eps m_n)\Br)^2\right\}^{1/2} 
  + 2\, \Ex_F L_{n1} \ind(L_{n1} > \eps m_n)\,.
\end{align*}

Thus (\ref{6stst}) follows by showing 
\be  \label{7ststst}
\lim_{n \ra \infty} \Ex_F\bigl( L_{n1} \ind(L_{n1} > \eps m_n)\bigr)\ =\ 0\,.
\ee

To this end, observe that $p_n \geq
\log^2n /n$ implies $np_n \geq (\log e/p_n)^2$, hence 
$$
\Delta_{\infty}=\sup_{x \in [0,1]} | f_{nj}(x)-1| = c_n \le \left(\frac{2\log(e/p_n)}{n p_n}\right)^{1/2}
\le  \left(\frac{2}{\log (e/p_n)}\right)^{1/2} \leq  \left(\frac{2}{\log m_n}\right)^{1/2}\,.
$$
Further, $\Delta_2^2 = \int_0^1 (f_{nj}(x)-1)^2dx =p_n c_n^2 +(1-p_n)a_n^2=
2(1-\eps_n)^2  \log (e/p_n)/n$, hence
\begin{equation*}
\begin{split}
\sqrt{\log m_n} \Bl( 1-\frac{n \Delta_2^2}{2 \log m_n}\Bl) \ & \geq \
\sqrt{\log m_n} \Bl(1-(1-\eps_n)\frac{\log (e/p_n)}{\log \lfloor 1/p_n \rfloor}\Br) \\
& \geq \ \sqrt{\log m_n} \Bl(1-(1-\eps_n)(1-\frac{3}{\log p_n})\Br) \\
& =\ \sqrt{\log m_n} \eps_n (1+o(1)) +o(1) \ \ra \ \infty\,,
\end{split}
\end{equation*}
by the assumption on $\eps_n$. It is shown in the proof of Lemma~7.4
in \cite{DueWal08}  that (\ref{7ststst}) follows from these properties of
$\Delta_{\infty}$ and $\Delta_2^2$. 
\end{proof}

\section{Computational details}\label{sp:cmp}

\subsection{{Computation of the threshold $\kappa_n(\alpha)$}}\label{ss:qnt}
{
Let $F$ be an arbitrary distribution function, i.e., right continuous with limits from left. We define its standard left (continuous) inverse as $F^{-1}(t) = \inf\{x : F(x) \ge t\}$.  Let $Z_1,\ldots, Z_n$ be independent and identically distributed uniform random variables on $[0,1]$, and denote their common distribution function by $U$.  Note that $X_i = F^{-1}(Z_i)$ are independent, and identically distributed according to $F$.  By $U_n$ and $F_n$ we denote the empirical distribution functions of $Z_1, \ldots, Z_n$ and $X_1,\ldots, X_n$, respectively. For an interval $I = (j, k]$, we define $Z_I = (Z_{(j)}, Z_{(k)}]$ and similarly $X_I = (X_{(j)}, X_{(k)}]$. Then, $T_n$ in~\eqref{eq:msStat} can be written as 
$$
T_n \ =\  \max_{I \in \JJ_0 \cap (\mathcal{D} \times \mathcal{D})} \Bl\{\sqrt{2 \lLR \bigl(F(X_I),F_n(X_I)\bigr)} -
\pen\bigl(F_n(X_I)\bigr)\Br\}
$$
where $\mathcal{D} =\{i\, :\, X_{(i)}\neq X_{(i+1)}\}$ as in~\eqref{eqJJ}, and $\JJ_0 $  is a collection of intervals
\begin{align*}
\JJ_0 &= \bigcup_{l =2}^{l_{max}} \JJ_0(l),\;\;\; \mbox{ where }
\;l_{max}=\Bl\lfloor \log_2 \tfrac{n}{\log n}\Br\rfloor \;\;
\mbox{ and}\\
\JJ_0(l) &= \Bl\{(j,k]: j,k \in \{1+i d_{l},
i\in \N_0 \}\ \mbox{ and } m_{l}<k-j\leq 2m_{l}\Br\},\\
&\;\;\;\;\mbox{ where } m_{l}=n2^{-l},\; d_{l}=
\Bl\lceil \tfrac{m_{l}}{6 {l}^{1/2}}\Br\rceil.
\end{align*}
}

{In the case that $F$ is continuous, it holds that $F(X_I) = U(Z_I)$ and $F_n(X_I) = U_n(Z_I)$ for any $I = (j, k]$. Then $T_n$ can be further written as
\begin{equation}\label{eq:qntC}
T_n \ =\  \max_{I \in \JJ_0} \Bl\{\sqrt{2 \lLR \bigl(U(Z_I),U_n(Z_I)\bigr)} -
\pen\bigl(U_n(Z_I)\bigr)\Br\},
\end{equation}
and is thus independent of $F$. This makes it possible to estimate the distribution, as well as the quantile $\kappa_n(\alpha)$, of $T_n$ via Monte Carlo simulations.}

Consider now that $F$ is a general distribution function. For $I = (j,k] \in \JJ_0 \cap (\mathcal{D} \times \mathcal{D})$, it still holds that $F_n(X_I) = U_n(Z_I) = (k-j)/n$, but $F(X_I) = U(Z_I)$ may no longer be valid. However, it is possible to find a universal distribution which stochastically dominates the distribution of $T_n$. More precisely, we introduce the notation $I_+ = (j, k+1]$ and $I_- = (j+1, k]$, for any $I = (j, k]$, and define 
\begin{multline}\label{eq:qntG}
T_n^*\ = \  \max_{I \in \JJ_0} \Bl\{\Bigl(2 \max\bigl\{\lLR \bigl(U(Z_{I_+}),U_n(Z_I)\bigr), \lLR \bigl(U(Z_{I_-}),U_n(Z_I)\bigr)\bigr\}\Bigr)^{1/2} \\
- \pen\bigl(U_n(Z_I)\bigr)\Br\},
\end{multline}
and $\kappa_n^*(\alpha)$ as its ($1-\alpha$)-quantile, which can be determined by Monte Carlo simulations.
\begin{lemma}\label{le:ks}
It holds that $\kappa_n(\alpha) \le \kappa_n^*(\alpha)$ and $\sup_n\kappa^*_n(\alpha) < \infty$, for every $\alpha \in(0,1)$.
\end{lemma}
\begin{proof}
We claim that $i \in \mathcal{D}$ implies $Z_{(i)} \le F(X_{(i)}) < Z_{(i+1)}$. In fact, since $F(F^{-1}(t)) \ge t$ for all $t$, we have $Z_{(i)} \le F(F^{-1}(Z_{(i)})) = F(X_{(i)})$. If $F(X_{(i)}) \ge Z_{(i+1)}$, then, by the definition of $F^{-1}$, we have $X_{(i)} \ge F^{-1}(Z_{(i+1)}) = X_{(i+1)}$, which contradicts with $i \in \mathcal{D}$. Thus, $F(X_{(i)}) < Z_{(i+1)}$.
\newline
\indent For an arbitrary interval $I = (j, k] \in \JJ_0 \cap (\mathcal{D} \times \mathcal{D})$, it then holds that 
$$
U(Z_{I_-})=Z_{(k)} - Z_{(j+1)}<F(X_I) = F(X_{(k)}) - F(X_{(j)}) < Z_{(k+1)} - Z_{(j)} = U(Z_{I_+}).
$$
As $\lLR(x, F_n(X_I))$ is decreasing for $x \le F_n(X_I)$ and increasing for $x \ge F_n(X_I)$, we have
$$
\lLR(F(X_I), F_n(X_I)) \le \max\left\{\lLR(U(Z_{I_+}), U_n(Z_I)), \, \lLR(U(Z_{I_-}), U_n(Z_I))\right\}.
$$
This proves the first part, i.e., $\kappa_n(\alpha) \le \kappa_n^*(\alpha)$. \\
\indent The second part, namely $\kappa_n^*(\alpha) = O(1)$, can be proven in a similar way as \citet[Theorem 1]{RivWal13}, by noticing that $|U_n(Z_{I_+}) -U_n(Z_I)| = |U_n(Z_{I_-}) -U_n(Z_I)| = {1}/{n}.$ 
\end{proof}

\begin{figure}[!h]
\centering
\begin{tabular}{lll}
(a) & (b) & (c)\\
\includegraphics[width=0.3\textwidth]{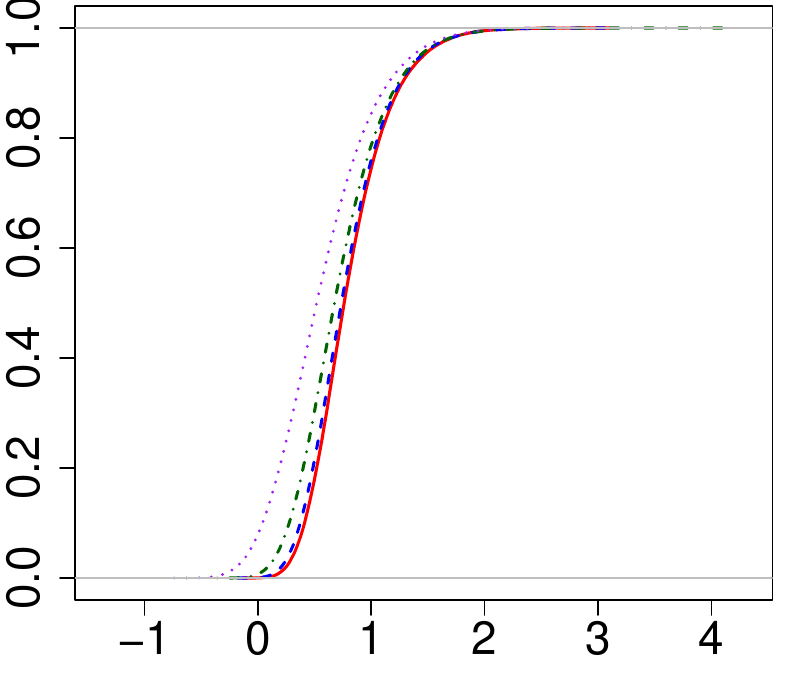} &
\includegraphics[width=0.3\textwidth]{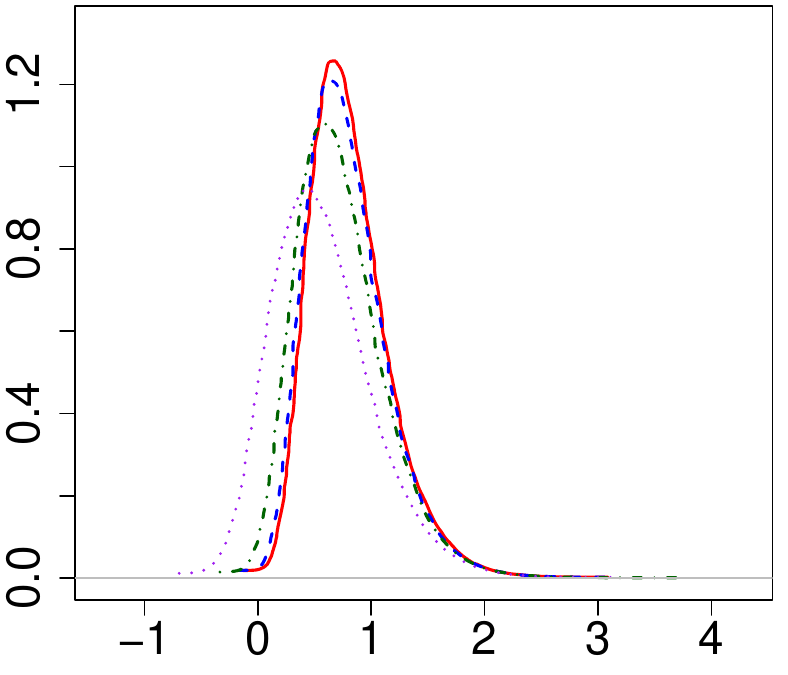} &
\includegraphics[width=0.3\textwidth]{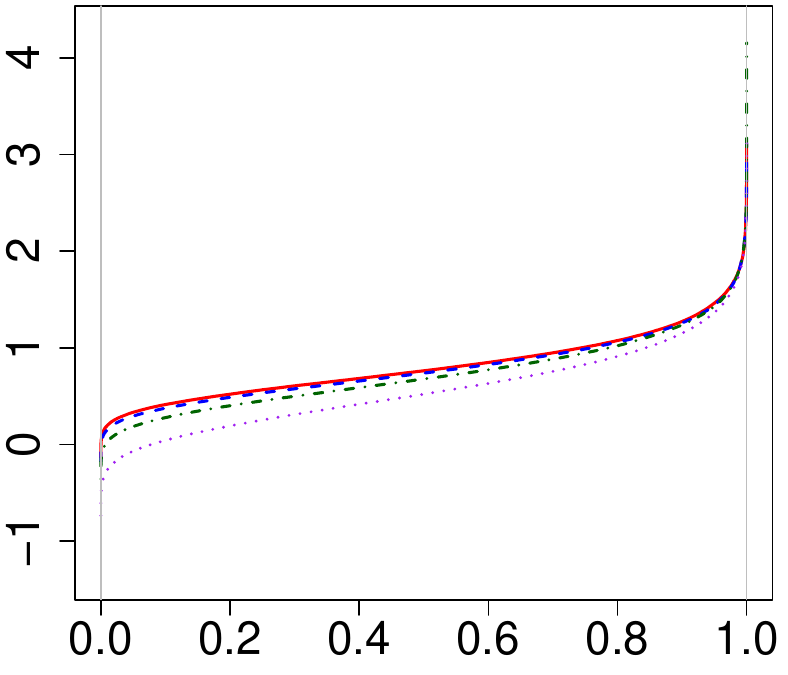} \\
 (d) & (e) & (f)\\
\includegraphics[width=0.3\textwidth]{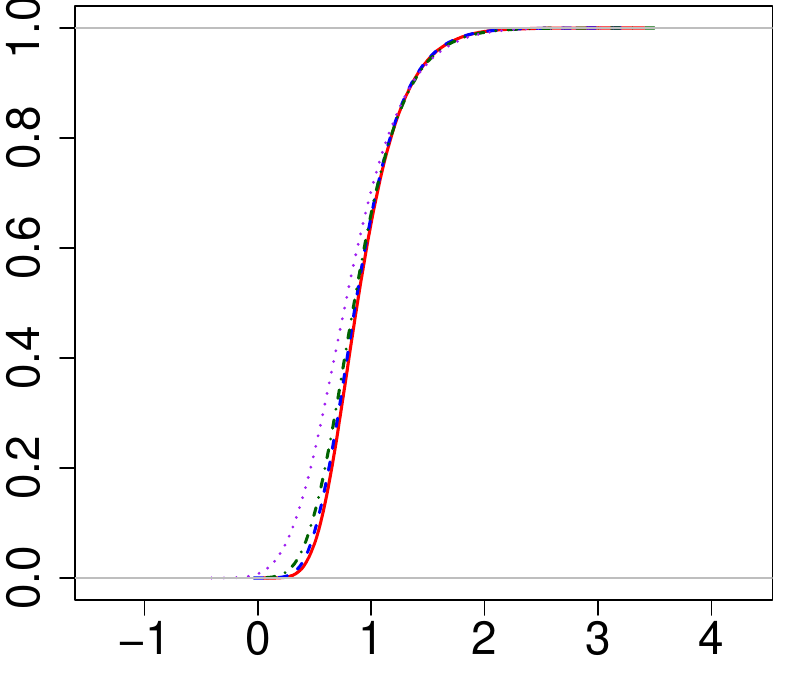} &
\includegraphics[width=0.3\textwidth]{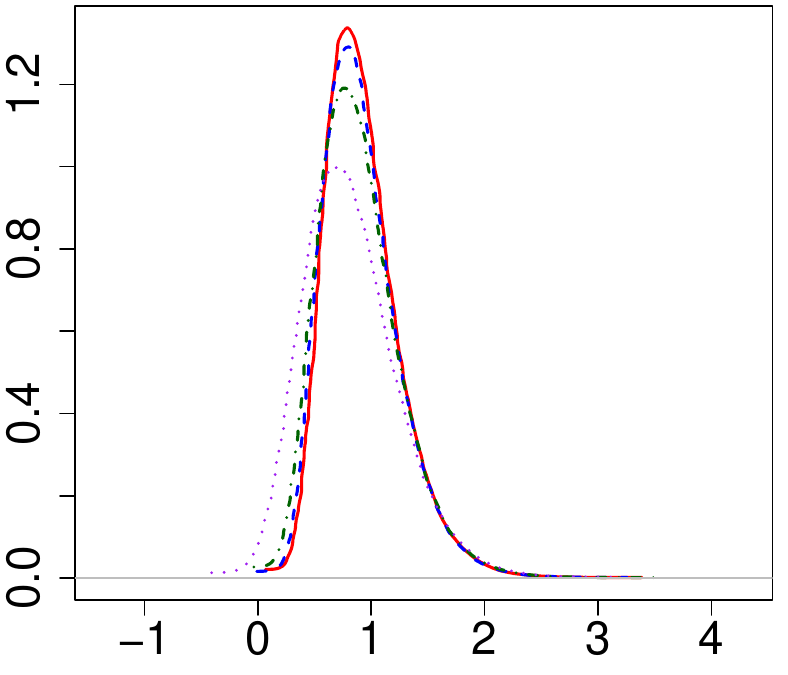} &
\includegraphics[width=0.3\textwidth]{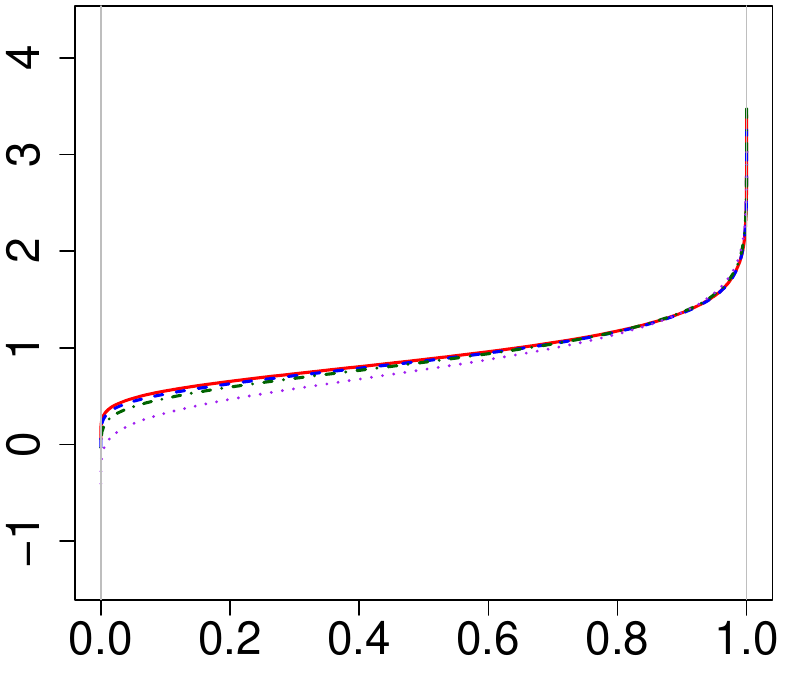} 
\end{tabular}
\caption{
{The (a) distribution function, (b) density and (c) quantile function of $T_n$ in~\eqref{eq:qntC} (equivalently, $T_n$ in \eqref{eq:msStat} under any continuous $F$); and the (d) distribution function, (e) density and (f) quantile function of $T_n^*$ in~\eqref{eq:qntG}; for sample size $n = 10^3$ (dotted), $n = 10^4$ (dot-dash), $n = 5\times 10 ^4$ (dashed) and $n = 10^5$ (solid). For each sample size $n$, the distribution function, density, and quantile function are estimated over $10^5$ random repetitions.}  \label{fig:mssim}}
\end{figure}

\begin{figure}[t]
\centering
\begin{tabular}{ll}
(a) & (b) \\
\includegraphics[width=0.45\textwidth]{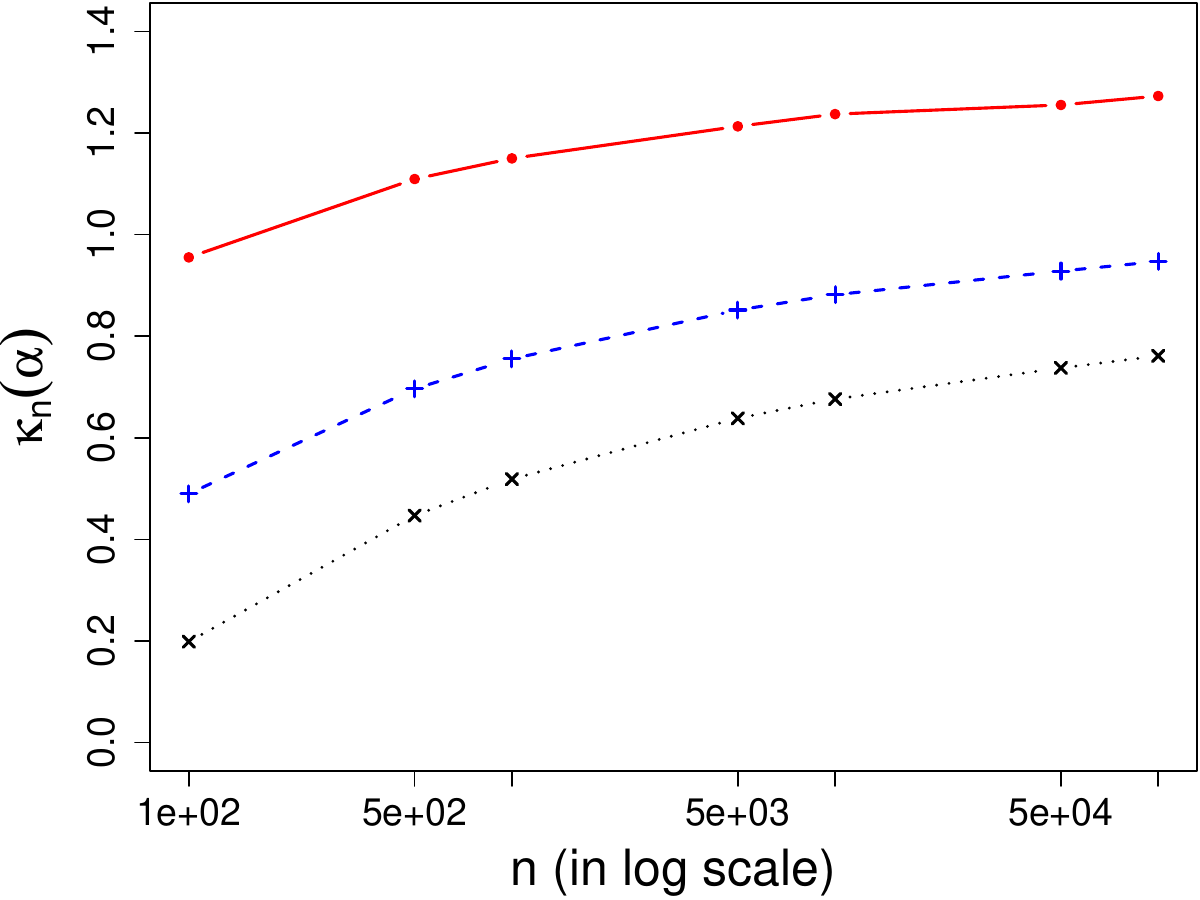} &
\includegraphics[width=0.45\textwidth]{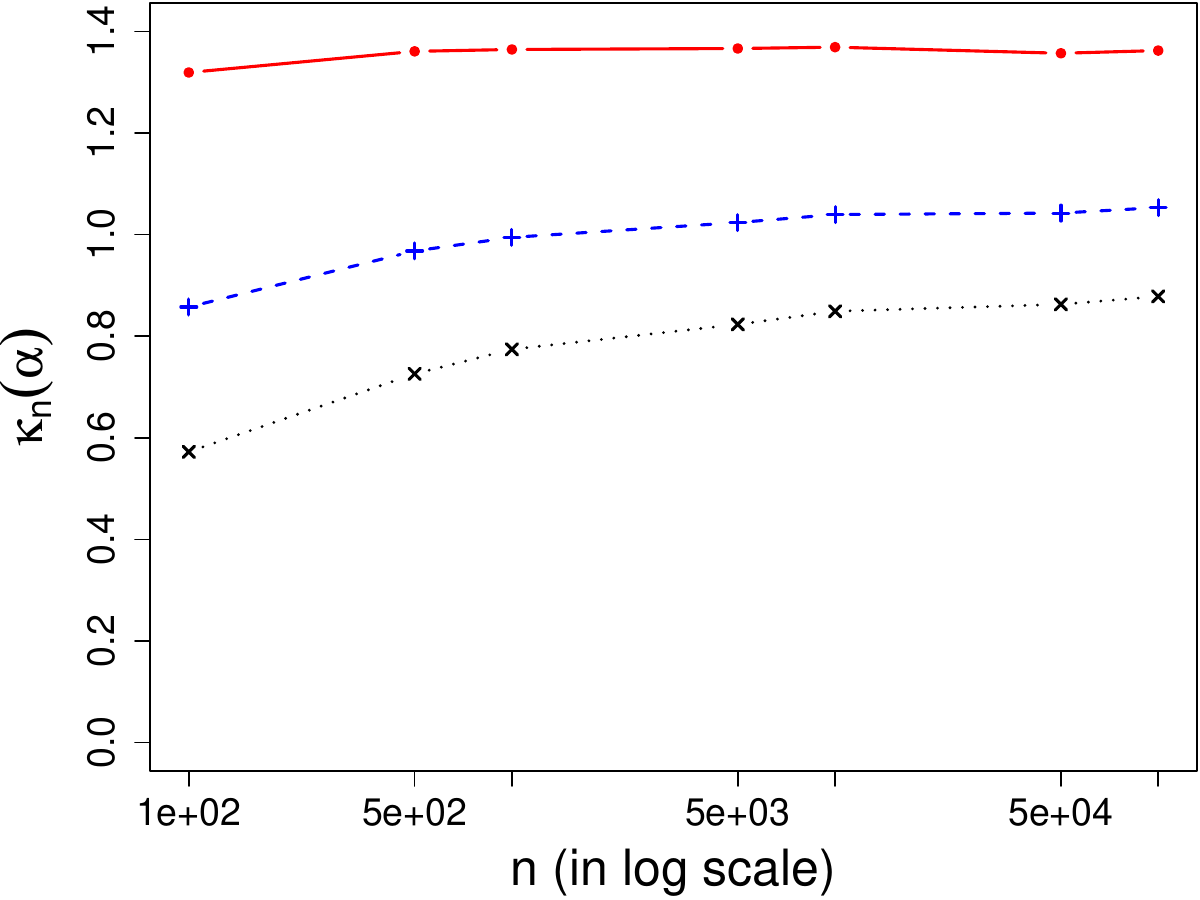} 
\end{tabular}
\caption{
{Quantiles: (a) $\kappa_n(\alpha)$ of $T_n$ in~\eqref{eq:qntC} (equivalently, $T_n$ in~\eqref{eq:msStat} under any continuous $F$); (b) $\kappa_n^*(\alpha)$ of $T_n^*$ in~\eqref{eq:qntG};  for $\alpha = 0.1$ (solid), $\alpha = 0.3$ (dashed) and $\alpha = 0.5$ (dotted).  For each sample size $n$, the quantiles are estimated over $10^5$ random repetitions.}\label{fig:qnt}}
\end{figure}

{Lemma~\ref{le:ks} ensures that all the theoretical results about the confidence set $C_n(\alpha)$ and the essential histogram remain valid if we use $\kappa_n^*(\alpha)$ in place of $\kappa_n(\alpha)$ as the threshold in~\eqref{eq:Cn}, cf.~\S\ref{proofs}. Therefore, in practice, we choose $\kappa_n^*(\alpha)$ as the threshold whenever there are tied observations; otherwise we treat $F$ as continuous, and use $\kappa_n(\alpha)$, determined by setting $F$ as uniform, as the threshold. This is also the default choice in our R-pakcage essHist. As mentioned, $\kappa_n(\alpha)$ and $\kappa_n^*(\alpha)$ are estimated via Monte-Carlo simulations, which are needed only once for a given sample size (which are automatically recorded for later usage in our R-package  ess\-Hist), while the computation time could be much longer than the pruned dynamic program in Algorithm~\ref{agDP} (see \S\ref{ss:pdp}) for large sample sizes, say $n \ge 10,000$. By simulation, we observe that the distributions of $T_n$ and $T_n^*$ seem to converge to a limit, and that such a convergence is fairly good when $n \ge 10,000$, see Figs.~\ref{fig:mssim} and~\ref{fig:qnt}. This is theoretically underpinned by the tightness of $T_n$ and $T_n^*$, see~\citet[Proposition 1]{RivWal13} and Lemma~\ref{le:ks}. For the sake of computational speed, we use the values of $\kappa_n(\alpha)$ and $\kappa_n^*(\alpha)$ with $n = 10,000$ as default values for sample sizes $n \ge 10,000$ in our R-package essHist. Through extensive simulation studies, we find that the performance of the essential histogram is hardly affected by such a default rule,  while the computation gets a significant speedup. }

\subsection{{Dynamic programming}}\label{ss:pdp}
Now we consider the computation of the (relaxed) essential histogram defined in~\eqref{eq:adj_ess_hist}. As argued in \S\ref{subAlg}, it can be computed by an accelerated dynamic programming algorithm. The idea of designing such an algorithm follows from \cite{FrMuSi14}, but the constraint that the solution should be a histogram introduces an additional difficulty. 

{
For notation simplicity, we consider the case that there is no ties in the data, otherwise it is sufficient to treat only distinct observations as the candidate locations of breakpoints, and the computation complexity will essentially depend on the number of distinct samples.} Since the data can be sorted in $O(n \log n)$ computation time, we simply assume $X_i = X_{(i)}$, for $i =1,\ldots, n$. For every interval $I$ in $\JJ$, defined in  \eqref{eqJJ}, the corresponding local constraint in $\tilde C_n(\alpha)$ in \eqref{eq:tldCn} leads to simply an interval, namely, 
\begin{equation}\label{eq:locI}
[l_I, u_I] =  \Bl\{\mu \mid \sqrt{2 \lLR \Bl(\mu\abs{I},F_n(I)\Br)} -
{\pen(F_n(I))} \le \kappa_n(\alpha)\Br\}. 
\end{equation}
Here $l_I$ and $u_I$ are roots of a smooth nonlinear equation, which can be computed in $O(1)$ time by standard algorithms, such as quasi-Newton methods~\citep[see e.g.][]{NoWr06}. For ease of exposition, we introduce the notation, for every $I \in \JJ$ and $\mu \in \R$ 
\[
T_{n,I}(\mu) = \sup_{J \subset I, \,J \in \JJ} \Bl(\sqrt{2 \lLR \Bl(\mu\abs{J},F_n(J)\Br)} -
{\pen(F_n(I))}\Br).
\]
We consider the multiscale constraint and entropy minimization (which is needed in case of multiple solutions) simultaneously, and thus define, for $j < i$
\[
c_{(j, i]} = 
\begin{cases}
-(i-j)\log\frac{i-j}{n(X_{(i)}-X_{(j)})} & \text{if } T_{n, (X_{(j)}, X_{(i)}]} (\frac{i-j}{n(X_{(i)}-X_{(j)})}) \le \kappa_n(\alpha)\\
\infty & \text{otherwise}
\end{cases}
\]
Let $K[i]$ be the number of blocks of the essential histogram $\hat h_i$ being applied to restricted data $X_{(1)},\ldots, X_{(i)}$, and $L[i]$ be the index of the leftmost sample in the last block of $\hat h_i$. Then it holds  
\begin{eqnarray}
K[0] &=& -1, \nonumber \\
K[i] &= &\min\left\{K[j] + 1 \mid c_{(j,i]} < \infty, j = 0, \ldots, i-1\right\}; \label{eq:bellK} \\
L[0] &= &1, \nonumber \\
L[i] &= &\mathop{\arg\min}_{j =0,\ldots, i-1}\left\{c_{(j,i]}  \mid K[j] + 1 = K[i]\right\}. \label{eq:bellL} 
\end{eqnarray}
The above relation is often referred to as \emph{Bellman equation}~\citep{Bel57}, which is the crucial part of a dynamic programing algorithm. Note that the histogram is completely determined by the segmentation (i.e.~breakpoints).  Thus, the essential histogram can be easily determined from vector $L = \{L[1], \ldots,L[n]\}$ in $O(K[n])$ (less than $O(n)$) computation time. Now we are ready to present the whole algorithm for computing the essential histogram in Algorithm~\ref{agDP}.

\begin{algorithm}[h]
  \KwData{$X_{(1)},\ldots, X_{(n)}$}
  \KwResult{the essential histogram}
  $k = 1$, $u = \infty$, $l = -\infty$, $\AA_0 = \{0\}$, $\BB = \{1,\ldots,n\}$, and $K, L$ vectors of length $n$\;
  \While{$i \in \BB$ and $\BB \neq \emptyset$}{
    $u = \min\bigl\{u, \min\{u_I \mid I \in \JJ, \max I = i\}\bigr\}$ with $u_I$ defined in \eqref{eq:locI}\;
    $l = \max\bigl\{v, \max\{l_I\mid I \in \JJ, \max I = i\}\bigr\}$  with $l_I$ defined in \eqref{eq:locI}\;
    Compute $K[i]$ and $L[i]$ by~\eqref{eq:bellK} and \eqref{eq:bellL} over $\AA_{k-1}$\;
    \If{$l > u$}{
    $\AA_k = \{i: K[i] = k\}$, and $\BB =  \BB \setminus  \AA_k$\;
    $u = \infty$, $l = -\infty$, and $k = k+1$\;
    }
  }
  Compute the histogram from $L$\;
  \caption{Pruned dynamic algorithm for the essential histogram}\label{agDP}
\end{algorithm}

Unlike standard dynamic programs, Algorithm~\ref{agDP} incorporates two important pruning steps: One is in line~5, where the search space is $\AA_{k-1}$ instead of $\{0, \ldots, i -1\}$; The other is in line~6, which prohibits further search towards right if no constant signal is admitted to the multiscale constraint, that is, no further right beyond $X_{(r_k)}$ with 
\begin{equation}\label{eq:rsLimit}
r_k = \Bl\{i \mid T_{n, (X_{(j)},X_{(i)}]}(\mu) \le \kappa(\alpha)\quad \text{for some } i \in \AA_{k-1} \text{ and } \mu \in \R\Br\}.
\end{equation}
Further pruning is possible, for instance, by introducing a similar stopping rule as $r_k$ in~\eqref{eq:rsLimit} on the reverse order of the data $\{X_{(n)}, \ldots, X_{(1)}\}$~\cite[see][]{PeSiMu15}.  We refer to our R-package essHist  
 for further technical details. 

The computation complexity of Algorithm~\ref{agDP} is bounded from above by
\[
O\Bl(\log n \sum_{k =1}^{K[n]}\#\AA_{k-1}(r_k - \min\AA_{k-1})\Br) \le O\Bl(n \log n \max_{k = 1, \ldots, K[n]}(r_k -  \min\AA_{k-1})\Br)
\]
In particular, if $K[n] = 1$, it implies that the computation complexity is $O(n \log n)$. In most cases, $\max_{k = 1, \ldots, K[n]}(r_k -  \min\AA_{k-1})$ stays bounded, which thus leads to a nearly linear computation complexity $O(n \log n)$. However, in very rare cases, $\max_{k = 1, \ldots, K[n]}(r_k -  \min\AA_{k-1})$ can be of order $n$, which gives the worst case complexity $O(n^2 \log n)$. Clearly, the memory complexity of Algorithm~\ref{agDP} is always linear, i.e.,~$O(n)$. 

Furthermore, we point out that Algorithm~\ref{agDP} applies to the multiscale constraint with arbitrary system of intervals besides $\JJ$.

\section{Additional simulation results}\label{sp:sim}
This section collects all quantitative comparison results for the essential histogram, the \DK estimator, and the classical histograms, being a companion of \S\ref{ss:compare}. 

While density estimation is not the primary purpose of the essential histogram, we now consider estimation errors measured by $L^2$-loss, and Kolmogorov loss. The former gives the mean integrated squared error (MISE), namely, 
\[
\Ex\left(\int \abs{f(x) - \hat{f}(x)}^2 dx\right)\qquad \text{ for true density } f \text{ and its estimator } \hat{f}; 
\]
The latter leads to 
\[
\Ex \sup_x \abs{\int_{-\infty}^x \bigl(f(t) - \hat f(t)\bigr)dt}\qquad \text{ for true density } f \text{ and its estimator } \hat{f}.
\]
In practice, the expectations are approximated by averages over independent repetitions. 

From data exploration and shape recovery perspective, we introduce evaluation measures via numbers of modes/troughs, and skewness. For a histogram density $h = \sum_{k = 0}^K c_k 1_{I_k}$ with $c_k \neq c_{k+1}$, we define the number of modes (i.e.~maxima) as $\#\bigl\{k:c_{k} > \max\{c_{k-1}, c_{k+1}\}\bigr\}$ and the number of troughs (i.e.~minima) as $\#\{k:c_{k} < \min\{c_{k-1}, c_{k+1}\}\}$. The number of extrema is then defined as the total number of modes and troughs. Note that the number of modes and troughs also capture the number of increases and decreases. From a slightly different viewpoint, the difference between the skewness of the estimator and that of the truth reflects how well the shape of the truth is recovered by the estimator.  

Table~\ref{tab:uniform} is for the uniform density example; 
Table~\ref{tab:monotone} for the monotone density example; 
Tables~\ref{tab:mix_unif_mode} and~\ref{tab:mix_unif_bin} for the histogram density example; 
Tables~\ref{tab:claw},~\ref{tab:claw_nbin} and~\ref{tab:claw_skewness} for the claw density example;
Tables~\ref{tab:harp_mode},~\ref{tab:harp_mise},~\ref{tab:harp_kol} and~\ref{tab:harp_skew} for the harp density example; 
and Tables~\ref{tab:cauchy} and~\ref{tab:cauchy_nbin} for the heavy tail example, in \S\ref{ss:compare}.
Within each table, the best value along each column (i.e.~among different methods) is marked in bold. 

\begin{table}
\caption{Number of modes (i.e.~false positives) on the uniform density (in~Fig.~\ref{fig:uniform}) averaged over 500 repetitions}%
\begin{tabular}{llcccccc}
\toprule
\multicolumn{2}{c}{\multirow{2}{*}{Methods}} & \multicolumn{5}{c}{Number of observations} \\ 
                                                                        & & $n = 100$ & $n = 300$ & $n = 500$ & $n = 700$ & $n = 900$  \\ 
                                                                        \midrule
\multirow{6}{*}{Essential histogram}   & $\alpha = 0.1$ &  \textbf{0.000}    &   \textbf{0.002}  & \textbf{0.000} & \textbf{0.000} & \textbf{0.000} \\  
                                  & $\alpha = 0.2$ & 0.004   &    0.006    &   0.004   &    0.006    &   0.008     \\
				        & $\alpha = 0.3$ &  0.006   &    0.010    &   0.016    &   0.012    &   0.016  \\
					 & $\alpha = 0.5$  &  0.030   &    0.046    &   0.054   &    0.048    &   0.072    \\
					 & $\alpha = 0.7$  & 0.108    &   0.148    &   0.162    &   0.188    &   0.178     \\
					 & $\alpha = 0.9$  & 0.424    &   0.442   &    0.518    &   0.548    &   0.560    \\
\multicolumn{2}{l}{\DK}                                 & 0.332    &   0.052    &   0.244    &   0.062    &   0.038    \\
\multirow{2}{*}{Equal bin width} & Sturges' rule                                            &   5.126    &   5.220    &   5.194    &   5.148    &   5.262    \\
					& Scott's rule &  1.938    &   3.344    &   3.914    &   4.798  & 5.300\\
\multirow{1}{*}{Equal block area}   & Scott's rule  &  2.060   &    3.378    &   4.062    &   4.732    &   5.326 \\
\bottomrule
\end{tabular}
 \label{tab:uniform}
\end{table}

\begin{table}
\caption{Number of modes (i.e.~false positives) on the exponential density (in~Fig.~\ref{fig:exponential}) averaged over 500 repetitions}
\begin{tabular}{llcccccc}
\toprule
\multicolumn{2}{c}{\multirow{2}{*}{Methods}} & \multicolumn{5}{c}{Number of observations} \\ 
                                                                        & & $n = 100$ & $n = 300$ & $n = 500$ & $n = 700$ & $n = 900$  \\ 
                                                                        \midrule
\multirow{6}{*}{Essential histogram}   & $\alpha = 0.1$ &  \textbf{0.000}& \textbf{0.000} & \textbf{0.000} & \textbf{0.000} & \textbf{0.000}   \\  
                                  & $\alpha = 0.2$ &  0.006 &\textbf{0.000}& \textbf{0.000}& \textbf{0.000} &0.004    \\
				        & $\alpha = 0.3$ &  0.006 &0.002 &\textbf{0.000}& \textbf{0.000}& 0.006  \\
					 & $\alpha = 0.5$  & 0.014 &0.012 &0.006 &0.008 &0.012    \\
					 & $\alpha = 0.7$  & 0.034 &0.034 &0.028 &0.018 &0.028   \\
					 & $\alpha = 0.9$  & 0.100 &0.066 &0.074& 0.064 &0.076 \\
\multicolumn{2}{l}{\DK}                                 & 0.048 &\textbf{0.000} &0.022 &0.004 &0.010  \\
\multirow{2}{*}{Equal bin width} & Sturges' rule                &   0.650 &0.610 &0.534 &0.610 &0.566   \\
			& Scott's rule  &   0.328& 0.942 &1.406& 1.674 &2.074  \\
\multirow{1}{*}{Equal block area}   & Scott's rule  &  1.158  &2.472 &3.348 &4.044 &4.580 \\
\bottomrule
\end{tabular}
\label{tab:monotone}
\end{table}

\begin{table}
\caption{Frequency of detecting the correct number of extrema on the histogram density (in Fig.~\ref{fig:mix_unif}) in 500 repetitions}
\begin{tabular}{llcccccc}
\toprule
\multicolumn{2}{c}{\multirow{2}{*}{Methods}} & \multicolumn{5}{c}{Number of observations} \\ 
                                                                        & 
                                                                        & $n = 600$ & $n = 700$ & $n = 800$ & $n = 900$ & $n = 1000$ \\ 
                                                                        \midrule
\multirow{6}{*}{Essential histogram}   & $\alpha = 0.1$&\textbf{95.6\%}& \textbf{98.0\%}& \textbf{99.2\%}& \textbf{98.8\%}&  \textbf{98.4\%} \\  
                                 & $\alpha = 0.2$ & 95.2\%& 97.6\%&97.6\%&96.4\%&96.6\% \\
				& $\alpha = 0.3$ &94.4\%&95.2\%& 95.0\%&94.8\%& 93.2\% \\
			        & $\alpha = 0.5$ &89.4\%& 90.4\%&88.6 \%&89.0\%&  88.6\% \\
				& $\alpha = 0.7$ & 83.0\%& 80.8\%& 79.8\%&81.0\%&78.6\% \\
				& $\alpha = 0.9$ & 64.6\%& 64.8\%&68.6\%&67.2\%& 60.8\% \\
\multicolumn{2}{l}{\DK}  &  50.2\%&51.8\%&50.4\%&48.2\%&46.2\% \\
\multirow{2}{*}{Equal bin width} & Sturges' rule  &  33.6\%&36.6\%&31.6\%&29.0\%&34.0\% \\
&Scott's rule & 30.8\%&28.6\%&32.2\%& 29.8\%& 28.0\% \\
 \multirow{1}{*}{Equal block area}   & Scott's rule&  36.0\%&42.2\%&44.2\%&44.8\%& 43.6\%  \\
 \bottomrule
\end{tabular}
\label{tab:mix_unif_mode}
\end{table}

\begin{table}
\caption{Number of false bins (i.e., $\max\{\nbin (\hat{F}) - \nbin (F), 0\}$) on the histogram density (in Fig.~\ref{fig:mix_unif}) averaged over 500 repetitions}
\begin{tabular}{llcccccc}
\toprule
\multicolumn{2}{c}{\multirow{2}{*}{Methods}} & \multicolumn{5}{c}{Number of observations} \\ 
                                                                        & 
                                                                        & $n = 600$ & $n = 700$ & $n = 800$ & $n = 900$ & $n = 1000$ \\ 
                                                                        \midrule
\multirow{6}{*}{Essential histogram}   & $\alpha = 0.1$ &\textbf{0.02}& \textbf{0.03} &\textbf{0.02} &\textbf{0.02} & \textbf{0.04} \\  
                                  & $\alpha = 0.2$  &   0.06& 0.07 &0.06 &0.08 & 0.08\\
				        & $\alpha = 0.3$   &0.11& 0.10& 0.11 &0.13 & 0.14 \\
					 & $\alpha = 0.5$   &0.23 &0.23& 0.23 &0.29&  0.28 \\
					 & $\alpha = 0.7$   &0.40& 0.47 &0.45 &0.49  &0.53\\
					 & $\alpha = 0.9$  & 0.88 &0.90 &0.85& 0.99  &1.10\\
\multicolumn{2}{l}{\DK}                                 &7.23 &6.89& 7.58& 7.75 & 8.73\\
\multirow{2}{*}{Equal bin width} & Sturges' rule   &2.81 &2.79 &2.83 &2.80 & 2.82\\
& Scott's rule  &0.90 &0.83 &0.23 &0.23 & 0.62  \\
 \multirow{1}{*}{Equal block area}   & Scott's rule  & 1.00 &1.13 &2.00 &2.00  &2.15  \\
 \bottomrule
\end{tabular}
\label{tab:mix_unif_bin}
\end{table}

\begin{table}
\caption{Number of modes on the claw density (in Fig.~\ref{fig:claw}) averaged over 500 repetitions; The truth has 5 modes}
\begin{tabular}{llcccccc}
\toprule
\multicolumn{2}{c}{\multirow{2}{*}{Methods}} & \multicolumn{5}{c}{Number of observations} \\ 
                                                                        & & $n = 1000$ & $n = 1200$ & $n = 1500$ & $n = 2000$ & $n = 3000$ \\ 
                                                                        \midrule
\multirow{6}{*}{Essential histogram}   & $\alpha = 0.1$ &  1.58 & 1.85  &2.46 & 3.24  & 4.74 \\  
                                 & $\alpha = 0.2$ &  1.92&  2.35&  2.99  &3.74 &  4.9\\
				& $\alpha = 0.3$ &  2.19 & 2.68 & 3.34  &4.13 &  4.96 \\
				& $\alpha = 0.5$  &  2.65 & 3.19 & 3.91&  4.6 &  4.99 \\
				& $\alpha = 0.7$  & 3.16&  3.72 & 4.38 & 4.82 &  \textbf{5.00} \\
				& $\alpha = 0.9$  & 3.84 & 4.39  &4.75 & 4.96  &   \textbf{5.00} \\
\multicolumn{2}{l}{\DK}                                 & \textbf{4.99}&  \textbf{5.00} & \textbf{5.00}&  5.01 & 5.01  \\
\multirow{2}{*}{Equal bin width} & Sturges' rule  &   1.23  &1.24  &1.19 & 1.11  & 1.08 \\
& Scott's rule  &  4.39  &4.86  &5.58  &6.23  & 6.84 \\
\multirow{1}{*}{Equal block area}   & Scott's rule  &  5.07  &5.08 & 5.15 & 5.24   & 5.50 \\
\bottomrule
\end{tabular}
\label{tab:claw}
\end{table}

\begin{table}
\caption{Average number of bins, with standard deviation given in the parenthesis, on the claw density (in Fig.~\ref{fig:claw}) over 500 repetitions}
\begin{tabular}{llccccc}
\toprule
\multicolumn{2}{c}{\multirow{2}{*}{Methods}} & \multicolumn{5}{c}{Number of observations} \\ 
                                                                        & & {$n = 1000$} & {$n = 1200$} & {$n = 1500$} & {$n = 2000$} & {$n = 3000$} \\ 
                                                                        \midrule
\multirow{6}{*}{Essential histogram}   & $\alpha = 0.1$ &  \textbf{7.1} (1.1)   &   \textbf{8.4} (1.1)  &   \textbf{9.8} (1.1)  & \textbf{11.2} (0.9)    & \textbf{13.3} (0.8) \\  
 & $\alpha = 0.2$ &  8.1(1.2)   &  9.3 (1.1)  & {10.6 (1.0) } & {11.9 (\textbf{0.8}) }     & {13.9 (0.8)} \\
& $\alpha = 0.3$ &  8.7 (1.2)   &   9.9 (1.1)  & 11.2 (1.0)  & 12.4 (\textbf{0.8})  & 14.2 (\textbf{0.7}) \\
& $\alpha = 0.5$  &  9.8 (1.2)   & 10.8 (1.1)  & 12.0 (0.9)  & 13.1 (\textbf{0.8})  &  14.7 (\textbf{0.7}) \\
& $\alpha = 0.7$  & 10.7 (1.1)  & 11.7 (1.0)  & 12.7 (0.9)  & 13.7 (\textbf{0.8})  &  15.1 (\textbf{0.7}) \\
& $\alpha = 0.9$  & 11.9 (1.0)  & 12.8 (1.0)  & 13.6 (0.9)  & 14.6 (0.9)  &  15.8 (0.9) \\
\multicolumn{2}{l}{\DK}   & 48.6 (4.9)    &  52.5 (5.6)  & {57.8 (5.9)}  & {67.7 (5.9)}  &  {79.1 (6.6)}  \\
\multirow{2}{*}{Equal bin width} & Sturges' rule &  12.5 (\textbf{0.8}) &   12.8 (\textbf{0.8}) &  13.1 (\textbf{0.8}) & 13.5 (0.9) &    14.0 (0.8) \\
& Scott's rule  &19.3 (1.2) &  20.8 (1.2) & 22.9 (1.3) &   25.6 (1.3) &  30.1 (1.4) \\
 \multirow{1}{*}{Equal block area}   & Scott's rule  & 20.5  (1.7) & 22.1 (1.7) & 24.3 (1.8) &  27.4 (2.0) &  32.4 (2.3)  \\
 \bottomrule
\end{tabular}
\label{tab:claw_nbin}
\end{table}

\begin{table}
\caption{Skewness of the histograms on the claw density (in Fig.~\ref{fig:claw}) averaged over 500 repetitions; The skewness of the truth is 0}
\begin{tabular}{llcccccc}
\toprule
\multicolumn{2}{c}{\multirow{2}{*}{Methods}} & \multicolumn{5}{c}{Number of observations} \\ 
                                                                        & & $n = 1000$ & $n = 1200$ & $n = 1500$ & $n = 2000$ & $n = 3000$ \\ 
                                                                        \midrule
  \multirow{6}{*}{Essential histogram}   & $\alpha = 0.1$ &  \textbf{0.010}& 0.002 &0.010& 0.007& \textbf{0.009} \\  
                                 & $\alpha = 0.2$ & 0.012& 0.007 &0.009& 0.006 &0.010\\
				& $\alpha = 0.3$ & 0.016& 0.005& 0.007& \textbf{0.005}& 0.010 \\
				& $\alpha = 0.5$  &  0.018 &0.002& 0.006& 0.006 & \textbf{0.009} \\
				& $\alpha = 0.7$  & 0.017& 0.002 &0.008& 0.007& 0.010 \\
				& $\alpha = 0.9$  & 0.019& 0.005& 0.007& \textbf{0.005} & 0.010 \\
\multicolumn{2}{l}{\DK}  & 0.057 &0.043 &0.039 &0.038& 0.035  \\
\multirow{2}{*}{Equal bin width} & Sturges' rule  &   0.076 &0.056 &0.060& 0.053 & 0.050 \\
& Scott's rule  & 0.063 &0.049& 0.049& 0.039 &0.040 \\
 \multirow{1}{*}{Equal block area}   & Scott's rule  &  0.032 &\textbf{0.001} &\textbf{0.004}& 0.008 & 0.014 \\
 \bottomrule
\end{tabular}
 \label{tab:claw_skewness}
\end{table}

\begin{table}
\caption{Frequency of detecting the correct number of extrema on the harp density (in Fig.~\ref{fig:harp}) in 500 repetitions }
\begin{tabular}{llccccc}
\toprule
\multicolumn{2}{c}{\multirow{2}{*}{Methods}} & \multicolumn{5}{c}{Number of observations} \\ 
& & {$n = 600$} & {$n = 800$} & {$n = 1000$} & {$n = 1200$} & {$n = 1500$}  \\ 
\midrule
\multirow{6}{*}{Essential histogram}   & $\alpha = 0.1$ &14.6\%& 64.4\%&  93.4\%&  98.8\%& \textbf{100\%}   \\  
&  $\alpha = 0.2$ & 35.6\%& 80.0\%&  96.6\%&  99.4\%&   {\textbf{100\%}}   \\
& $\alpha = 0.3$ & 49.6\%& 88.6\%&  97.4\%&  {\textbf{99.8\%}}&  {\textbf{100\%}}   \\
& $\alpha = 0.5$  & 69.6\%& 95.2\%&  97.8\%&  \textbf{99.8\%}& \textbf{100\%}       \\
& $\alpha = 0.7$  & 82.6\%& 97.4\%&  99.4\%&  \textbf{99.8\%}& \textbf{100\%}  \\
& $\alpha = 0.9$  & 93.0\%& 98.4\%&  99.6\%&  \textbf{99.8\%}& \textbf{100\%}  \\
\multicolumn{2}{l}{\DK} & \textbf{99.8\%}& \textbf{99.6\%}&  \textbf{99.8\%}&  \textbf{99.8\%}& \textbf{100\%}  \\
                   \multirow{2}{*}{Equal bin width} & Sturges' rule   &   0.6\%&  0.0\%&   0.0\%&   2.0\%&   1.4\%       \\
& Scott's rule  &  0.0\%&  0.0\%&   0.0\%&   0.0\%&   0.0\% \\
 \multirow{1}{*}{Equal block area}   & Scott's rule  &  1.0\%& 24.2\%&  81.0\%&  98.8\%& \textbf{100\%} \\
 \bottomrule
\end{tabular}
\label{tab:harp_mode}
\end{table}

\begin{table}
\caption{MISE $\times 10 ^5$ on the harp density (in Fig.~\ref{fig:harp}) averaged over 500 repetitions }
\begin{tabular}{llSSSSS}
\toprule
\multicolumn{2}{c}{\multirow{2}{*}{Methods}} & \multicolumn{5}{c}{Number of observations} \\ 
&   & {$n = 600$} & {$n = 800$} & {$n = 1000$} & {$n = 1200$} & {$n = 1500$}  \\ 
\midrule
\multirow{6}{*}{Essential histogram}   & $\alpha = 0.1$  &3.9 & 2.8  &2.4  &2.2 & 1.9 \\  
&  $\alpha = 0.2$ &3.4 & 2.6&  2.3 & 2.0  &1.8  \\
& $\alpha = 0.3$ & 3.2  &2.5 & 2.2  &1.9 & 1.7 \\
& $\alpha = 0.5$  &2.9  &2.4 & 2.1 & 1.8 & 1.6        \\
& $\alpha = 0.7$  &2.7  &2.3 & 2.0 & 1.7 & 1.4 \\
& $\alpha = 0.9$  & 2.6  &2.2  &1.9&  1.6&  1.3 \\
\multicolumn{2}{l}{\DK} & {\textbf{1.5}}& {\textbf{0.8}}& {\textbf{0.6}}& {\textbf{0.4}}& {\textbf{0.3}}  \\
                   \multirow{2}{*}{Equal bin width} & Sturges' rule    &61.9 &63.1 &63.7 &55.8 &57.3      \\
& Scott's rule  & 35.0 & 31.7 & 30.1 &31.0 & 36.0  \\
 \multirow{1}{*}{Equal block area}   & Scott's rule  &  20.4& 15.6&  9.8 & 8.5 & 7.7 \\
 \bottomrule
\end{tabular}
\label{tab:harp_mise}
\end{table}

\begin{table}
\caption{Error measured by~Kolmogorov metric on the harp density (in Fig.~\ref{fig:harp}) averaged over 500 repetitions }
\begin{tabular}{llcccccc}
\toprule
\multicolumn{2}{c}{\multirow{2}{*}{Methods}} & \multicolumn{5}{c}{Number of observations} \\ 
& & $n = 600$ & $n = 800$ & $n = 1000$ & $n = 1200$ & $n = 1500$  \\ 
\midrule
\multirow{6}{*}{Essential histogram}   & $\alpha = 0.1$  &0.05 &\textbf{0.04}&  \textbf{0.03} & \textbf{0.03} & \textbf{0.03}  \\  
&  $\alpha = 0.2$ & 0.05 &\textbf{0.04}  &\textbf{0.03}  &\textbf{0.03}  &\textbf{0.03}  \\
& $\alpha = 0.3$  &0.05 &\textbf{0.04}  &\textbf{0.03}  &\textbf{0.03}  &\textbf{0.03}    \\
& $\alpha = 0.5$  & \textbf{0.04}& \textbf{0.04}  &\textbf{0.03} & \textbf{0.03} & \textbf{0.03}      \\
& $\alpha = 0.7$  & \textbf{0.04}& \textbf{0.04}  &\textbf{0.03} & \textbf{0.03} & \textbf{0.03}  \\
& $\alpha = 0.9$  & \textbf{0.04} &\textbf{0.04}  &\textbf{0.03} & \textbf{0.03} & \textbf{0.03}  \\
\multicolumn{2}{l}{\DK} &  0.05& 0.05 & 0.05 & 0.04 & 0.04  \\
                   \multirow{2}{*}{Equal bin width} & Sturges' rule    &0.18 &0.18 & 0.18&  0.15  &0.16    \\
& Scott's rule  & 0.09& 0.09 & 0.08 & 0.08&  0.08  \\
 \multirow{1}{*}{Equal block area}   & Scott's rule   &0.06 &0.06 & 0.05  &0.05  &0.04 \\
 \bottomrule
\end{tabular}
\label{tab:harp_kol}
\end{table}

\begin{table}
\caption{Skewness of the histograms on the claw density (in Fig.~\ref{fig:harp}) averaged over 500 repetitions; The skewness of the truth is 0.9 }
\begin{tabular}{llcccccc}
\toprule
\multicolumn{2}{c}{\multirow{2}{*}{Methods}} & \multicolumn{5}{c}{Number of observations} \\ 
& &  $n = 600$ & $n = 800$ & $n = 1000$ & $n = 1200$ & $n = 1500$  \\ 
\midrule
\multirow{6}{*}{Essential histogram}   & $\alpha = 0.1$& 0.94 &0.92 & 0.91  &0.91  &0.91 \\  
&  $\alpha = 0.2$ &0.93 &\textbf{0.91} & 0.91 & 0.91 & 0.91 \\
& $\alpha = 0.3$ & 0.92 &\textbf{0.91}  &0.91  &0.91  &0.91  \\
& $\alpha = 0.5$  &0.91& \textbf{0.91}&  0.91&  0.91 & \textbf{0.90}   \\
& $\alpha = 0.7$  & 0.91& \textbf{0.91}  &0.91  &\textbf{0.90}  &\textbf{0.90}  \\
& $\alpha = 0.9$   & \textbf{0.90} & \textbf{0.91} & \textbf{0.90}  &\textbf{0.90} & \textbf{0.90} \\
\multicolumn{2}{l}{\DK} & 0.60 &0.65  &0.69  &0.72 & 0.75 \\
                   \multirow{2}{*}{Equal bin width} & Sturges' rule & 1.08 &1.10 & 1.12 & 1.09  &1.10   \\
& Scott's rule & 0.83 &0.85 & 0.86 & 0.87 & \textbf{0.90} \\
 \multirow{1}{*}{Equal block area} & Scott's rule & 0.98& 0.94 & 0.93  &0.96  &0.98 \\
 \bottomrule
\end{tabular}
\label{tab:harp_skew}
\end{table}

\begin{table}
\caption{Frequency of detecting the correct number of extrema on the Cauchy density (in Fig.~\ref{fig:cauchy}) in 500 repetitions }
\begin{tabular}{llcccccc}
\toprule
\multicolumn{2}{c}{\multirow{2}{*}{Methods}} & \multicolumn{5}{c}{Number of observations} \\ 
                                                                        & & $n = 100$  & $n = 200$ & $n = 300$ & $n = 400$ & $n = 500$  \\ 
                                                                        \midrule
\multirow{6}{*}{Essential histogram}   & $\alpha = 0.1$ & \textbf{100\%}     &  \textbf{100\%}      & \textbf{100\%}     & \textbf{100\%}      & \textbf{100\%}   \\  
                                  & $\alpha = 0.2$ & \textbf{100\%}  &  \textbf{100\%}  & \textbf{100\%}  & \textbf{100\%}   & \textbf{100\%}     \\
				        & $\alpha = 0.3$ & \textbf{100\%} &  \textbf{100\%}  & \textbf{100\%}    & \textbf{100\%}  & \textbf{100\%}        \\
					 & $\alpha = 0.5$  & \textbf{100\%}    &  \textbf{100\%}     & \textbf{100\%}      & 99.8\%  & \textbf{100\%}      \\
					 & $\alpha = 0.7$  & \textbf{100\%}   & \textbf{100\%}      & 99.8\%     & 99.8\%  & \textbf{100\%}    \\
					 & $\alpha = 0.9$  & 99.6\% & 99.6\% & 99.0\%& 99.4\%& 99.4\%      \\
\multicolumn{2}{l}{\DK}                                &  \textbf{100\%} & \textbf{100\%}  & \textbf{100\%}  & \textbf{100\%} & \textbf{100\%}  \\
                   \multirow{2}{*}{Equal bin width} & Sturges' rule   &  76.6\% &   72.8\%  &  76.0\%  & 74.0\% & 74.0\%       \\
					& Scott's rule  & 35.6\% &   25.6\%  &  21.0\%  & 15.2\%  &  15.4\%   \\
\multirow{1}{*}{Equal block area}   & Scott's rule  & 5.4\%  &  0.0\%  &  0.0\% &  0.0\% &  0.0\%  \\
\bottomrule
\end{tabular}
\label{tab:cauchy}
\end{table}

\begin{table}
\caption{Average number of bins, with standard deviation given in the parenthesis, on the Cauchy density (in Fig.~\ref{fig:cauchy}) over 500 repetitions}
\begin{tabular}{llcccccc}
\toprule
\multicolumn{2}{c}{\multirow{2}{*}{Methods}} & \multicolumn{5}{c}{Number of observations} \\ 
                                                                        & & $n = 100$ & $n = 200$ & $n = 300$ & $n = 400$ & $n = 500$  \\ 
                                                                        \midrule
\multirow{6}{*}{Essential histogram}   & $\alpha = 0.1$ &  \textbf{4.3} (0.58) & 5.3 (\textbf{0.53}) & 6.1 (0.59) &  6.9 (\textbf{0.47}) &  7.3 (\textbf{0.51})  \\  
                                  & $\alpha = 0.2$ &  4.5 (0.56) & 5.6 (0.61) & 6.4  (0.57) &  7.2  (0.52) &  7.7 (0.59)  \\
				        & $\alpha = 0.3$ &  4.7 (\textbf{0.53}) &   5.8 (0.61) &   6.6 (0.56) &   7.4 (0.58) &   7.9 (0.63)      \\
					 & $\alpha = 0.5$  &  4.9 (0.54) &   6.1 (0.63) &   6.9 (\textbf{0.51}) &   7.8 (0.68) &   8.4 (0.66)   \\
					 & $\alpha = 0.7$  &  5.2 (0.54) &   6.5 (0.66) &   7.2 (0.58) &   8.3 (0.71) &   8.8 (0.68)  \\
					 & $\alpha = 0.9$  &  5.6 (0.66) &   7.1 (0.74) &   7.7 (0.70) &   9.0 (0.79) &   9.4 (0.80)    \\
\multicolumn{2}{l}{\DK} & 14.7 (2.59) &  22.8 (3.27) &  28.2 (3.60) &  33.3 (3.85) &  37.5  (4.14)   \\
\multirow{2}{*}{Equal bin width} & Sturges' rule &  4.4 (1.26) &   \textbf{4.6} (1.47) &   \textbf{4.7} (1.50) &   \textbf{4.7} (1.48) &   \textbf{4.7} (1.48) \\
& Scott's rule  & 6.0 (2.13) &   7.7 (2.91) &   8.9 (3.60) &  10.0 (4.06) &  10.8 (4.30) \\
 \multirow{1}{*}{Equal block area}   & Scott's rule  & 14.2 (1.79) &  24.7 (3.44) &  34.6 (4.75) &  44.1 (6.22) &  53.0 (7.35) \\
 \bottomrule 
\end{tabular}
\label{tab:cauchy_nbin}
\end{table}

\end{appendices}


\begin{thebibliography}{}

\bibitem[Azzalini and Bowman, 1990]{AzBo90}
Azzalini, A. and Bowman, A.~W. (1990).
\newblock A look at some data on the {O}ld {F}aithful geyser.
\newblock {\em J. R. Stat. Soc. Ser. C. Appl. Statist.}, 39(3):357--365.

\bibitem[Bellman, 2010]{Bel57}
Bellman, R. (2010).
\newblock {\em Dynamic programming}.
\newblock Princeton University Press, Princeton, NJ.

\bibitem[Birg{\'e} and Rozenholc, 2006]{BirRoz06}
Birg{\'e}, L. and Rozenholc, Y. (2006).
\newblock How many bins should be put in a regular histogram.
\newblock {\em ESAIM Probab. Stat.}, 10:24--45 (electronic).

\bibitem[Davies and Kovac, 2004]{DavKov04}
Davies, P.~L. and Kovac, A. (2004).
\newblock Densities, spectral densities and modality.
\newblock {\em Ann. Statist.}, 32(3):1093--1136.

\bibitem[Denby and Mallows, 2009]{DenMal09}
Denby, L. and Mallows, C. (2009).
\newblock Variations on the histogram.
\newblock {\em J. Comput. Graph. Statist.}, 18(1):21--31.

\bibitem[Dijkstra, 1959]{Dij59}
Dijkstra, E.~W. (1959).
\newblock A note on two problems in connexion with graphs.
\newblock {\em Numer. Math.}, 1:269--271.

\bibitem[D{\"u}mbgen, 2003]{Due03}
D{\"u}mbgen, L. (2003).
\newblock Optimal confidence bands for shape-restricted curves.
\newblock {\em Bernoulli}, 9(3):423--449.

\bibitem[D{\"u}mbgen and Spokoiny, 2001]{DueSpo01}
D{\"u}mbgen, L. and Spokoiny, V.~G. (2001).
\newblock Multiscale testing of qualitative hypotheses.
\newblock {\em Ann. Statist.}, 29(1):124--152.

\bibitem[D{\"u}mbgen and Walther, 2008]{DueWal08}
D{\"u}mbgen, L. and Walther, G. (2008).
\newblock Multiscale inference about a density.
\newblock {\em Ann. Statist.}, 36(4):1758--1785.

\bibitem[D{\"{u}}mbgen and Wellner, 2014]{DueWel14}
D{\"{u}}mbgen, L. and Wellner, J. (2014).
\newblock {Confidence bands for distribution functions: A new look at the law
  of the iterated logarithm}.
\newblock {\em arXiv: 1402.2918}.

\bibitem[Dvoretzky et~al., 1956]{DvoKieWolf56}
Dvoretzky, A., Kiefer, J., and Wolfowitz, J. (1956).
\newblock Asymptotic minimax character of the sample distribution function and
  of the classical multinomial estimator.
\newblock {\em Ann. Math. Statist.}, 27:642--669.

\bibitem[Freedman and Diaconis, 1981]{FreDia81}
Freedman, D. and Diaconis, P. (1981).
\newblock On the histogram as a density estimator: {$L\sb{2}$} theory.
\newblock {\em Z. Wahrsch. Verw. Gebiete}, 57(4):453--476.

\bibitem[Freedman et~al., 2007]{FrPiPu07}
Freedman, D., Pisani, R., and Purves, R. (2007).
\newblock {\em Statistics}.
\newblock W.W. Norton, Inc., New York, 4 th edition.

\bibitem[Frick et~al., 2014]{FrMuSi14}
Frick, K., Munk, A., and Sieling, H. (2014).
\newblock Multiscale change point inference.
\newblock {\em J. R. Stat. Soc. Ser. B. Stat. Methodol.}, 76(3):495--580.
\newblock With 32 discussions by 47 authors and a rejoinder by the authors.

\bibitem[Hocking et~al., 2017]{HRFB17}
Hocking, T.~D., Rigaill, G., Fearnhead, P., and Bourque, G. (2017).
\newblock {A log-linear time algorithm for constrained changepoint detection}.
\newblock {\em arXiv preprint arXiv: 1703.03352}.

\bibitem[{Killick} et~al., 2012]{KillFeaEck12}
{Killick}, R., {Fearnhead}, P., and {Eckley}, I.~A. (2012).
\newblock Optimal detection of changepoints with a linear computational cost.
\newblock {\em J. Amer. Statist. Assoc.}, 107(500):1590--1598.

\bibitem[Li et~al., 2016]{LiMuSi16}
Li, H., Munk, A., and Sieling, H. (2016).
\newblock {FDR}-control in multiscale change-point segmentation.
\newblock {\em Electron. J. Stat.}, 10(1):918--959.

\bibitem[Maidstone et~al., 2017]{MHRF17}
Maidstone, R., Hocking, T., Rigaill, G., and Fearnhead, P. (2017).
\newblock On optimal multiple changepoint algorithms for large data.
\newblock {\em Stat. Comput.}, 27(2):519--533.

\bibitem[Marron and Wand, 1992]{MarWan92}
Marron, J.~S. and Wand, M.~P. (1992).
\newblock Exact mean integrated squared error.
\newblock {\em Ann. Statist.}, 20(2):712--736.

\bibitem[Nocedal and Wright, 2006]{NoWr06}
Nocedal, J. and Wright, S.~J. (2006).
\newblock {\em Numerical optimization}.
\newblock Springer, New York, second edition.

\bibitem[Pearson, 1895]{Pea895}
Pearson, K. (1895).
\newblock Contributions to the mathematical theory of evolution. {II. Skew}
  variation in homogeneous material.
\newblock {\em Philos. Trans. Royal Soc. A}, 186:343--414.

\bibitem[Pein et~al., 2017]{PeSiMu15}
Pein, F., Sieling, H., and Munk, A. (2017).
\newblock Heterogeneous change point inference.
\newblock {\em J. R. Stat. Soc. Ser. B. Stat. Methodol.}, 79(4):1207--1227.

\bibitem[Rivera and Walther, 2013]{RivWal13}
Rivera, C. and Walther, G. (2013).
\newblock Optimal detection of a jump in the intensity of a {P}oisson process
  or in a density with likelihood ratio statistics.
\newblock {\em Scand. J. Stat.}, 40(4):752--769.

\bibitem[Scott, 1979]{Sco79}
Scott, D.~W. (1979).
\newblock On optimal and data-based histograms.
\newblock {\em Biometrika}, 66(3):605--610.

\bibitem[Scott, 1992]{Sco92}
Scott, D.~W. (1992).
\newblock {\em Multivariate density estimation}.
\newblock John Wiley \& Sons, Inc., New York.

\bibitem[Shorack and Wellner, 1986]{ShoWel86}
Shorack, G.~R. and Wellner, J.~A. (1986).
\newblock {\em Empirical processes with applications to statistics}.
\newblock John Wiley \& Sons, Inc., New York.

\bibitem[Sturges, 1926]{Stu26}
Sturges, H.~A. (1926).
\newblock The choice of a class interval.
\newblock {\em J. Amer. Statist. Assoc.}, 21(153):65--66.

\bibitem[Taylor, 1987]{Tay87}
Taylor, C.~C. (1987).
\newblock Akaike's information criterion and the histogram.
\newblock {\em Biometrika}, 74(3):636--639.

\bibitem[Tukey, 1961]{Tuk61}
Tukey, J.~W. (1961).
\newblock Curves as parameters, and touch estimation.
\newblock In {\em Proc. 4th {B}erkeley {S}ympos. {M}ath. {S}tatist. and
  {P}rob., {V}ol. {I}}, pages 681--694. Univ. California Press, Berkeley,
  Calif.

\bibitem[Unwin, 2015]{Unw15}
Unwin, A. (2015).
\newblock {\em Graphical data analysis with R}.
\newblock Chapman and Hall/CRC.

\bibitem[Walther, 2010]{Wal10}
Walther, G. (2010).
\newblock Optimal and fast detection of spatial clusters with scan statistics.
\newblock {\em Ann. Statist.}, 38(2):1010--1033.

\end{thebibliography}
\end{document}